\documentclass[10pt]{article}

\usepackage{epsfig,verbatim}
\usepackage{fullpage}
\usepackage{latexsym}
\usepackage{enumerate}
\usepackage[shortlabels]{enumitem}

\usepackage{tikz}
\usetikzlibrary{decorations.markings,arrows,snakes,calc}
\usetikzlibrary{positioning}

\tikzstyle{uStyle}=[shape = circle, minimum size = 4pt, inner sep = 1pt,
outer sep = 0pt, fill=white, semithick, draw]
\tikzstyle{uBstyle}=[shape = circle, minimum size = 9pt, inner sep = 1pt,
outer sep = 0pt, fill=white, semithick, draw]
\tikzstyle{lStyle}=[shape = circle, minimum size = 5pt, inner sep =
0.5pt, outer sep = 0pt, draw=none]
\tikzstyle{vlStyle}=[shape = circle, minimum size = 4pt, inner sep =
1.0pt, outer sep = 0pt, draw, fill=white, semithick, font=\footnotesize]

\usepackage{caption}
\usepackage{url}
\usepackage{color}
\usepackage{amssymb}
\usepackage{amsmath}
\usepackage{amsthm}
\usepackage{graphicx}
\usepackage{bbm}

\newtheorem{theorem}{Theorem}[section]
\newtheorem*{main}{Main Theorem}
\newtheorem{lemma}[theorem]{Lemma}
\newtheorem{lem}[theorem]{Lemma}

\newtheorem{corollary}[theorem]{Corollary}
\newtheorem{cor}[theorem]{Corollary}

\newtheorem{clm}[theorem]{Claim}
\newtheorem{remark}[theorem]{Remark}

\theoremstyle{definition}
\newtheorem{example}[theorem]{Example}
\newtheorem{definition}[theorem]{Definition}
\newtheorem{defn}[theorem]{Definition}
\newtheorem*{lem-reword}{Lemma~\ref{coloring-lem} (Rephrased)}

\def\HHH{\mathbb{H}}
\def\vph{\varphi}
\def\ch{\textrm{ch}}
\def\aftermath{\par\vspace{-\belowdisplayskip}\vspace{-\parskip}\vspace{-\baselineskip}}

\RequirePackage{marginnote,hyperref}
\newcommand{\aside}[1]{\marginnote{\scriptsize{#1}}[0cm]}
\newcommand{\aaside}[2]{\marginnote{\scriptsize{#1}}[#2]}
\newcommand\Emph[1]{\emph{#1}\aside{#1}}
\newcommand\EmphE[2]{\emph{#1}\aaside{#1}{#2}}

\def\vph{\varphi}
\def\dist{\textrm{dist}}
\def\tG{\tilde{G}}

\def\tT{\tilde{T}}
\def\tB{\tilde{B}}
\def\adj{\leftrightarrow}
\def\nonadj{\not\leftrightarrow}
\newenvironment{clmproof}[1]{\par\noindent\underline{Proof.}\space#1}{\leavevmode\unskip\penalty9999\hbox{}\nobreak\hfill\quad\hbox{$\diamondsuit$}\smallskip}

\def\arrowdrawthick{\draw[thick,decoration={markings, mark=at position
1.0 with {\arrow[>=stealth]{>}}}, postaction={decorate}]}

\title{Sparse Graphs are Near-bipartite}
\author{Daniel W. Cranston\thanks{%
Department of Mathematics and Applied
Mathematics, Virginia Commonwealth University, Richmond, VA, USA;
\texttt{dcranston@vcu.edu}; 
This research is partially supported by NSA Grant H98230-15-1-0013.} \and Matthew P. Yancey\thanks{%
Institute for Defense Analyses - Center for Computing Sciences, Bowie, MD, USA; \texttt{mpyancey1@gmail.com}
}
}

\begin{document}
\maketitle

\def\VG{V(G)}
\def\VGp{V(G')}
\def\VH{V(H)}
\def\VHp{V(H')}
\def\VHj{V(H_j)}
\def\EG{E(G)}
\def\EGp{E(G')}
\def\EH{E(H)}
\def\EHp{E(H')}
\def\EJ{E(J)}
\def\HH{\mathcal{H}}
\def\Re{\mathbf{R}}

\begin{abstract}
A multigraph $G$ is near-bipartite if $V(G)$ can be partitioned as $I,F$ such that
$I$ is an independent set and $F$ induces a forest.  We prove that a multigraph
$G$ is near-bipartite when $3|W|-2|E(G[W])|\ge -1$ for every $W\subseteq
V(G)$, and $G$ contains no $K_4$ and no Moser spindle.  We prove that a simple graph $G$ is
near-bipartite when $8|W|-5|E(G[W])|\ge -4$ for every $W\subseteq V(G)$, and
$G$ contains no subgraph from some finite family $\HH$.  We also construct
infinite families to show that both results are best possible in a very sharp
sense.
\end{abstract}

\newcommand{\bbs}[1]{\boldsymbol{#1}}
\def\dcup{\uplus}
\newcommand\Memph[1]{#1\aside{#1}}
\newcommand\MemphE[2]{#1\aaside{#1}{#2}}
\def\mytilde{\raise.17ex\hbox{$\scriptstyle\mathtt{\sim}$}}

%

\section{Introduction}
A multigraph\footnote{Without loss of generality, we assume that each edge has
multiplicity at most 2, as we explain at the start of Section~\ref{prelims-sec}.} 
$G$ is \Emph{near-bipartite} if its vertex set can be partitioned into
sets $I$ and $F$ such that $I$ is an independent set and $F$ induces a forest.
This condition is somewhat stronger than being 3-colorable, but the two problems
are closely related.
We call $I,F$ a \EmphE{near-bipartite coloring}{3mm} of $G$, or simply an
\Emph{nb-coloring}.  The goal of this paper is to prove sufficient conditions
for multigraphs and simple graphs to be near-bipartite, in terms of their
edge-densities; this is akin to the work done for $k$-coloring in~\cite{KY}.  
Since a near-bipartite coloring
of $G$ restricts to a near-bipartite coloring of each subgraph $J$ of $G$,
naturally our edge-density hypothesis for $G$ should also hold for each subgraph
$J$.  To facilitate a proof by induction, we also allow some vertices to be
precolored.  That is, we allow vertex subsets $I_p$ and $F_p$\aside{$I_p$,
$F_p$, $U_p$} such that our
near-bipartite coloring $I,F$ must have $I_p\subseteq I$ and $F_p\subseteq F$.
For convenience, let $U_p=V(G)\setminus (I_p\cup F_p)$.
We prove results for both the class of multigraphs and the class of
simple graphs.  For simple graphs, to facilitate our proof by induction, we
allow some edges to be specified as {edge-gadgets}.  In practice this means
that, for each edge-gadget $vw$, in every near-bipartite coloring
one of $v$ and $w$ appears in $I$ and the other appears in $F$; intuitively,
this is the same as if $vw$ was a multiedge. 
For a multigraph $G$ and $W\subseteq V(G)$, let \Emph{$e(W)$} denote the set of edges
with both endpoints in $W$.  For a simple graph, we let $e''(W)$ and $e'(W)$\aside{$e''(W),
e'(W)$} denote the subsets of $e(W)$ that are, respectively, edge-gadgets and
not edge-gadgets (but still edges).  Most of our other terminology and notation
is standard, but for reference we collect it in Section~\ref{defns-sec}.
Now we can define our measures of
edge-density, called \Emph{potential}, and denoted $\rho_{m,G}$ and
$\rho_{s,G}$.
(Here $m$ is for multigraph and $s$ is for simple graph.)

For a multigraph $G$ with precoloring $I_p,F_p$, for each $W\subseteq V(G)$ let

$$\rho_{m,G}(W)=3|W\cap U_p|+|W\cap F_p|-2|e(W)|\aaside{$\rho_{m,G}$}{-1mm}$$
and
$$\rho_{s,G}(W)=8|W\cap U_p|+3|W\cap
F_p|-5|e'(W)|-11|e''(W)|\aaside{$\rho_{s,G}$}{-1mm}.$$
Let $M_7$ denote the Moser spindle, shown in Figure~\ref{examples fig}, and let
$\HH$ be a finite family of simple graphs that we define in
Section~\ref{constructH-sec}, none of which is near-bipartite.
The following is the main result of this paper.

\begin{main}
(A) If $G$ is a multigraph with precoloring $I_p,F_p$ such that 
$\rho_{m,G}(W) \geq -1$ for all $W\subseteq V(G)$ and $G$ does not contain $K_4$ or
$M_7$ as a subgraph, then $G$ has a near-bipartite coloring $I,F$ that extends the
precoloring $I_p,F_p$.  Moreover, $I,F$ can be found in polynomial time.

(B) If $G$ is a simple graph with precoloring $I_p,F_p$ such that 
$\rho_{s,G}(W) \geq -4$ for all $W\subseteq V(G)$ and $G$ does not contain any
graph from $\mathcal{H}$ as a subgraph, then $G$ has a near-bipartite coloring
$I,F$ that extends the precoloring $I_p,F_p$.  Moreover, $I,F$ can be found in
polynomial time.
\end{main}

It is NP-complete to decide if a graph is near-bipartite\footnote{This is
unsurprising, since nb-coloring is closely connected with 3-coloring, a
well-known NP-complete problem.}, and this is attributed to Monien \cite{BLS}.
This problem remains NP-complete for several restriced families of graphs.
Brandst\"{a}dt, Brito, Klein, Nogueira, and Protti \cite{BBKNP} showed this for
perfect graphs, and Bonamy, Dabrowski, Feghali, Johnson, and
Paulusma~\cite{BDFJP} showed it for graphs with diameter 3.
Dross, Montassier, and Pinlou \cite{DMP} 
showed it for planar graphs, and Yang and Yuan \cite{YY} showed it for graphs
with maximum degree 4.  In contrast, Bonamy, Dabrowski, Feghali, Johnson, and
Paulusma \cite{BDFJP2} showed that for a simple graph $G$ with $\Delta(G) \leq
3$ and with no $K_4$, an nb-coloring (which exists by the
results below) can be found in time $O(|\VG|)$.

Borodin and Glebov \cite{BG} proved that if $G$ is planar with girth at
least $5$, then $G$ is near-bipartite.
Kawarabayashi and Thomassen \cite{KT} extended this result to allow a small
set of precolored vertices. 
Dross, Montassier, and Pinlou \cite{DMP} conjectured that every planar
graph with girth at least 4 is near-bipartite (which would strengthen the result
of~\cite{BG}).
Because they each considered different generalizations, multiple groups
\cite{BM,B,BKT,CL,YY} proved that if $G$ has no $K_4$ as a subgraph and
$\Delta(G) \leq 3$, then $G$ is near-bipartite.  Yang and Yuan \cite{YY} 
characterized near-bipartite graphs with diameter $2$.  Zaker
\cite[Theorem 4]{Z} proved that $G$ is near-bipartite if and only if 
its vertices can be ordered as $v_1, v_2, \ldots, v_n$ such that each
triple of edges with a common endpoint $v_i v_{j_1}, v_i v_{j_2}, v_i v_{j_3}$
does not satisfy $j_1 < i < j_2 \leq j_3$.

Finding an nb-coloring $I,F$ is also called ``finding a stable cycle cover'' \cite{BBKNP}.
When we want $I$ to have bounded size, the problem is called
finding an ``independent feedback vertex set'', and related work is described
in the references of~\cite{BDFJP}.

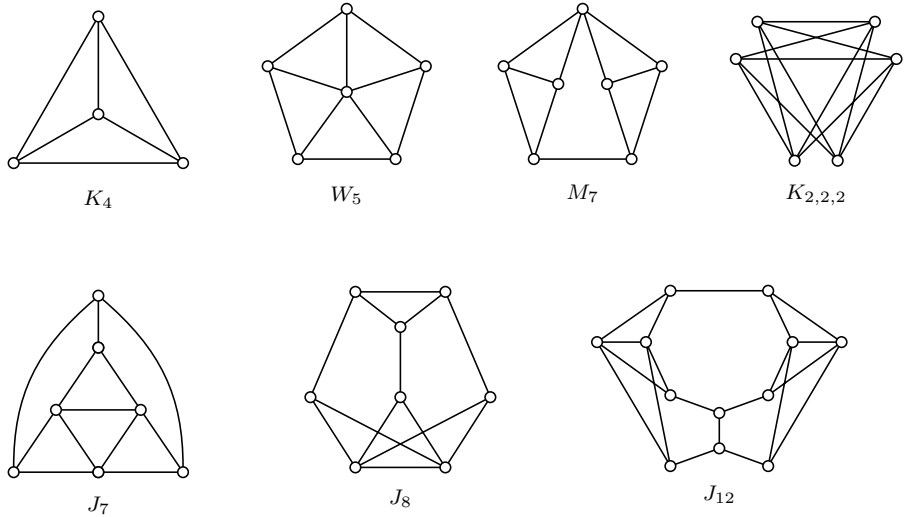
\begin{figure}[!h]
\centering
\begin{tikzpicture} [semithick,scale = .65]
\tikzset{every node/.style=uStyle}

\begin{scope} 
\newcommand\rad{2cm}
\foreach \i in {1,...,3}
  \draw (-30+120*\i:\rad) node (v\i) {};
\draw (0,0) node (v0) {};
\foreach \i/\j in {0/1, 0/2, 0/3, 1/2, 1/3, 2/3}
  \draw (v\i) -- (v\j);
\draw (0,-1.75) node[lStyle] {\footnotesize{$K_4$}};
\end{scope}

\begin{scope}[xshift=2in,yshift=.18in, scale=.85] 
\newcommand\rad{2cm}
\foreach \i in {1,...,5}
  \draw (18+72*\i:\rad) node (v\i) {};
\draw (0,0) node (v0) {};
\foreach \i/\j in {1/2, 2/3, 3/4, 4/5, 5/1}
  \draw (v\i) -- (v\j) (v\i) -- (v0);
\draw (0,-2.5) node[lStyle] {\footnotesize{$W_5$}};
\end{scope}

\begin{scope}[xshift=3.9in,yshift=.18in, scale=.85] 
\newcommand\rad{2cm}
\foreach \i in {1,...,5}
  \draw (18+72*\i:\rad) node (v\i) {};
\foreach \i/\j in {1/2, 2/3, 3/4, 4/5, 5/1}
  \draw (v\i) -- (v\j);
\draw node (w1) at (barycentric cs:v1=1,v3=1) {};
\draw node (w2) at (barycentric cs:v1=1,v4=1) {};
\draw (v1) -- (w1) -- (v3) (w1) -- (v2) (v1) -- (w2) -- (v4) (w2) -- (v5);
\draw (0,-2.45) node[lStyle] {\footnotesize{$M_7$}};
\end{scope}

\begin{scope}[scale=.85, xshift=6.8in,yshift=.32in] 
\newcommand\rad{2cm}
\foreach \i in {1,2,3}
  \draw (45+120*\i:\rad) node (v\i) {} (15+120*\i:\rad) node (w\i) {};
\foreach \i/\j in {1/2, 1/3, 2/3}
  \draw (v\i) -- (v\j) (v\i) -- (w\j) (w\i) -- (v\j) (w\i) -- (w\j);
\draw (0,-2.75) node[lStyle] {\footnotesize{$K_{2,2,2}$}};
\end{scope}

\begin{scope}[yscale=.85, yshift=-3.0in] 
\draw (90:2cm) node (w1) {} (210:2cm) node (w2) {} (330:2cm) node (w3) {}
(90:3.25cm) node (w4) {};
\draw node (v1) at (barycentric cs:w1=1,w2=1) {};
\draw node (v2) at (barycentric cs:w2=1,w3=1) {};
\draw node (v3) at (barycentric cs:w3=1,w1=1) {};
\draw (w2) -- (v2) -- (w3) (v2) -- (v3) -- (w1) -- (v1) -- (v2);
\draw (w2) -- (v1) -- (v3) -- (w3); 
\draw (w1) -- (w4) (w2) to [out= 90, in = 225](w4) (w4) to [in= 90, out = -45](w3);
\draw (w1) node {} (v2) node {};
\draw (0,-1.8) node[lStyle] {\footnotesize{$J_7$}};
\end{scope}

\begin{scope}[scale=1.15, xscale=.8, yscale=1.25, xshift=2.25in, yshift=-1.98in] 
\draw (0,0) node (w1) {} (2,0) node (w2) {} (-1,1) node (z1) {} (1,1) node (z2)
{} (3,1) node (z3) {} (0,2.5) node (v1) {} (1,2) node (v2) {} (2,2.5) node (v3) {};
\draw (v1) -- (v2) -- (v3) -- (v1) (z1) -- (w1) -- (z2) -- (w2) -- (z3) -- (w1)
-- (w2) -- (z1) -- (v1) (z2) -- (v2) (z3) -- (v3);
\draw (1,-0.435) node[lStyle] {\footnotesize{$J_8$}};
\end{scope}

\begin{scope}[yshift=-2.55in, xshift=5in, yscale=.725] 
\draw (-1,3.95) node (w1) {}  -- (1,3.95) node (w2) {};
\draw (-2.5,2.5) node (v1) {} -- (-1.5,2.5) node (v2) {} (1.5,2.5) node (v3) {} -- (2.5,2.5) node (v4) {};
\draw (-1,1) node (z1) {} -- (0,.5) node (z5) {} -- (1,1) node (z3) {}
 (-1,-1) node (z2) {} -- (0,-.5) node (z6) {} -- (1,-1) node (z4) {} (z5) -- (z6);
\draw (v1) -- (z1) -- (v2) -- (z2) -- (v1);
\draw (v3) -- (z3) -- (v4) -- (z4) -- (v3);
\draw (v1) -- (w1) -- (v2) (v3) -- (w2) -- (v4);
\draw (0,-1.8) node[lStyle] {\footnotesize{$J_{12}$}};
\end{scope}
\end{tikzpicture}
\caption{Examples of nb-critical graphs.  The graph $W_5$ is called the 5-wheel
and $M_7$ is called the Moser spindle.  All graphs shown are 4-critical, except
for $K_{2,2,2}$, which is 3-colorable.\label{examples fig}}
\end{figure}

The purpose of this paper is to give an algorithm for finding a near-bipartite
coloring when $G$ is sufficiently sparse.  This motivates the following
definitions.  A multigraph is \Emph{nb-critical} if it is not near-bipartite, but
every proper subgraph is near-bipartite.  Figure~\ref{examples fig} shows
examples of nb-critical graphs.  A multigraph $G$ is \Emph{$(a,b)$-sparse} if
every nonempty subset of vertices $W$ satisfies $|e(W)| \leq a |W| - b$.  A
graph is a forest if and only if it is $(1,1)$-sparse.  A vertex set $I$ is
independent if and only if $G[I]$ is $(0,0)$-sparse.
Our next two theorems rephrase parts (A) and (B) of the Main Theorem, 
state explicit bounds on the running times of algorithms to find the colorings,
and also mention constructions to show that both parts are very sharp.
We give these constructions in Section~\ref{constructH-sec}.
In Section~\ref{goldberg-alg-sec} we describe a key subroutine of our coloring
algorithm, but we defer presenting the algorithm in full until
Section~\ref{algorithm section}, when we have proved the Main Theorem.

\begin{theorem} \label{main thm multi}
\label{thmA}
There exists an infinite family of $(1.5,-1)$-sparse nb-critical multigraphs.
If $G$ is $(1.5,-0.5)$-sparse and has no $K_4$ and no $M_7$, then $G$ is
near-bipartite.  We can find an nb-coloring in time $O(|\VG|^6)$.
\end{theorem}

A graph $G$ is 2-degenerate if every nonempty subgraph $J$ satisfies $\delta(J) \leq 2$.
Every $2$-degenerate graph is near-bipartite, and we can find an nb-coloring
in time $O(|\VG|)$ using the obvious greedy algorithm.  Graphs that are
$(1.5,0.5)$-sparse are 2-degenerate, so Theorem~\ref{main thm multi} shows
that the greedy algorithm is sufficient in many of the cases where sparsity
implies a graph is near-bipartite.  Our more impressive result is that we can
do better when $G$ is simple.

\begin{theorem} \label{main thm simple}
\label{thmB}
There exists an infinite family of $(1.6,-1)$-sparse nb-critical simple graphs.
There exists a finite graph family $\HH$ such that if $G$ is a simple graph
that is $(1.6,-0.8)$-sparse and contains no subgraph isomorphic to a graph in
$\HH$, then $G$ is near-bipartite. We can find an nb-coloring in time $O(|\VG|^{22})$.
\end{theorem}


Theorems~\ref{thmA} and~\ref{thmB} are both best possible in a strong sense, due
to the infinite families of sharpness examples.  Since the proof of
Theorem~\ref{thmB} is long, we naturally wondered whether there is a shorter
proof of a slightly weaker result, e.g., that a simple graph $G$ is
near-bipartite whenever it is $(1.6,0)$-sparse and contains no subgraph in
$\HH$.  We answer this question more fully in Section~\ref{why-sec}; in short, we
believe the answer is No, no such shorter proof exists.

The most striking aspect of Theorem~\ref{main thm simple} is that we handle the family
$\HH$, which has hundreds of forbidden subgraphs.
Each graph in $\HH$ is both nb-critical (and so must be forbidden in such a
theorem) and also 4-critical\footnote{A graph is 4-critical if it is not
3-colorable, but each of its proper subgraphs is 3-colorable.}.  Although we
have not explicitly constructed all graphs in $\HH$, its recursive
definition in Section~\ref{define H sec} allows us to show that each of these
graphs has at most 22 vertices; so $\HH$ is finite.  Kostochka and the second
author~\cite{KY} showed that each $n$-vertex
4-critical graph $G$ has $|E(G)|\ge (5n-2)/3$.  As we show in Theorem~\ref{main
thm simple}, each $n$-vertex nb-critical graph with $n\ge 22$ has $|E(G)|\ge
(8n+4)/5$.  Intuitively, the familly $\HH$ is due to the fact that
$(5n-2)/3<(8n+4)/5$ when $n<22$.

Although $\HH$ is finite, it is is a natural subset of an infinite family
$\HH'$, and each graph of $\HH'$ is also both nb-critical and 4-critical. 
Thus, our description of $\HH'$ provides insight into the structure of sparse
nb-critical and 4-critical graphs.  In view of $\HH'$, it is natural to ask whether
nb-criticality implies 4-criticality, or vice versa.  But neither implication is true. 
In Section~\ref{sharpness-sec} we construct an infinite family of nb-critical
graphs $H_k$ that are 3-colorable (so not 4-critical).  There also exist
infinitely many 4-critical graphs such that even after removing
multiple (specified) edges from any one of these, it does not become
near bipartite\footnote{For example, we can start with the 4-critical graphs
$G_k$ constructed by Yao and Zhou in \cite{YZ}.  Even if we remove all
edges $x_1u_i$ and $y_1v_j$ with $4\le i,j\le 2k-5$ the graph fails to
become near-bipartite.  The proof of this is a straightforward case
analysis (considering the nb-colorings of $H_{2k}$ and of
$G[\{x_1,x_2,x_3,y_1\}]$), but the details are too numerous to include
here.}.

%

\subsection{Proof Outline}
To conclude this introduction, we outline the proof of the Main Theorem.
The proofs of parts (A) and (B) are similar, but (B) is harder because the
family $\HH$ of forbidden subgraphs is much larger.  Thus, we just outline the
proof of (B).  

\begin{proof}[(Proof sketch of Main Theorem (B))]
Our proof has three cases. 
The first two cases use induction on $|V(G)|$, and
the third case simply constructs an explicit nb-coloring.

\textbf{Case 1: There exists $\bbs{W\subset V(G)}$ with $\bbs{2\le|W|\le |V(G)|-2}$ and
$\bbs{\rho_{s,G}(W)\le 3}$.}
By induction, $G[W]$ has an nb-coloring $I_W,F_W$.  We form a new graph $G'$
from $G$ by coloring $G[W]$ with $I_W,F_W$, and then identifying each vertex in
$W$ colored $I$ and identifying each vertex in $W$ colored $F$.  We call these
new vertices $w_i$ and $w_f$, and they retain their colors.  It is easy to check
that every nb-coloring of $G'$ extends to an nb-coloring of $G$ (by coloring
$G[W]$ with $I_W,F_W$).  So the key step is showing that $G'$ satisfies the
hypotheses of the Main Theorem.

Suppose that $G'$ contains a subset $W'$ such that $\rho_{s,G'}(W')\le -5$.  We
can check that also $\rho_{s,G}(W'\setminus\{w_i,w_f\}\cup W)\le -5$, a
contradiction.  That is, ``uncontracting'' the set $W'$ with potential too small 
in $G'$ gives a set with potential too small in $G$, which contradicts our
hypothesis.  So suppose instead that $G'$ contains a subgraph $H'$ that is
forbidden; that is $H'\in \HH$.  If $H'\notin\{K_4,M_7\}$, then
Corollary~3.8(iii) implies that $\rho_{s,H'}(V(H'))\le 0$, which yields
$\rho_{s,G}((V(H')\setminus\{w_i,w_f\}\cup W)\le -5$, a contradiction.
If $H'\in \{K_4,M_7\}$, then a short case analysis again reaches a
contradiction.

\textbf{Case 2: $G$ contains some ``reducible configuration'' (and Case~1 does
note apply).}  
Since Case~1 does not apply, we know that $\rho_{s,G}(W)\ge 4$ for all
$W\subseteq V(G)$ with $2\le |W|\le |V(G)|-2$.  We call this
inequality our ``gap lemma'', since it implies a gap
between the lower bound on $\rho_{s,G}$ required by the hypothesis ($-4$) and
the actual value of $\rho_{s,G}$ (at least 4).  A reducible
configuration is one that allows us to proceed by induction.  An easy example is
an uncolored vertex $v$ of degree at most 2.  By induction, $G-v$ has an
nb-coloring $I',F'$.  To extend this coloring to $G$, we color $v$ with $F$
unless all of its neighbors are colored $F$; in that case we color $v$ with $I$.
Our gap lemma has the following powerful consequence: For any $W\subsetneq V(G)$
and any $w\in W$ that is uncolored, we can color $G[W]$ with $w$ colored
$I$ and we can also color $G[W]$ with $w$ colored $F$.  This is because
precoloring a vertex decreases its potential (and that of any set containing
it) by at most 8.  So the gap lemma implies that each vertex subset (containing
the precolored vertex $w$) has potential at least $4-8=-4$.  Thus, the Main
Theorem still applies, even after precoloring $w$.

Let $L$ denote the set of degree 3 vertices that are uncolored and not
incident to any edge-gadget.  We claim that $G[L]$ is a forest.  Suppose, to
the contrary, that $G[L]$ contains a cycle $C$.  Since $G$ contains no
subgraph in $\HH$, cycle $C$ has successive vertices $v_1$ and $v_2$ such that their
neighbors outside of $C$, say $z_1$ and $z_2$ are not linked (this is a
technical term defined when constructing the family of forbidden subgraphs;
it means that adding the edge $z_1z_2$ would create a copy of a subgraph in
$\HH$).  Now we form a new graph $G(C,z_1,z_2)$ from $G$ by deleting $V(C)$ and
adding edge $z_1z_2$; if $z_1z_2$ already exists, then we replace it with an
edge-gadget.  Since $z_1$ and $z_2$ are not linked, $G(C,z_1,z_2)$ satisfies the
hypotheses of the Main Theorem.  It is straightforward to check that every
nb-coloring of $G(C,z_1,z_2)$ extends to an nb-coloring of $G$.

\textbf{Case 3: Neither Case 1 nor Case 2 applies.}  We use discharging to show
that $G$ is very nearly an uncolored graph with no edge-gadgets and consists
of an independent set of vertices of degree 4 and a set of vertices of degree 3
that induces a forest.  In this case, we can color the independent set with $I$
and color the forest with $F$.  If $G$ exactly matches this description, then
$\rho_{s,G}(V(G))=-\ell$, where $\ell$ is the number of components in the
forest.  Further, each place in the graph that differs from this description
slightly decreases $\rho_{s,G}(V(G))$.  By hypothesis, $\rho_{s,G}(V(G))\ge -4$,
so this number of differences is small (as is $\ell$).  In each case, we
explicitly construct an nb-coloring of $G$.
\end{proof}

In Section~\ref{alg-sec} we translate the proof of our Main Theorem into a
polynomial-time algorithm.
Implementing most of the steps is straightforward.  But two parts of this
process merit more comment.  In Section~\ref{goldberg-alg-sec}, we show how to
find a vertex subset $W$ with minimum potential; we can also further require
that $|W|$ be at least some constant distance away from 0 or from $|V(G)|$. 
This task reduces to a series of max-flow/min-cut problems, each of which runs
in time $O(|V(G)|^3\log |V(G)|)$.  Finally, to check whether two vertices are
linked, we simply use brute force.  This relies on the fact that each graph in
$\HH$ has at most 22 vertices, so $\HH$ has only finitely many graphs.  Thus we
can answer this question in time $O(|V(G)|^{20})$.

\section{Preliminaries}
\label{prelims-sec}
In Section~\ref{sharpness-sec} we construct the sharpness examples promised in
Theorems~\ref{thmA} and~\ref{thmB}.  In Section~\ref{potential-sec} we motivate
our choice of coefficients in the definitions of $\rho_{m,G}$ and $\rho_{s,G}$,
and record for reference the potentials of many small graphs.
Section~\ref{goldberg-alg-sec} presents an algorithm for finding a vertex
subset with lowest potential; 
this will be useful in Section~\ref{alg-sec}, where we convert our proofs that
certain graphs have nb-colorings into algorithms to construct those nb-colorings.
Finally, Section~\ref{defns-sec} collects all of our definitions, most of which
are standard.  To simplify our notation throughout, we assume that any sets $I$
and $F$ are disjoint.  This assumption is free, since induced subgraphs of
forests are forests.  We also assume that each pair of vertices is joined by at
most two edges, since allowing further parallel edges puts no further
constraints on the coloring.

\subsection{Sparse nb-critical Graphs} 
\label{constructing sharp examples sec}
\label{sharpness-sec}
Here we describe the sharpness examples in Theorems~\ref{main thm multi} and~\ref{main thm simple}.
For each $k \geq 1$, we construct a family of graphs \Emph{$G_k$} as follows.
The top of Figure \ref{sharp fig} shows $G_3$.
Let $V(G_k) = \{a,b,v_1, \ldots, v_{2k},c,d\}$ and 
$$E(G_k) = \{ab,ab,av_1,bv_1,v_{2k}c,v_{2k}d,cd,cd\} \cup \{v_1v_2, v_2v_3,
\ldots, v_{2k-1}v_{2k}\} \cup \{v_1v_2, v_3v_4, \ldots, v_{2k-1}v_{2k}\}. $$
To check that each $G_k$ is $(1.5,-1)$-sparse, we use induction on
$k$, as follows.  Fix $W\subseteq V(G_k)$.  Suppose that $W$ contains
$v_i,v_{i+1}$, for some $i\le 2k-2$.  Let $W'=W/\{v_i,v_{i+1},v_{i+2}\}$ and
$G'_k=G_k/\{v_i,v_{i+1},v_{i+2}\}$; here ``/'' denotes contraction.  Note that $G'_k\cong G_{k-1}$. By
hypothesis, $e(W')\le 1.5|W'|+1$.  Thus $e(W)\le e(W')+3 \le
1.5|W'|+3+1=1.5|W|+1$.
The case when no such $i$ exists is
straightforward, as is the base case. 
So $G_k$ is $(1.5,-1)$-sparse, as desired.

We claim that each $G_k$ is nb-critical.
To begin we show that $G_k$ is not near-bipartite.
The key observation, which is easy to check, is that
when $I,F$ is an nb-coloring of $G_k$ 
\begin{equation}\label{multiedge is superedge}
\mbox{if $vw$ is a multiedge, then } |I \cap \{v,w\}| = |F \cap \{v,w\}| = 1.
\end{equation}
Assume, contrary to our claim, that $G$ has an nb-coloring $I,F$.
Applying (\ref{multiedge is superedge}) to multiedge $ab$ shows that $|\{a,b\} \cap I|= 1$, which implies $v_1 \in F$.
Similarly, $|\{c,d\} \cap I|= 1$, so $v_{2k} \in F$.
We prove by induction that $v_{2i-1} \in F$ for all $i$, which contradicts \eqref{multiedge is superedge}
for multiedge $v_{2k-1}v_{2k}$.  Assume, by hypothesis, that $v_{2i-3}\in F$.
(The base case is when $i=2$.)
Applying \eqref{multiedge is superedge} to $v_{2i-3}v_{2i-2}$ shows that
$v_{2i-2}\in I$; this, in turn, means that $v_{2i-1}\in F$, as desired.
So $v_{2k-1}v_{2k}\in F$, which is a contradiction.  Thus, $G_k$ is not
near-bipartite.

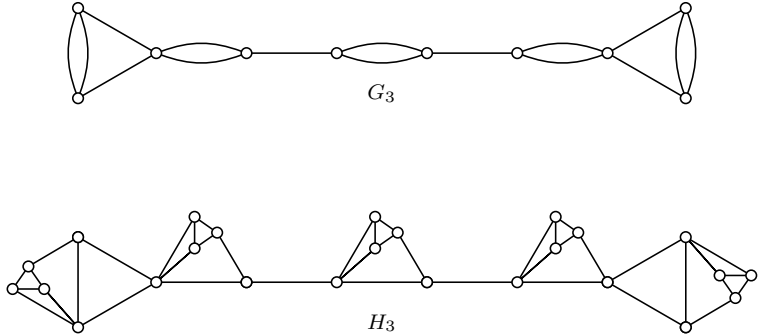
\begin{figure}[!h]
\centering
\begin{tikzpicture} [semithick, scale=1.2]
\tikzset{every node/.style=uStyle}

\draw (-.866,.5) node (a) {} -- (0,0) node (v1) {} -- (-.866,-.5) node (b) {}
(1,0) node (v2) {} -- (2,0) node (v3) {} (3,0) node (v4) {} -- (4,0) node (v5)
{} (5,0) node (v6) {} -- (5.866,.5) node (c) {} (v6) -- (5.866,-.5) node (d) {};
\draw (2.5,-.45) node[lStyle] {\footnotesize{$G_3$}};
\foreach \v/\w in {a/b, v1/v2, v3/v4, v5/v6, c/d}
  \draw (\v) edge[bend left=20] (\w) (\w) edge[bend left=20] (\v);

\begin{scope}[yshift=-1in]
\draw (-.866,.5) node (a) {} -- (0,0) node (v1) {} -- (-.866,-.5) node (b) {}
(1,0) node (v2) {} -- (2,0) node (v3) {} (3,0) node (v4) {} -- (4,0) node (v5)
{} (5,0) node (v6) {} -- (5.866,.5) node (c) {} (v6) -- (5.866,-.5) node (d) {};
\draw (2.5,-.45) node[lStyle] {\footnotesize{$H_3$}};

\foreach \left/\ang in {v1/0, v3/0, v5/0, b/90, c/-90} 
{
\begin{scope}[scale=.25,rotate=\ang]
\draw (\left) {} --+ (1.7,1.5) node {} --+ (2.7,2.2) node {} --+ (1.7,2.9) node
{} -- (\left) --+(1.7,1.5) --+ (1.7,2.9) (\left) --+ (4.0,0) node {} --+ (2.7,2.2);
\end{scope}
}

\end{scope}

\end{tikzpicture}
\caption{Two examples of nb-critical graphs.  Top: $G_3$ is a multigraph, with
vertices $a,b,v_1,\ldots,v_6,c,d$ in order from left to right.
Bottom: $H_3$ is formed from $G_3$ by replacing each pair of parallel edges by
a multiedge-replacement.\label{sharp fig}}
\end{figure}

To see that each subgraph $G_k-v_iv_{i+1}$ is near-bipartite, we color greedily
in the order $\{a,b,v_1,\ldots,v_i,$ $d,c, v_{2k},\ldots,v_{i+1}\}$, adding each vertex to any set where
it does not contradict the definition of $I,F$-coloring.  
For each other edge $e$, we can color $G-e$ similarly.
This completes the proof that each $G_k$ is nb-critical.

We now construct a family \Emph{$H_k$} of simple nb-critical graphs.
The bottom of Figure~\ref{sharp fig} shows $H_3$.
To do this, we define a \Emph{multiedge-replacement} for endpoints $a,b$ as
vertices $x_{ab},y_{ab},z_{ab}$ and edges $ab,ax_{ab},ay_{ab},x_{ab}y_{ab},$ $x_{ab}z_{ab},y_{ab}z_{ab},z_{ab}b$.
We say it is \emph{rooted} at $a$ and $b$ and that they are its roots.
As an example of an multiedge-replacement, consider the 5 leftmost (or 5 rightmost) vertices
in $H_3$ and the edges they induce, as shown on the bottom in Figure~\ref{sharp fig}.
To construct $H_k$ we replace each multiedge of $G_k$ with a
multiedge-replacement.  (These multiedge-replacements allow us to simulate
multiedges in simple graphs.)
It is straightforward to show by induction on $k$ that
each $H_k$ is $(1.6,-1)$-sparse.

The proof that $H_k$ is nb-critical follows from the proof that $G_k$ is
nb-critical, together with the fact (proved below) that in any nb-coloring
$I,F$ of a multiedge-replacement,
\begin{equation}\label{gadget is superedge}
\mbox{if the multiedge-replacement is rooted at $v$ and $w$, then } |I \cap \{v,w\}| = |F \cap \{v,w\}| = 1.
\end{equation}

We also need the observation that removing any edge from a multiedge-replacement
allows an nb-coloring with both roots colored $F$; this is easy to check
directly. This observation implies that every proper
subgraph of $H_k$ is near-bipartite.

We now prove (\ref{gadget is superedge}). 
If $z_{vw} \in I$, then $\{w,x_{vw},y_{vw}\}\subseteq F$.
So the circuit $v,x_{vw},y_{vw}$ implies that $v \in I$, and (\ref{gadget is superedge}) holds.
If instead $z_{vw} \in F$, then the circuit $x_{vw},y_{vw},z_{vw}$ forces
$\{x_{vw},y_{vw}\}\not\subset F$; by symmetry, assume $x_{vw}\in F$ and
$y_{vw}\in I$.  Thus $v\in F$.  But now the circuit $vwz_{vw}x_{vw}$
forces $w\in I$.  Again, (\ref{gadget is superedge}) holds.
This completes the proof of (\ref{gadget is superedge}). So $H_k$ has no
nb-coloring precisely because $G_k$ has no nb-coloring.  Thus, $H_k$ is nb-critical.

\subsection{Potential Functions} \label{define pot func sec}
\label{potential-sec}

Recall from the introduction that
 $$ \rho_{m,G}(W) = 3|W \cap U_p| + |W \cap F_p| - 2|e(W)|$$ 
and
 $$ \rho_{s,G}(W) = 8|W \cap U_p| + 3|W \cap F_p| - 5|e'(W)| - 11|e''(W)|.$$ 

Our choice of coefficients in $\rho_{m,G}$ and $\rho_{s,G}$ has a simple
explanation based on the constructions in the previous section.  We begin with $\rho_{m,G}$.
The ratio $3/2$ of the coefficients on $|W\cap U_p|$ and $e(W)$ arises because
$\lim_{k\to\infty}|E(G_k)|/|V(G_k)|=3/2$.  To understand the coefficient 1 on
$|W\cap F_p|$, consider an arbitrary vertex $w \in U_p$.
We create vertices $y_w, y_w' \in U_p$ and add edges $wy_w, wy_w', y_wy_w',
y_wy_w'$; see left of Figure~\ref{pot-fig}.
Because $y_wy_w'$ is a multiedge, every nb-coloring $I,F$ of this
graph must have $|I \cap \{y_w, y_w'\}|=1$, so $w \in F$.
Thus, functionally speaking, this construction is equivalent to moving $w$ from $U_p$ to $F_p$.
The weight of $w$ in $F_p$ represents the combined contribution to $\rho_{m,G}$ of $w, y_w,
y_w'$, and the associated edges: the 3 vertices and 4 edges give us $3(3) - 2(4) = 1$.

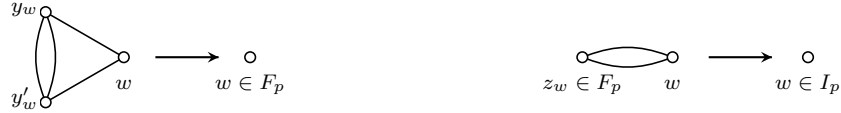
\begin{figure}[!h]
\centering
\begin{tikzpicture} [semithick, scale=1.2]
\tikzstyle{uStyle}=[shape = circle, minimum size = 4pt, inner sep = 1pt,
outer sep = 0pt, fill=white, semithick, draw]
\tikzstyle{lStyle}=[shape = rectangle, minimum size = 5pt, inner sep =
0.5pt, outer sep = 0pt, draw=none]

\tikzset{every node/.style=uStyle}

\draw (-.866,.5) node (a) {} -- (0,0) node (v1) {} -- (-.866,-.5) node (b) {};
\draw (-1.1,.53) node[lStyle] {\footnotesize{$y_w$}};
\draw (-1.1,-.45) node[lStyle] {\footnotesize{$y'_w$}};

\foreach \v/\w in {a/b}
  \draw (\v) edge[bend left=20] (\w) (\w) edge[bend left=20] (\v);
\draw (0,-.3) node[lStyle] {\footnotesize{$w$}};
\draw (1.4,0) node (v2) {};
\draw (1.4,-.3) node[lStyle] {\footnotesize{$w\in F_p$}};
\arrowdrawthick (.35,0) to (1.05,0);

\begin{scope}[xshift=2in]
\draw (0,0) node (zw) {} (1,0) node (w) {};
\draw (zw) edge[bend left=20] (w) (w) edge[bend left=20] (zw);
\draw (0,-.3) node[lStyle] {\footnotesize{$z_w\in F_p$}};
\draw (1,-.3) node[lStyle] {\footnotesize{$w$}};

\arrowdrawthick (1.4,0) to (2.1,0);
\draw (2.5,0) node {};
\draw (2.5,-.3) node[lStyle] {\footnotesize{$w\in I_p$}};
\end{scope}

\end{tikzpicture}
\caption{Constructions to require $w\in F_p$ (left) and $w\in I_p$
(right).\label{pot-fig}%
}
\end{figure}

To understand the coefficient 0 on $|W\cap I_p|$, consider an arbitrary
vertex $w \in U_p$, and create vertex $z_w \in F_p$ and add edges $wz_w, wz_w$;
see right of Figure~\ref{pot-fig}.
By construction, every nb-coloring $I,F$ of this graph must have $w
\in I$, and so we have mimicked moving $w$ from $U_p$ to $I_p$.
The weight of $w$ in $I_p$ represents the combined contribution of $w, z_w$, and
the two associated edges: $3 + 1 - 2(2) = 0$.

To double-check that our coefficients make sense, suppose we want to move a
vertex $v$ from $U_p$ to $F_p$.  We can also achieve this by adding a
vertex $w\in I_p$ and adding edge $vw$.  Functionally, now $v\in
F_p$, so combining the weights of $v$, $w$, and $vw$ should give the weight of a
single vertex in $F_p$, and it does: $3+0-2=1$.


Similarly, we can analyze the coefficients of $\rho_{s,G}$.  Note that
$\lim_{k\to \infty}|E(H_k)|/|V(H_k)|=8/5$.  To compute the weight of an
edge-gadget, we have $8(3)-5(7)= -11$, since it is simulated by a
multiedge-replacement.  To effectively move a vertex from
$U_p$ to $F_p$ or $I_p$, we use the same method as above, but with edge-gadgets
in place of multiedges. For a vertex in $F_p$ we count the contributions of 3
vertices, 2 edges, and one additional edge-gadget to get $8(3)-2(5)-11=3$.  For
a vertex in $I_p$ we count contributions of one vertex in $F_p$, one vertex in
$U_p$, and one edge-gadget to get $3+8-11=0$.

\begin{example}\label{example potentials for many things}
\label{example1}
We calculate the potential for several examples (assuming that no vertices are precolored).
\begin{enumerate}[(i)]
	\item $\rho_{m,K_k}(V(K_k)) = 3k-2{k\choose 2}=4k-k^2$ and
$\rho_{s,K_k}(V(K_k)) = 8k-5{k\choose 2} = \frac{21}2k - \frac52k^2$.
	\item $\rho_{m,W_5}(V(W_5)) = 3(6) - 2(10) = -2$ and $\rho_{s,W_5}(V(W_5)) = 8(6) - 5(10) = -2$.
	\item $\rho_{m,K_{2,2,2}}(V(K_{2,2,2})) = 3(6) - 2(12) = -6$ and
$\rho_{s,K_{2,2,2}}(V(K_{2,2,2})) = 8(6) - 5(12) = -12$.
	\item $\rho_{m,M_7}(V(M_7)) = 3(7)-2(11) = -1$ and $\rho_{s,M_7}(V(M_7)) = 8(7) - 5(11) = 1$.
	\item $\rho_{m,J_7}(V(J_7)) = 3(7)-2(12) = -3$ and $\rho_{s,J_7}(V(J_7)) = 8(7) - 5(12) = -4$.
	\item $\rho_{m,J_8}(V(J_8)) = 3(8) - 2(13) = -2$ and
$\rho_{s,J_8}(V(J_8)) = 8(8)-5(13) = -1$.
	\item $\rho_{m,J_{12}}(V(J_{12})) = 3(12)-2(20) = -4$ and
$\rho_{s,J_{12}}(V(J_{12})) = 8(12) - 5(20) = -4$. 
	\item $\rho_{m,G_k}(V(G_k)) = 3(2k+4) - 2 (3k+7) = -2$; further
$\rho_{m,G_k}(W) > -2$ for all $W \subsetneq V({G_k})$.
	\item $\rho_{s,H_k}(V(H_k)) = 8(5(k+2))-5(8(k+2)+1) = -5$; further
$\rho_{m,H_k}(W) > -5$ for all $W \subsetneq V({H_k})$. 
\end{enumerate}
The second statements in (viii) and (ix) are proved by induction on $k$.
\end{example}

\subsection{Computational Aspects of Sparsity} \label{computing sparsity sec}
\label{goldberg-alg-sec}
Recall that a graph $G$ is \Emph{$(a,b)$-sparse} if every nonempty 
$W\subseteq V(G)$ satisfies $|e(W)|\le a|W|-b$.  Similarly, $G$ is
\EmphE{$(a,b)$-tight}{2mm} if it is $(a,b)$-sparse and $|\EG| = a|\VG| - b$,
and $G$ is \EmphE{$(a,b)$-strictly sparse}{2mm} if it is $(a,b)$-sparse and no
subgraph is $(a,b)$-tight.
These sparsity notions have connections to many other concepts.  Lee and
Streinu~\cite[\S]{LS} survey several applications, emphasizing the
equivalence between $(2,3)$-tight graphs and Laman graphs for planar
bar-and-joint rigidity.  Sparsity is also related to minimal bends in
vertex contact representations of paths on a grid; see~\cite{AL}.  

Kostochka and the second author~\cite{KY} showed how to
color $(\frac{k}2-\frac1{k-1},\frac{k(k-3)}{2(k-1)})$-strictly
sparse graphs in polynomial time. Later they proved~\cite{KY2} that certain
known critical graphs are in fact
$(\frac{k}2-\frac1{k-1},\frac{k(k-3)}{2(k-1)})$-tight.  Their 
coloring algorithm fits into a larger body of work that uses the so-called
``Potential Method'' to color sparse graphs.
We will use the Potential Method to prove parts (A) and (B) of our Main Theorem.
When we color an $(a,b)$-sparse graph, 
a key step is to either find a proper $(a,b')$-tight subgraph $J$, for
specifically chosen $b' > b$, or else report that no such $J$ exists.
We may also impose additional constraints, for instance that $2\le |J|\le
|V(G)|-2$ or that $|J|$ is maximized or minimized.

The \Emph{maximum average degree} of a graph $G$ is the minimum $a$ such that
$G$ is $(a/2,0)$-sparse.
Researchers have recently discovered
new applications for finding a subgraph with maximum average degree, and
algorithms achieving this have grown in interest (Google Scholar claims that a
paper with a foundational algorithm~\cite{G} for this problem has over 250 citations).
Finding the subgraph with largest maximum average degree among subgraphs whose
order is bounded either from above or below is conjectured to be 
computationally hard~\cite{AC}, but it can be done in polynomial
time~\cite{CPWLT} if the bounds are $O(1)$ away from being trivial.
We are unaware of any work bounding the subgraph's order from
both above and below simultaneously.

Much of the work above generalizes to hypergraphs.
Fix a hypergraph $\HHH$, vertex weights $w_v:V(\HHH) \rightarrow \Re^+$, and
edge weights $w_e:E(\HHH) \rightarrow \Re^+$.
The potential of a vertex set $X$, denoted \EmphE{$\rho(X)$}{-4mm}, is defined as
$\rho(X) = \sum_{u \in X} w_v(u) - \sum_{f \subseteq X} w_e(f)$.
Hypergraph $\HHH$ is \Emph{$b$-sparse} if $\rho(X) \geq b$ for every nonempty
vertex subset $X$.  A graph $G$ is $(a,b)$-sparse if and only if for weights
$w_v \equiv a$, $w_e \equiv 1$ we have that $G$ is $b$-sparse.

Lee and Steinu~\cite{LS} gave an algorithm to find an $(a,b)$-tight subgraph of
maximum order when $0 \leq b < 2a$, and Streinu and Theran~\cite{ST}
generalized it to hypergraphs.  Goldberg~\cite{G} gave an algorithm to find a
subgraph with largest maximum average degree.  The core routine of Goldberg's
algorithm is a max-flow/min-cut method; for a fixed $a'$ it finds the largest
$b'$ such that the graph is $(a',b')$-sparse and returns an $(a',b')$-tight
subgraph.  Goldberg's algorithm may return the empty subgraph, so it always
returns with $b' \geq 0$.
Kostochka and the second author~\cite{KY} modified Goldberg's algorithm to fit
the needs of the Potential Method, but they only proved the modifications work for
the case needed in that paper.  Goldberg~\cite{G} also generalized his work to
allow for edge weights and ``vertex weights,'' but his vertex weights are
functionally equivalent to the presence of loops and differ from what we do here.
To simplify current and future work with the Potential Method, we describe here
the most general version of the algorithm in~\cite{KY}.

\begin{theorem}\label{find small potential 1}
Fix a hypergraph $\HHH$, vertex weights $w_v:V(\HHH) \rightarrow \Re^+$, and
edge weights $w_e:E(\HHH) \rightarrow \Re^+$.  We can find a vertex subset $W$ such that $\rho(W) = \min_{U \subseteq V(\HHH)} \rho(U)$ in time $O((|V(\HHH)|+|E(\HHH)|)^3)$.
If each hyperedge has bounded size, then we can find $W$ in time $O((|V(\HHH)|+|E(\HHH)|)^2\log(|V(\HHH)|+|E(\HHH)|))$.
\end{theorem}
\begin{proof}
The following is a straightforward adaptation of Goldberg's argument in~\cite{KY}; we get to add weights for free. 
(Figure~\ref{goldberg-example-fig} shows an example.)

Using a Max-flow/Min-cut algorithm, we will find a minimum weight cut $E'$ in the
following auxiliary digraph $P$.  Let $V(P) = \{s,t\} \cup V(\HHH) \cup E(\HHH)$.
For each vertex $v$ of $\HHH$, add an arc from $s$ to the corresponding vertex in $P$ with capacity $w_v$.
For each edge $e$ of $\HHH$, add an arc from the corresponding vertex in $P$ to $t$ with capacity $w_e$.
For each vertex $v$ in an edge $e$ of $\HHH$, add an arc in $P$ with infinite
capacity from the vertex corresponding to $v$ to the vertex corresponding to $e$.

Let \Emph{$w_e^{tot}$} denote the sum of all edge weights in $\HHH$.
Observe that if $v$ is a vertex in an edge $e$ (in $\HHH$), then either $sv$ is in
the edge cut $E'$ of $P$ or else $et$ is in $E'$.
Let $W=\{v\in V(\HHH):~sv\in E'\}$, and note that $e(W)=\{e\in E(\HHH):~et\notin E'\}$.
Thus, the weight of $E'$ is precisely 
\begin{align*}
&\sum_{x\in W}w_v(x)+\sum_{f\notin \HHH[W]}w_e(f)\\
= &\sum_{x\in W}w_v(x)- \sum_{f\in \HHH[W]}w_e(f)+ \sum_{f\in E(\HHH)}w_e(f) \\
= &~\rho(W)+w_e^{tot}.
\end{align*}
The algorithm's running time is dominated by the cost of finding a minimum $s-t$ edge-cut in $P$.
Since $|V(P)| = |V(\HHH)|+|E(\HHH)| + 2$, the algorithm of Karzanov~\cite{K}
runs in time $O((|V(\HHH)|+|E(\HHH)|)^3)$.
If each hyperedge has bounded size, then $|E(P)|= O(|V(\HHH)|+|E(\HHH)|)$, so
the algorithm of Sleater and Tarjan~\cite{ST2} runs in time
$O((|V(\HHH)|+|E(\HHH)|)^2\log(|V(\HHH)|+|E(\HHH)|))$.
\end{proof}

\begin{figure}[!t]
\centering
\begin{tikzpicture}[semithick, scale=.85]
\tikzstyle{uStyle}=[shape = circle, minimum size = 10pt, inner sep = 1pt,
outer sep = 0pt, fill=white, semithick, draw]
\tikzstyle{uBStyle}=[shape = circle, minimum size = 13pt, inner sep = 1pt,
outer sep = 0pt, fill=white, semithick, draw]
\tikzset{every node/.style=uStyle}

\draw (0,2.5) node (u) {\footnotesize{$u$}} -- (-1,1) 
node (v) {\footnotesize{$v$}}
-- (1,1) node (w) {\footnotesize{$w$}}
(-1,-1) node (x) {\footnotesize{$x$}} -- (1,-1) node (y) {\footnotesize{$y$}} --
(0,-2.5) node (z) {\footnotesize{$z$}};
\draw (u) -- (w) -- (y) (v) -- (x) -- (z);

\foreach \start/\end/\where/\far/\outer/\value in {
u/v/above/.7/5/5,
u/w/above/.7/5/8,
v/w/below/.5/0.5/2,
v/x/left/.5/.5/7,
w/y/right/.5/.5/5,
x/y/above/.5/1.5/3,
x/z/below/.3/2.5/6,
y/z/below/.3/2.5/7
}
\draw (\start) -- (\end) 
node[pos=\far,\where, lStyle, outer sep=\outer, fill=none, inner sep=1, shape=circle]{\scriptsize{\value}};

\foreach \where/\dir/\far/\value in {u/270/.4/3, v/30/.4/4, w/150/.4/2,
x/330/.4/1, y/210/.4/9, z/90/.4/15}
\draw (\where) ++ (\dir:\far cm) node[lStyle, outer sep=1, fill=none, inner
sep=1, shape=circle]{\scriptsize{\value}};

\node[above = .7cm of u, lStyle] (u') {};
\node[above left=.1cm and .25cm of v, lStyle] (v') {};
\node[below right=.23cm and .89cm of v, lStyle] (v'') {};
\node[above right=.15cm and .25cm of w, lStyle] (w') {};
\node[below right=.15cm and .29cm of x, lStyle] (x') {};
\node[below left=.15cm and .23cm of x, lStyle] (x'') {};
\draw[ultra thick] (v'.center) .. controls (u'.center) and (u'.center)  ..
(w'.center)
.. controls ++(300:1cm) and ++(360:1cm) .. (v''.center) .. controls ++(180:1cm) and
++(90:2cm) ..  (x'.center) .. controls ++(240:.35cm) and ++(300:.25cm) ..
(x''.center) -- cycle; .. controls ++(180:.05cm) .. (v'.center);


\begin{scope}[xshift=3.5in, yshift=-1.15in, xscale=1.25, yscale=2, rotate = 90]

\draw (0,0) node (s) {\scriptsize{$s$}} 
(3,0) node (t) {\scriptsize{$t$}} 
(1,3.5) node (u) {\footnotesize{$u$}}
(1,2.1) node (v) {\footnotesize{$v$}}
(1,0.7) node (w) {\footnotesize{$w$}}
(1,-.7) node (x) {\footnotesize{$x$}}
(1,-2.1) node (y) {\footnotesize{$y$}}
(1,-3.5) node (z) {\footnotesize{$z$}}

(2,3.5) node[uBStyle] (uv) {\footnotesize{$uv$}}
(2,2.5) node[uBStyle] (uw) {\footnotesize{$uw$}}
(2,1.5) node[uBStyle] (vw) {\footnotesize{$vw$}}
(2,0.5) node[uBStyle] (vx) {\footnotesize{$vx$}}
(2,-.5) node[uBStyle] (wy) {\footnotesize{$wy$}}
(2,-1.5) node[uBStyle] (xy) {\footnotesize{$xy$}}
(2,-2.5) node[uBStyle] (xz) {\footnotesize{$xz$}}
(2,-3.5) node[uBStyle] (yz) {\footnotesize{$yz$}};

\draw[gray!50!white, semithick] (u) edge[bend left = 40] (uv);
\draw[gray!50!white, semithick] (w) edge[bend left = 35] (uw);
\draw[gray!50!white, semithick] (v) edge[bend left = 10] (vw);
\draw[gray!50!white, semithick] (v) edge[bend right = 35] (vx);
\draw[gray!50!white, semithick] (w) edge[bend right = 40] (wy);
\draw[gray!50!white, semithick] (x) edge[bend right = 40] (xy);
\draw[gray!50!white, semithick] (x) edge[bend right = 15] (xz);
\draw[gray!50!white, semithick] (y) edge[bend right = 40] (yz);
\foreach \vert in {u,v,w,x,y,z}
  \draw (s) -- (\vert);
\foreach \vert in {uv, uw, vw, vx, wy, xy, xz, yz}
  \draw (t) -- (\vert);

\tikzset{
    partial ellipse/.style args={#1:#2:#3}{
        insert path={+ (#1:#3) arc (#1:#2:#3)}
    }
}

\draw[ultra thick] (s) [partial ellipse=-15:55:.35cm and 1.7cm];
\draw[ultra thick] (t) [partial ellipse=180:235:.35cm and 1.7cm];

\foreach \start/\end/\where/\far/\outer/\value in {
s/u/left/.7/6.5/3,
s/v/left/.7/2.5/4,
s/w/left/.7/0.5/2,
s/x/right/.7/0.5/1,
s/y/right/.7/2.5/9,
s/z/right/.7/6.5/15,
t/uv/left/.79/6.5/5,
t/uw/left/.81/4.5/8,
t/vw/left/.83/2.5/2,
t/vx/left/.86/1.5/7,
t/wy/right/.86/1.5/5,
t/xy/right/.83/2.5/3,
t/xz/right/.81/4.5/6,
t/yz/right/.79/6.5/7,
u/uw/left/.4/1.5/(3),
v/uv/right/.4/1.5/(4),
w/vw/left/.4/0.5/(2),
x/vx/left/.4/1.5/(1),
y/wy/left/.4/1.5/(5),
y/xy/right/.4/0.4/(3),
z/xz/left/.4/1.0/(6),
z/yz/right/.4/0.5/(7)
}
\draw (\start) -- (\end) 
node[pos=\far,\where, lStyle, outer sep=\outer, fill=none, inner sep=1, shape=circle]{\scriptsize{\value}};
\end{scope}

\end{tikzpicture}
\caption{Left: A graph $G$ with weights on edges and vertices, and its subgraph
with minimum potential.  
The potential of the subgraph of $G$ indicated is $3+4+2+1-(5+8+2+7)=-12$.
Right: A minimum cut and maximum flow---in an auxilliary graph---that
correspond to the subgraph of $G$ with minimum potential. Flow value are
shown as $(a)$ and capacities as $a$; recall that each ``center'' edge has
infinite capacity.  Edges without flow value shown have the flow needed to
conserve flow at their endpoints. (Curved, light gray edges do not receive any
flow, but are drawn for completeness.)
The minimum cut has value $3+4+2+1+5+3+6+7=31$.  To calculate the minimum
potential of a subgraph in $G$ we subtract the sum of capacities of top 
edges ($5+8+2+7+5+3+6+7=43$) from the value of the maximum flow.  Thus, the
potential is $31-43=-12$.\label{goldberg-example-fig}
} 
\end{figure}
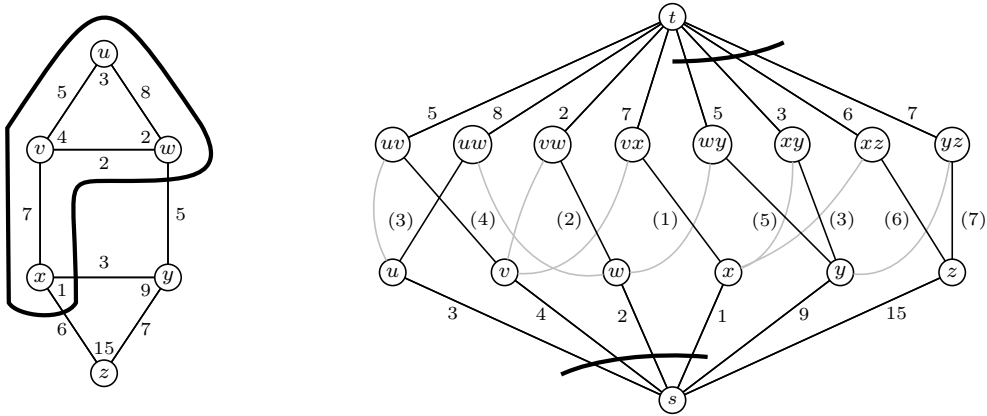

We have two immediate uses for the vertex weights.
First, we can adapt the algorithm to the problem of extending a precoloring,
as discussed in Section \ref{define pot func sec}.
Second, we can specify vertices as mandatory to include in
our subgraph, as we show in the proof of our next result.

\begin{theorem}\label{find small potential 2}
Theorem~\ref{find small potential 1}
can be adapted to allow the condition that we find the largest or smallest
subset among optimal sets. Further, for constants $m_1$ and $m_2$, we can also require
the subset to have order at least $m_1$ and at most $|\HHH| - m_2$, where the
algorithm now runs in time $O(|V(\HHH)|^{m_1 + m_2}(|V(\HHH)|+|E(\HHH)|)^3)$.
\end{theorem}
\begin{proof}[(Proof sketch)]
To find an optimal subgraph of maximal order, we increase the weight of
each vertex by $\epsilon$. To find one of minimal order, we decrease each weight.

Let $W$ denote the vertex subset returned by the algorithm in Theorem~\ref{find
small potential 1}.  To ensure that $|W|\le |\VG|-m_2$, we remove a set $X$ of
$m_2$ vertices before running the algorithm.  By considering all ${|V(\HHH)|
\choose m_2}$ choices for $X$, we find our desired $W$.

To ensure that $|W|\ge m_1$, we choose a set $Y$ of $m_1$ vertices
and add a new hyperedge over those vertices, with extremely high capacity.
Any optimal cut must contain those vertices, so we can account for the weight of this
new hyperedge at the end.  Again we consider all possible choices for $Y$.
The theorem follows from the inequality ${|V(\HHH)| \choose m_1}{|V(\HHH)|
\choose m_2} \leq |V(\HHH)|^{m_1 + m_2}$.
\end{proof}

\begin{corollary}\label{find small potential 3}
Let $m_1, m_2$ be fixed nonnegative integers.
If $G$ is a connected graph with $O(|\VG|)$ edges, then a largest (or
smallest) vertex subset $W$ with smallest potential satisfying $m_1 \leq |W| \leq |\VG| - m_2$ can be found in 
time $O(|\VG|^{2+m_1 + m_2}\log(|\VG|))$.
\end{corollary}

\subsection{Definitions and Notation}
\label{defns-sec}
For completeness, below we collect our definitions, many of which are standard.
A graph $G$ consists of a vertex set $\VG$ and a multiset $\EG$ of unordered
pairs of vertices, called the edge multiset.
An edge $e$ that is the pair of vertices $v$ and $w$ is written as $e=vw$.
This paper deals with loopless graphs, so if $vw$ is an edge, then $v \neq w$.
Two edges $e_1, e_2$ are \Emph{parallel} if they are the same pair of vertices.
A \EmphE{multiedge}{4mm} is an equivalence class of edges that contains exactly two edges.
(Recall that we allow at most two parallel edges joining any pair of vertices,
since more parallel edges put no further constraints on the coloring.)
A graph is \Emph{simple} if it has no multiedges.

A \Emph{circuit} of length $k$ in a graph is a sequence of vertices $v_1, v_2,
\ldots, v_{k+1}$ and edges $e_1, e_2, \ldots, e_k$ such that (a) $v_1 =
v_{k+1}$, (b) $e_i = v_iv_{i+1}$, and (c) $e_i \neq e_j$ when $i \neq j$.
In particular, a multiedge forms a circuit of length $2$.
A \Emph{forest} is a graph with no circuits.

For a vertex subset $W \subseteq \VG$, let $e(W) = \{e \in E(G) : e\mbox{ has
both endpoints in }W\}$. We write $G[W]$ for the subgraph induced by $W$; that
is $V(G[W])=W$ and $E(G[W])=e(W)$.
A vertex subset $W$ is \Emph{independent} if $|e(W)| = 0$.
For each vertex $v$, let $d(v)$ denote the number of edges (including
edge-gadgets) incident to $v$.  Specifically, multiedges contribute 2 to the
degree of each endpoint, but edge-gadgets only contribute 1.
We write $\Delta(G)$ and $\delta(G)$ to denote the maximum and minimum degrees,
respectively.
Let $N(v)$ denote the set of vertices that share an edge or edge-gadget with $v$.
If $G$ is simple, then $d(v) = |N(v)|$.


\section{Constructing $\HH$}
\label{constructH-sec}
\subsection{Linked Vertices} \label{Linked sec}
In this subsection and the next, we construct the family $\HH$ of subgraphs
forbidden in part (B) of the Main Theorem.  On a first pass, the reader may
prefer to focus on the proof of part (A), since it uses many of the same ideas,
but is much easier than that of part (B).  In that case, we recommend skipping
to Section~\ref{proof-sec}.

While trying to color $G$, we often want to color by minimality a graph $J$
formed by adding an edge to some proper subgraph of $G$.  A major hurdle we
face is showing that $J$ satisfies the hypotheses of the Main Theorem.  To
understand when adding an edge creates a copy of some forbidden subgraph, we
study the following notion of linked vertices.

\begin{defn} \label{linked def}
Let $H$ be an nb-critical graph.
Form $H'$ from $H$ by removing a single edge $vw$.
Vertices $s,t$ are \EmphE{linked in $G$}{-4mm} if $G$ contains a subgraph $H''$ that
is isomorphic to $H'$, where vertices $s,t \in V(H'')$ correspond to $v,w$ in the isomorphism.
We call $H$ the \EmphE{linking graph}{-4mm}.
\end{defn}

As an example, if $G$ contains a copy of $K_4-e$, then its non-adjacent vertices
are linked.  The following lemma generalizes a key concept from the proof of
\eqref{gadget is superedge} in Section~\ref{sharpness-sec}.

\begin{lemma} \label{linked vertices lemma}
If vertices $s,t$ are linked in a graph $J$, then for any
nb-coloring $I,F$ of $J$, either
\begin{itemize}
\item[(i)] $\{s,t\} \subseteq I$, or 
\item[(ii)] $\{s,t\} \subseteq F$ and there exists a path from $s$ to $t$ in $J[F]$.
\end{itemize}
\end{lemma}
\begin{proof}
We use notation as in Definition \ref{linked def}, and let $e=st$.
Suppose, to the contrary, that $J$ has an nb-coloring $I,F$ 
with $|\{s,t\} \cap F| \geq 1$ and that if $s,t\in F$, then $G[F]$ has no path
from $s$ to $t$.  Now $I,F$ is also an nb-coloring for $J+e$.
Since $I,F$ restricts to an nb-coloring for $H''+e$,
and $H''+e\cong H$, this contradicts our assumption that $H$ is not near-bipartite.
\end{proof}

\begin{lemma} \label{linked have large degree}
Using the notation of Definition \ref{linked def},
we know that $\delta(H) \geq 3$. So for each $w \in V(H'')\setminus \{s,t\}$ we
have $d_G(w) \geq 3$.
\end{lemma}
\begin{proof}
The second statement clearly follows from the first, so we prove the first.
Suppose, to the contrary, that $w\in \VH$ and $d_H(w) \leq 2$.
By nb-criticality, $H-w$ has an nb-coloring $I',F'$.  If $I'$
contains a neighbor of $w$ in $H$, then let $I = I'$ and $F = F' \cup \{w\}$. 
Otherwise let $I = I' \cup \{w\}$ and $F = F'$.  But now $I,F$ is an
nb-coloring of $H$, which contradicts that $H$ is nb-critical.
\end{proof}

\subsection{The Forbidden Subgraphs} \label{define H sec}
To define $\HH$ we first define an infinite family of graphs $\HH'$.
The graphs $K_4$, $W_5$, $J_7$, and $J_{12}$ are called \Emph{base graphs}.
We define $\HH'$ recursively: each graph in $\HH'$ is either a base graph or
else is formed by merging smaller graphs in
$\HH'$ in a certain way.
To explain this construction, we define \emph{specially-linked}
vertices (in Definition~\ref{defn H'}); this idea builds on Definition~\ref{linked
def}, but also assumes that the nb-critical graph $H$ is in $\HH'$.

All graphs in $\HH'$ contain no edge-gadgets and only contain uncolored vertices.
This assumption will persist throughout this subsection.  (However, when we
forbid a subgraph in the Main Theorem, we also forbid it with precolored
vertices and/or with some edges replace by edge-gadgets, since such variations
are no easier to color.)

\begin{definition}
\label{defn H'}
If two vertices $s,t$ are linked in a graph $J$, then they are \Emph{specially-linked}
if the linking nb-critical graph $H$ (in Definition~\ref{linked def}) is in
$\HH'$, where $\HH'$ is defined next.\footnote{Formally, perhaps we should
define specially-linked only after defining $\HH'$.  Explicitly making that
substitution in (ii.c) below gives a correct recursive definition of $\HH'$, but also
renders (ii.c) harder to parse.}

A graph $H$ is in \Emph{$\HH'$} if (i) $H$ is one of the four base graphs, or (ii) $H$
is nb-critical and contains an induced cycle $C = (x_1, x_2, \ldots, x_k)$ such
that each of the following three conditions holds:
\begin{enumerate}
\item[(ii.a)] the length of the cycle, $k$, satisfies $k \in \{3,5\}$,
\item[(ii.b)] each vertex in $C$ has degree $3$, and 
\item[(ii.c)] if $\{x_{j-1},x_{j+1},z_j\}$ denotes $N(x_j)$, with indices modulo
$k$, then $z_j$ and $z_{j+1}$ are specially-linked in $H - C$ (whenever $z_j \neq z_{j+1}$). 
\end{enumerate}

The family of graphs $\HH$ is defined as 
$$\HH = \{H \in \HH' :  \rho_{s,H}(W) \geq -4~\mbox{for all}~W \subseteq
\VH\}.\aside{$\HH$}$$
\end{definition}

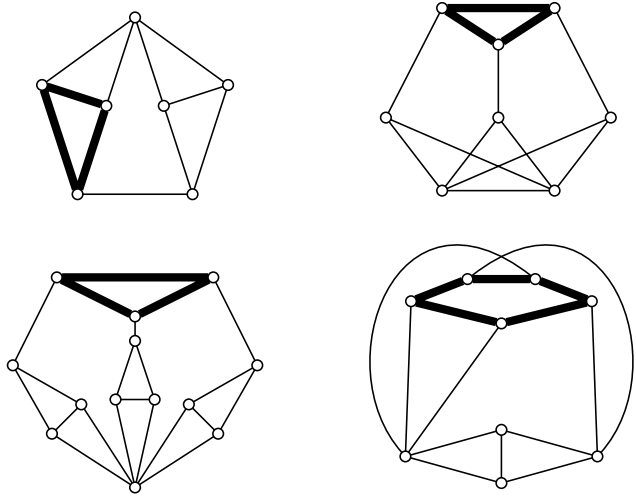
\begin{figure}[!h]
\centering
\begin{tikzpicture} [semithick,scale=.65]
\tikzset{every node/.style=uStyle}
\tikzstyle{lStyle}=[shape = rectangle, minimum size = 0pt, inner sep =
0pt, outer sep = 0pt, draw=none]

\draw (0,0) node (v) {} 
(-.4,1.8) node (w4) {}
(.4,1.8) node (w5) {}
(0,3) node (w6) {}
(0,3.5) node (z2) {}
(-1.7,1.1) node (w1) {}
(-1.1,1.7) node (w2) {}
(-2.5,2.5) node (w3) {}
(-1.6,4.3) node (z1) {}
(1.7,1.1) node (w7) {}
(1.1,1.7) node (w8) {}
(2.5,2.5) node (w9) {}
(1.6,4.3) node (z3) {};

\draw 
(v) -- (w1) -- (w2) -- (w3) -- (w1)  (w2) -- (v)
(v) -- (w4) -- (w5) -- (w6) -- (w4)  (w5) -- (v)
(v) -- (w7) -- (w8) -- (w9) -- (w7)  (w8) -- (v)
(w3) -- (z1)
(w6) -- (z2)
(w9) -- (z3);
\draw[line width=1.1mm] (z1) -- (z2) -- (z3) -- (z1);

\begin{scope}[yshift = 3.00in, scale=1] 
\newcommand\rad{2cm}
\foreach \i in {1,...,5}
  \draw (18+72*\i:\rad) node (v\i) {};
\foreach \i/\j in {1/2, 2/3, 3/4, 4/5, 5/1}
  \draw (v\i) -- (v\j);
\draw node (w1) at (barycentric cs:v1=1,v3=1) {};
\draw node (w2) at (barycentric cs:v1=1,v4=1) {};
\draw (v1) -- (w1) -- (v3) (w1) -- (v2) (v1) -- (w2) -- (v4) (w2) -- (v5);
\draw[line width = 1.1mm] (v2) -- (v3) -- (w1) -- (v2);
\end{scope}

\begin{scope}[scale=1.15, yscale=1.3, xshift=2.15in, yshift=1.6in] 
\draw (0,0) node (w1) {} (2,0) node (w2) {} (-1,1) node (z1) {} (1,1) node (z2)
{} (3,1) node (z3) {} (0,2.5) node (v1) {} (1,2) node (v2) {} (2,2.5) node (v3) {};
\draw (v1) -- (v2) -- (v3) -- (v1) (z1) -- (w1) -- (z2) -- (w2) -- (z3) -- (w1)
-- (w2) -- (z1) -- (v1) (z2) -- (v2) (z3) -- (v3);
\draw[line width=1.1mm] (v1) -- (v2) -- (v3) -- (v1);
\end{scope}

\begin{scope}[xshift=2.95in,yshift=.25in, yscale=.825, xscale=1.05, scale=1.1]
\draw 
(-1.7,0) node (v1) {}
(0,-.6) node (v2) {}
(0,.6) node (v3) {}
(1.7,0) node (v4) {}
(0,3) node (z1) {}
(-1.6,3.5) node (z2) {}
(1.6,3.5) node (z3) {}
(-.6,4) node (z4) {}
(.6,4) node (z5) {};
\draw (v2) -- (v1) -- (v3) -- (v2) -- (v4) -- (v3);
\draw[line width=1.1mm] (z1) -- (z2) -- (z4) -- (z5) -- (z3) -- (z1);
\draw (z1) -- (v1) -- (z2) (z3) -- (v4);
\draw (z5) .. controls ++(135:3.5) and ++(125:2.5) .. (v1);
\draw (z4) .. controls ++(45:3.5) and ++(55:2.5) .. (v4);
\end{scope}
\end{tikzpicture}
\caption{$M_7$, $J_8$ (top), and two other graphs in $\HH$.  For each
graph, the cycle $x_1,\ldots,x_k$ is shown in bold.}
\end{figure}

\begin{remark}\label{two zi}
If $H \in \HH'$ and $H$ is not a base graph, then there exists $j$ such that $z_j \neq z_{j+1}$.
\end{remark}
\begin{proof}
If not, then $\{x_1,\ldots,x_k, z_1\}$ induces either $K_4$ or $W_5$, which contradicts that $H$ is nb-critical.
\end{proof}

Examples of graphs in $\HH'$ include $M_7$ and $J_8$; the graph $K_{2,2,2}$ is 
nb-critical, but is not in $\HH'$ since it is 4-regular, so fails condition
(ii.b) in Definition~\ref{defn H'}.
In Lemma~\ref{small examples in hh}
we will show that, among graphs in $\HH'$ that are not base graphs, $M_7$ is the
smallest and $J_8$ is the second smallest (although we do not prove that $J_8$
is uniquely the second smallest).  


In the introduction, we claimed that each graph in $\HH$ is $4$-critical and
that $\HH$ is a finite family.  We now prove these claims, as well as a few
properties of $\HH'$ that we will need later.  The most important result from 
this subsection is Corollary~\ref{how to show H is not present}.

\begin{lem}\label{HH is 4-crit}
Each graph $H\in \HH'$ is $4$-critical.
\end{lem}
\begin{proof}
We use induction on $|\VH|$.  It is
easy to check that each base graph is 4-critical (due to symmetry, case analysis is quite short).

Now we assume that $H \in \HH'$ and $H$ is larger than the base graphs.
By definition, $H$ has a cycle $C$ that satisfies $ii.a$, $ii.b$, and $ii.c$ from Definition \ref{defn H'}.
To prove that $H$ is $4$-critical, we show that $\chi(H) \geq 4$ and that
$\chi(H-e) \leq 3$ for every $e \in \EH$.
The latter is easy: since $H$ is nb-critical, $H-e$ is near-bipartite, so $\chi(H-e) \leq 3$.

Assume, to the contrary, that $H$ admits a proper $3$-coloring $\vph$.
By definition, if $z_i \neq z_{i+1}$, then $z_i$ and $z_{i+1}$ are specially-linked in $H-C$.
By induction, this implies that the linking graph $J$ is a $4$-critical graph.
A basic fact of $4$-critical graphs is that for any edge $vw \in \EJ$, any
proper $3$-coloring of $J-vw$ uses the same color on $v$ and $w$.
It follows that $\vph(z_i) = \vph(z_j)$ for all $i,j\in\{1,\ldots,k\}$.
But now $\vph(z_1)$ is forbidden from use on each vertex of the odd cycle $C$;
since $C$ has no 2-coloring, this contradicts the existence of $\vph$. 
\end{proof}

\begin{lemma}\label{small examples in hh}
If $H \in \HH'$ and $H \notin \{K_4, W_5, M_7, J_7\}$, then $|\VH| \geq 8$.
\end{lemma}
\begin{proof}

Fix $H \in \HH' \setminus \{K_4, W_5, M_7, J_7\}$.  
Clearly the unique smallest graph in $\HH'$ is $K_4$.  If $H\ne J_{12}$, then Remark \ref{two zi}
implies that $H-C$ contains two linked vertices, so $H-C$ has at least $4$ vertices.
Thus, $H$ has at least $7$ vertices.  Further, $H$ has at least $8$
vertices unless $H-C$ is $K_4-e$ and $C$ is a $3$-cycle.  In this case,
$H$ is $M_7$, which contradicts our hypothesis.
\end{proof}

\begin{cor}
The families $\HH'$ and $\HH$ satisfy the following four properties.
\label{how to show H is not present}
\label{cor1}
\begin{enumerate}[(i)]
\item If $H \in \HH'$, then $\rho_{s,H}(\VH) \leq (10-|\VH|)/3 \leq 2$.
\item If $H \in \HH$, then $|\VH| \leq 22$. 
\item If $H \in \HH'$ and $H \notin \{K_4, M_7\}$, then $\rho_{s,H}(\VH) \leq 0$.
\item If $H \in \HH'$ and $\rho_{s,H}(\VH) \geq 0$, then for every $\emptyset \neq
W \subsetneq \VH$, we have $\rho_{s,H}(W) \geq 6$.
\end{enumerate}
\end{cor}
\begin{proof}
We start with (i).  Lemma~\ref{HH is 4-crit} implies $H$ is $4$-critical.
Kostochka and Yancey \cite{KY} proved that if $H$ is $4$-critical, then $|\EH| \geq (5|\VH|-2)/3$.
So $\rho_{s,H}(\VH)=8|\VH|-5|\EH|\le (24|\VH|-25|\VH|+10)/3\le 2$; the final
inequality uses that $|\VH|\ge 4$. Now (ii) follows from (i), since
$\rho_{s,H}(\VH)\ge -4$.

Next we consider (iii).  Kostochka and Yancey \cite{KY2} constructed a family of
4-critical graphs that they called $4$-Ore graphs, and proved that if $H$ is
$4$-critical and not $4$-Ore, then $|\EH| \geq (5|\VG|-1)/3$.  They also showed
that if $H$ is $4$-Ore, then $|\VH| \equiv 1 (\mbox{mod }3)$.  Moreover, if $H$
is $4$-Ore and $|\VH| \leq 7$, then $H$ is $K_4$ or $M_7$.

Recall from Example \ref{example potentials for many things} that
$\rho_{s,W_5}(V_{W_5}) = -2$ and $\rho_{s,S}(V_S) = -4$.  So we assume that
$H\in\HH'\setminus\{K_4,W_5,M_7,J_7\}$.
If $H$ is $4$-Ore, then the previous paragraph and
Lemma~\ref{small examples in hh} imply $|\VH| \geq 10$. So (i) implies (iii).
If $H$ is not $4$-Ore, then $|\EH|\ge (5|\VH|-1)/3$ implies $\rho_{s,H}(\VH)
\leq (5-|\VH|)/3$. Now $|\VH| \geq 8$ again implies (iii).

Finally, consider (iv).
We omit the tedious calculations when $G \in \{K_4,M_7\}$.
The proof of Part (iii) shows that if $H \in \HH'$ and $\rho_{s,H}(\VH) = 0$,
then $H$ is a $4$-Ore graph with $10$ vertices.
It was shown in \cite{KY2} (see Claim 16) that if $H$ is $4$-Ore and $\emptyset
\neq W \subsetneq \VH$, then $|E(W)| \leq (5|W|-5)/3$.
Since $\rho_{s,H}$ is integer valued and $|W| < |\VH| = 10$, part (iv) holds because 
$$\rho_{s,H}(W) \geq 8|W| - 5\left(\frac{5|W|-5}3\right) = \frac{25 - |W|}3 > 5.$$
\aftermath
\end{proof}

We omit the work, but case analysis revealed that there are exactly $7$
$4$-Ore graphs with $10$ vertices, and all $7$ are in $\HH$.
Corollary \ref{how to show H is not present}(ii) immediately implies the following.

\begin{remark}
There exists finitely many graphs in $\HH$.
\end{remark}

\section{Proof of the Main Theorem} 
\label{proof-sec}
In Section~\ref{proof-sec}, we start proving the Main Theorem. 
The proofs of parts (A) and (B) rely on many common lemmas, which we prove in
Section~\ref{basic-sec}.  To unify our presentation, we write
\Emph{$\rho_{*,G}$} to denote a statement that holds for both $\rho_{m,G}$ and
$\rho_{s,G}$.  In Section \ref{multi proof sec} we finish proving 
part (A).  In Sections \ref{simple proof sec 1} and
\ref{simple proof sec 2} we finish proving part (B).
To prove each part of our Main Theorem, we assume it is false, and let $G$ be a
counterexample minimizing $|V(G)|+|E(G)|$.  Ultimately, we will reach a
contradiction, by constructing an nb-coloring $I,F$ of $G$.

\subsection{Basic Lemmas} \label{common work}
\label{basic-sec}
In Section~\ref{basic-sec}, we have two goals: (i) to show that $G$ is fairly
``well-behaved'', and (ii) to prove our first gap lemma.  We say a bit more
about each.  To facilitate our proofs, we have allowed precolored vertices, as
well as edge-gadgets.  But we hope that our minimal counterexample $G$ has few,
if any, of these.  It is also easy to check that $\delta(G)\ge 2$.  To get more
control on $G$, we want to show that $G$ has few 2-vertices.  By
``well-behaved'' we mean all of these hoped-for properties.

We will often nb-color some proper subgraph $J$ of $G$, by minimality.  To
get more power in our proof, we would like the option of slightly modifying $J$
before coloring it.  A small modification can only decrease potential by a small
amount.  For example, adding an edge decreases $\rho_{m,G}$ by 2 and decreases
$\rho_{s,G}$ by 5.  So to allow adding an edge, we must show (for each
$W\subsetneq V(G)$), that $\rho_{m,G}(W)\ge -1+2=1$ and $\rho_{s,G}(W)\ge -4+5=1$.
This is the content of Lemma~\ref{small subgraphs part 1}.  We call this a \Emph{gap lemma}, since
it establishes a gap between the actual value of $\rho_{*,G}(W)$ and the lower
bound required by the hypothesis of the Main Theorem.  In later sections, we
prove stronger gap lemmas for both multigraphs and simple graphs, but those
proofs all rely on Lemma~\ref{small subgraphs part 1}.

\begin{lem}
\label{submodular}
The potential function is submodular, i.e., for any graph $J$ all $W_1, W_2
\subseteq V(J)$ satisfy 
$$\rho_{*,J}(W_1 \cap W_2) + \rho_{*,J}(W_1 \cup W_2) \leq \rho_{*,J}(W_1) + \rho_{*,J}(W_2).$$
\end{lem}
\begin{proof}
Each vertex is counted equally many times on both sides of the inequality.  Each edge is
counted at least as often on the left as on the right.
\end{proof}

\begin{lem} \label{easy cases}
$|\VG| \geq 3$, $G$ is connected, and $\delta(G)\ge 2$.
\end{lem}
\begin{proof}
The only graphs with at most two vertices with precolorings that do not extend
to nb-colorings are (i) when $I_p = \VG$ and $G$ contains an edge
and (ii) when $F_p = \VG$ and $G$ contains a multiedge or edge-gadget.
In each case, $\rho_{*,G}(\VG)$ is too small to satisfy the hypothesis of
the Main Theorem.

If $G$ is disconnected, then each component has an nb-coloring by
minimality. Together these give an nb-coloring of $G$.  If $G$ has a
1-vertex $v$, then $G-v$ has an nb-coloring, and we extend it to $G$ by
adding $v$ to $F$.
\end{proof}

Recall that, for each vertex $v$, $d(v)$ denotes the number of edges (including
edge-gadgets) incident to $v$.  Specifically, multiedges contribute 2 to the
degree of each endpoint, but edge-gadgets only contribute 1.
By a \Emph{forbidden subgraph}, we mean $K_4$ or $M$ in the case of
multigraphs, and we mean some graph in the family $\HH$ in the case of simple graphs.

\begin{lemma} \label{remove I}
$I_p = \emptyset$.
\label{Ip-empty-lem}
\end{lemma}
\begin{proof}
Suppose there exists some vertex $w \in I_p$.
By the lower bound on $\rho_{*,G}$, for each edge $vw$ we know $v \notin I_p$.
Let $N = \{v: vw \in \EG\}$.
Let $G' = G - w$, and define a precoloring $I_p',F_p'$ as $ I_p' = I_p - \{w\}$ and $ F_p' = F_p \cup N.$
We claim that $G'$ with the precoloring $I_p',F_p'$ satisfies the hypotheses of
the Main Theorem. 
We did not add any edges, so any subgraph contained in $G'$ is also contained in $G$.
Let $W \subseteq V(G')$, and observe that $\rho_{*,G'}(W) \geq \rho_{*,G}(W \cup \{w\})$.
This proves the claim.
Now by minimality, we can find in polynomial time an nb-coloring
$I',F'$ that extends the precoloring $I_p',F_p'$.
Let $I = I' \cup \{w\}$ and $F = F'$.
\end{proof}

Although we know that $I_p = \emptyset$ in $G$, the notion of $I_p$ will still
be useful.  In particular, we will often use minimality to color a graph $G'$
with a precoloring $I_p',F_p'$ such that $I_p' \neq \emptyset$.

\begin{lemma}\label{min degree 2}
$|N(v)| \geq 2$ for each $v\in\VG$.
\end{lemma}
\begin{proof}
Suppose there exist vertices $v, w$ such that $N(v) = \{w\}$.
By minimality, $G-v$ has an nb-coloring $I',F'$.
If $v \notin F_p$, then we extend $I',F'$ to $G$ by coloring $v$ with the color
unused on its neighbor.  So assume $v\in F_p$.  If $v$ is not incident to a
multiedge or an edge-gadget, then $I', F' \cup \{v\}$ is an nb-coloring of $G$.

Now assume that both $v \in F_p$ and also $vw$ is either a multiedge or an edge-gadget.
If $w \in F_p$, then $\rho_{*,G}(\{v,w\})$ contradicts the hypotheses of the
theorem; so assume $w\notin F_p$.
Let $G' = G-v$, $F_p' = F_p$, and $I_p' = I_p \cup \{w\}$.
We claim that $G'$ with precoloring $I_p', F_p'$ satisfies the
hypotheses of the Main Theorem. 
For if $w \notin W$, then $\rho_{*,G'}(W) = \rho_{*,G}(W)$; and if $w \in W$,
then $\rho_{*,G'}(W) = \rho_{*,G}(W \cup \{v\})$.
So, by minimality, $G'$ has an nb-coloring $I',F'$ that extends the
precoloring $I_p',F_p'$.  Now $I',F'\cup\{v\}$ is an nb-coloring of $G$, a
contradiction.
\end{proof}

\begin{lemma}
\label{uncolored is min degree 3}
If $v$ is uncolored and not incident to an edge-gadget, then $d(v) \geq 3$.
\end{lemma}
\begin{proof}
Suppose, to the contrary, that $v$ is uncolored, $v$ is not incident to an edge-gadget, and
$d(v)=2$.  Since $|N(v)|\ge 2$ by the previous lemma, we denote $\{x,y\}$ by $N(v)$. 
By minimality, $G - v$ has an nb-coloring $I',F'$.
If $\{x,y\}\subseteq F'$, then $I'\cup\{v\},F'$ is an nb-coloring of $G$.
Otherwise
$I',F\cup\{v\}$ is an nb-coloring of $G$.
\end{proof}

Now we can prove our gap lemma.

\begin{lemma}
\label{small subgraphs part 1}
If $\emptyset \neq W \subsetneq \VG$, then $\rho_{*,G}(W) \geq 1$.
\label{gap-lemma-basic}
\end{lemma}
\begin{proof}
Suppose, to the contrary, there exists $W\subsetneq \VG$ such that $|W|\ge 1$
and $\rho_{*,G}(W) \leq 0$.  Among such subsets, choose $W$ to minimize
$\rho_{*,G}(W)$.  Since $I_p=\emptyset$, we must have $|W|\ge 2$. 
Further, if $|W|=2$, then $E(G[W])\ne\emptyset$. 

By minimality, $G[W]$ has an nb-coloring $I_W, F_W$ with $F_p\cap
W\subseteq F_W$.
Let $\overline{W} = \VG \setminus W$.  

\begin{clm}
Each $v\in \overline{W}$ has at most one incident edge (and no edge-gadget)
with endpoint in $W$.  
\end{clm}
\begin{clmproof}
Suppose, to the contrary, that there exists $v\in \overline{W}$ with
two incident edges, or an incident edge-gadget, with endpoints in $W$.  Now
$\rho_{*,G}(W\cup\{v\})<\rho_{*,G}(W)$.  So, by the minimality of $W$, we must
have $W\cup\{v\}=\VG$.  If $v$ has at least three incident edges into $W$, or
an edge and another edge-gadget, then $\rho_{*,G}(W\cup\{v\})$ violates the
hypothesis of the Main Theorem: $\rho_{m,G}(W\cup\{v\})\le
\rho_{m,G}(W)+3-2(3)=-3$ or $\rho_{s,G}(W\cup\{v\})\le \rho_{s,G}(W)+8-5(3)\le
-7$.  
So assume $v$ has exactly two edges into $W$ or exactly one
edge-gadget and no other edges.  Further, $v\in U_p$, since otherwise
$\rho_{*,G}(W\cup\{v\})$ is too small.  By minimality, $G-v$ has
an nb-coloring.  Since $W=V(G)\setminus\{v\}$, we can easily extend this
coloring to $G$, which contradicts that $G$ is a counterexample.  Thus, each
$v\in \overline{W}$ has at most one neighbor in $W$, and no incident
edge-gadget into $W$, as desired.
\end{clmproof}

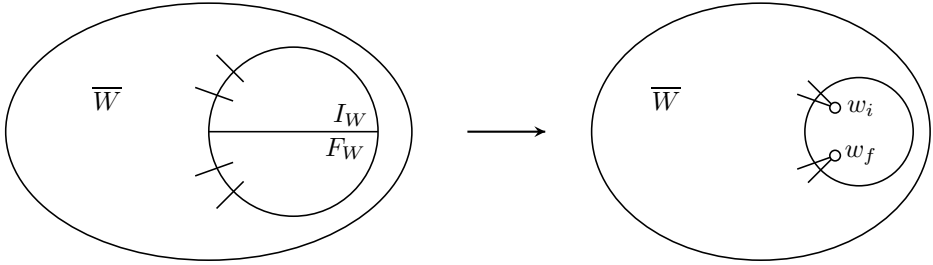
\begin{figure}[!h]
\centering
\begin{tikzpicture} [semithick, scale=.9]
\tikzstyle{lStyle}=[shape = rectangle, minimum size = 5pt, inner sep =
0.5pt, outer sep = 0pt, draw=none]
\tikzset{every node/.style=uStyle}

\begin{scope}[xshift=-.5in]
\draw (0,0) circle (3cm  and 1.9cm)  (1.25,0) circle (1.25cm) (-1.5,.5)
node[lStyle] {$\overline{W}$} 
(2.1,.25) node[lStyle] {$I_W$}
(2.02,-.25) node[lStyle] {$F_W$};
\draw (0,0) --++ (2.5cm,0);
\draw (1.25,0) ++(153:1.0cm) -- ++(-.56,.2); 
\draw (1.25,0) ++(135:1.04cm) -- ++(-.4,.4); 
\draw (1.25,0) ++(207:1.0cm) -- ++(-.56,-.2); 
\draw (1.25,0) ++(225:1.04cm) -- ++(-.4,-.4); 
\end{scope}

\begin{scope}[xshift=2.75in]
\draw (-.1,0) circle (2.5cm  and 1.9cm) (-1.5,.5) node[lStyle] {$\overline{W}$};
\draw (1.35,0) circle (0.8cm);
\draw (1.38,.35) node[lStyle] {$w_i$};
\draw (1.35,0) ++(135:0.5cm) node (wi) {};
\draw (wi) -- ++(-.56,.2); 
\draw (wi) -- ++(-.4,.4); 
\draw (1.38,-.35) node[lStyle] (wf) {$w_f$};
\draw (1.35,0) ++(225:0.5cm) node (wf) {};
\draw (wf) -- ++(-.56,-.2); 
\draw (wf) -- ++(-.4,-.4); 
\end{scope}

\arrowdrawthick (2.55,0) to (3.7,0);

\end{tikzpicture}
\caption{The construction of $G'$ from $G$ in the proof of Lemma~\ref{gap-lemma-basic}.\label{gap-lemma}}
\end{figure}

We construct a graph $G'$ with vertex set $\overline{W} \cup \{w_f, w_i\}$.
We give $G'$ the precoloring $I_p', F_p'$, where $I_p' = \{w_i\}$ and $F_p' =
(F_p\cap\overline{W}) \cup \{w_f\}$.  The edge set of $G'$ is 
$$E(G') = E(G[\overline{W}]) \cup \{uw_i : u x \in \EG, u \in \overline{W}, x
\in I_W\} \cup \{vw_f : v z \in \EG, v \in \overline{W}, z \in F_W\} .$$
If $w_f$ or $w_i$ has degree $0$, then we delete it.
Note that $G'$ is smaller than $G$, since either $|W|\ge 3 > |\{w_i,w_f\}|$ or
else $|W|=2$ and $|E(G'[\{w_i,w_f\}])|=0<|E(G[W])|$.
Because $|N(v)\cap W|\le 1$ for each $v\in \overline{W}$, if $G$ is a simple
graph, then so is $G'$.

Suppose $G'$ has an nb-coloring $I', F'$ that extends the
precoloring $I_p', F'_p$.  It is easy to check that $I' \setminus \{w_i\}
\cup I_W, (F' \setminus \{w_f\}) \cup F_W$ is an nb-coloring
of $G$.  This contradicts that $G$ is a counterexample.
So $G'$ is not near-bipartite.  Recall that $G'$ is smaller than $G$.
So, to reach a contradiction, we will show that $G'$, with precoloring 
$I_p', F_p'$ satisfies the hypotheses of the Main Theorem. 

To begin, we show that $\rho_{m,G'}(W')\ge -1$ and $\rho_{s,G'}(W')\ge -4$ for all
$W'\subseteq V(G')$.  First assume, to the contrary, that there exists $W'$ with
$\rho_{m,G'}(W')\le -2$.  
The key observation is that
\begin{align}
\rho_{m,G}(W'\setminus\{w_i,w_f\}\cup W)\le
\rho_{m,G'}(W') +\rho_{m,G}(W)-\rho_{m,G'}(W'\cap\{w_i,w_f\}).
\label{4.6-ineq}
\end{align}

Although we use \eqref{4.6-ineq} in the form above, it is perhaps easier to understand the
equivalent version: 
\begin{align*}
\rho_{m,G}(W'\setminus\{w_i,w_f\}\cup W)-\rho_{m,G}(W)\le
\rho_{m,G'}(W') -\rho_{m,G'}(W'\cap\{w_i,w_f\}).
\end{align*}
The left side equals $\rho_{m,G'}(W'\setminus \{w_i,w_f\})-2|e(W'\setminus
\{w_i,w_f\},W)|$.
The right side equals $\rho_{m,G}(W'\setminus \{w_i,w_f\})-
2|e(W'\setminus \{w_i,w_f\},\{w_i,w_f\})|$.  The right is no less than the left, since each
edge in $e_{G'}(W'\setminus \{w_i,w_f\},\{w_i,w_f\})$ is the image of at least one
edge in $e_G(W'\setminus \{w_i,w_f\},W)$, and $G[W'\setminus\{w_i,w_f\}]\cong
G'[W'\setminus\{w_i,w_f\}]$.

Now \eqref{4.6-ineq} implies $\rho_{m,G}(W'\setminus\{w_i,w_f\}\cup W)\le -2+0-0=-2$,
which is a contradiction.
Inequality \eqref{4.6-ineq} is the key to proving all of our gap
lemmas.  We use it repeatedly below, often with less detail.
Now assume, to the contrary, that there exists $W'$ with $\rho_{s,G'}(W')\le -5$.
Similar to the previous case, now $\rho_{s,G}(W'\setminus\{w_i,w_f\}\cup W)\le
\rho_{s,G'}(W') +\rho_{s,G}(W)-\rho_{s,G'}(W'\cap\{w_i,w_f\})\le -5+0-0=-5$,
which is a contradiction.

Now we must show that $G'$ does not contain forbidden subgraphs.  In the case of
multigraphs, 
we must show that $G'$ contains neither $K_4$ nor $M_7$.  
%
Suppose instead that $G'$ contains $K_4$ or $M_7$, and let $W'$ denote its vertex set.
Recall from Example~\ref{example1} that (with no precolored vertices)
$\rho_{m,K_4}(V(K_4))=0$ and $\rho_{m,M_7}(V(M_7))=-1$.
So $\rho_{m,G'}(W')-\rho_{m,G'}(W'\cap \{w_i,w_f\})\le 0-3$. 
Now $\rho_{m,G}((W'\setminus\{w_f,w_i\})\cup W)\le
-3+0=-3$, which is a contradiction.
%

Finally, we consider the case of simple graphs.  We must show that $G'$ does not
contain any graph in $\HH$.  Suppose that it does; call this graph $H'$, and
let $W'$ denote its vertex set.  By Corollary~\ref{cor1}(i), we know that (with no
precolored vertices) $\rho_{s,H'}(W')\le 2$.  Thus,
$\rho_{s,G'}(W')-\rho_{s,G'}(W'\cap\{w_i,w_f\})\le 2-8=-6$.  As a result,
$\rho_{m,G'}((W'\setminus\{w_f,w_i\})\cup W)\le -6+0=-6$, which is a
contradiction.
\end{proof}

Lemma \ref{small subgraphs part 1} is useful in many ways.
It immediately implies our next lemma, which is a
strengthening of the submodularity condition in Lemma~\ref{submodular},
and it also implies Lemmas~\ref{fancy edges incident to uncolored} and \ref{put
one vertex in F}.

\begin{lemma} \label{subadditive}
In $G$ the function $\rho$ is subadditive on proper subsets: unless $W_1 = W_2 = \VG$, 
$$\rho_{*,G}(W_1 \cup W_2) \leq \rho_{*,G}(W_1) + \rho_{*,G}(W_2).$$
\end{lemma}
\begin{proof}
Assume that $W_1\ne \VG$ or $W_2\ne \VG$.
Since $\rho_{*,G}(\emptyset)=0$, the previous lemma gives
$\rho_{*,G}(W_1\cap W_2)\ge 0$ for all $W_1,W_2\subseteq \VG$.
So
$\rho_{*,G}(W_1 \cup W_2) 
\le \rho_{*,G}(W_1 \cup W_2) +\rho_{*,G}(W_1\cap W_2)
\le \rho_{*,G}(W_1) + \rho_{*,G}(W_2),$
by Lemma~\ref{submodular}.
\end{proof}


The proof of the following lemma is simple arithmetic, so we omit it.
\begin{lemma} \label{fancy edges incident to uncolored} 
Both endpoints of a multiedge (for part (A)) or an edge-gadget (for part (B))
are uncolored.  Further, when $G$ is a multigraph, at least one endpoint of
each edge is uncolored.  
\end{lemma}

\begin{lemma} \label{put one vertex in F}
If $W \subsetneq \VG$ and $w \in \VG$, then $G[W]$ has an nb-coloring
$I,F$ that extends the precoloring $\emptyset, F_p \cup \{w\}$. 
\end{lemma}
\begin{proof}
Let $F_p' = F_p \cup \{w\}$ and $I_p' = \emptyset$.
Because $G[W]$ is a subgraph of $G$, it contains no forbidden subgraphs.
By Lemma~\ref{small subgraphs part 1}, the graph $G[W]$ with
precoloring $I_p',F_p'$ satisfies the hypotheses of the Main Theorem. 
So by minimality, $G[W]$ has an nb-coloring that extends $I_p',F_p'$.
\end{proof}

\subsection{Multigraphs} 
\label{multi proof sec}
\label{multi-sec}
In this section, we prove part (A) of the Main Theorem.
The key step, which we begin with, is to strengthen by 1 the gap lemma we proved
in Lemma~\ref{small subgraphs part 1} of the previous section.
Everything after this stronger gap lemma is a chain of implications that culminates with
the fact that $G$ cannot exist.


\begin{lem}
If $W\subsetneq \VG$ and $|W|\ge 2$, then $\rho_{m,G}(W)\ge 2$.
\label{small subgraphs part 2}
\label{gap-multi}
\end{lem}
\begin{proof}
The proof is very similar to that of Lemma~\ref{small subgraphs part 1}, 
so we mainly emphasize the differences.
Suppose, the lemma is false; that is, some vertex subset $W$
satisfies $2\le |W|<|\VG|$ and $\rho_{m,G}(W)\le 1$.  
Among such $W$, choose one to minimize $\rho_{m,G}(W)$.
By Lemma~\ref{small subgraphs part 1} we know that $\rho_{m,G}(W)=1$.  
First, we note that $|W|\ge 3$.  Suppose, to the contrary, that $|W|=2$.
By Lemma~\ref{Ip-empty-lem}, $I_p=\emptyset$, so each vertex contributes odd
weight (1 or 3) and each edge contributes even weight ($-2$), which implies 
$\rho_{m,G}(W)\equiv 0 \bmod 2$.  By Lemma~\ref{small subgraphs part 1}, we
have $\rho_{m,G}(W)\ge 1$; thus $\rho_{m,G}(W)\ge 2$.  So, $|W|\ge 3$, as desired.

As in the previous proof, each $v\in\overline{W}$ has at most one neighbor
in $W$.  Since $G$ is connected and $W\subsetneq \VG$,
there exists $w\in W$ with a neighbor not in $W$.  Let $G'=G[W]$ with
precoloring $I'_p, F'_p$, where $I'_p=\emptyset$ and $F'_p = (F_p\cap
W)\cup\{w\}$.  For each $X\subseteq W$, we have $\rho_{m,G'}(X)\ge
\rho_{m,G}(X)-2\ge 1-2=-1$, by Lemma~\ref{gap-lemma-basic}\footnote{This step in
the proof is the only place where we actually use Lemma~\ref{gap-lemma-basic},
and it is why we prove that weaker result before proving this one.}.  
Thus, by minimality, $G'$ has an nb-coloring $I', F'$ that extends the
precoloring $I_p',F_p'$.  Now we repeat the construction of graph $G'$ from
the proof of Lemma~\ref{small subgraphs part 1}.

%

Suppose $G'$ has an nb-coloring $I',F'$ that extends
the precoloring $I'_p, F'_p$.  It is easy to check that $(I' \setminus \{w_i\}
\cup I_W), (F' \setminus \{w_f\} \cup F_W)$ is an nb-coloring of
$G$.  This contradicts that $G$ is a counterexample.
So $G'$ is not near-bipartite.  Recall that $G'$ is smaller than $G$.
So to reach a contradiction, we will show that $G'$, with precoloring
$I_p', F_p'$, satisfies the hypotheses of the Main Theorem. 

We must show that $G'$ does not contain $K_4$ or $M_7$.  Recall that (with all
vertices uncolored), we have $\rho_{m,K_4}(V(K_4))=0$ and
$\rho_{m,M_7}(V(M_7))=-1$.  Suppose, to the contrary, that $G'$ contains a copy
of $H \in \{K_4, M_7\}$, and let $W'$ denote its vertex set.  Since
$H \not\subseteq G$, we have $W'\cap\{w_f,w_i\}\ne \emptyset$.  
As in the proof of Lemma~\ref{small subgraphs part 1},
we have $\rho_{m,G}(W'\setminus\{w_f,w_i\}\cup W) \le \rho_{m,G'}(W') +
\rho_{m,G}(W) - \rho_{m,G'}(W'\cap\{w_f,w_i\})$. 
The subgraph of $G'$ without $W'\cap\{w_f,w_i\}$ is isomorphic to $H$ 
with one or two uncolored vertices removed,
so we have $\rho_{m,G'}(W')-\rho_{m,G'}(W'\cap\{w_f,w_i\})\le 
\rho_{m,H}(\VH)-\rho_{m,K_1}(V(K_1))\le -3$. Thus $\rho_{m,G}(W'\setminus\{w_f,w_i\}\cup W) \le 
-3+1=-2$, which is a contradiction.


Finally, we show that $\rho_{m,G'}(W')\ge -1$ for all $W\subseteq \VG$.
Assume, to the contrary, that there exists $W'$ with $\rho_{m,G'}(W')\le -2$.
Now $\rho_{m,G}(W'\setminus\{w_i,w_f\}\cup W)\le
\rho_{m,G'}(W') +\rho_{m,G}(W)-\rho_{m,G'}(W'\cap\{w_i,w_f\})\le -2 +
\rho_{m,G}(W)<\rho_{m,G}(W)$.  By our choice of $W$, we know that
$W'\setminus\{w_i,w_f\}\cup W=\VG$.  If $w_f\in W'$, then
$\rho_{m,G}(W'\cap\{w_i,w_f\})=1$.  Now $\rho_{m,G}(W'\setminus\{w_i,w_f\}\cup
W)\le -2 +1 -1 = -2$, a contradiction.  So instead, assume $w_f\notin W'$.  However,
now we have $\rho_{m,G}(W'\setminus\{w_i,w_f\}\cup W)< -2 +1 -0=-1$. Here the
inequality is strict, since the left side counts an edge from $w_f$ to a
neighbor outside of $W$, but that edge is not counted on the right (recall from
the second paragraph that $w$ is precolored to be in $F$ and $w$ has a neighbor
in $\overline{W}$).  Again, $\rho_{m,G}(W'\setminus \{w_i,w_f\}\cup W)\le -2$, which
is a contradiction.  So $G'$ satisfies the hypotheses of the Main Theorem, 
which finishes the proof.
\end{proof}

\begin{lemma} \label{min degree 3 multi}
$\delta(G)\ge3$.
\end{lemma}
\begin{proof}
Assume, to the contrary, that $G$ contains a vertex $v$ with $d(v)\le 2$.
By Lemmas~\ref{min degree 2} and~\ref{uncolored is min degree 3} we know 
that $d(v)=2$ and $v\in F_p$.  Let $w$ and $x$ denote the neighbors of
$v$.  Form $G'$ from $G-v$ by adding edge $wx$.  (Note that $wx\notin E(G)$,
since otherwise $\rho_{m,G}(\{v,w,x\})=2(3)+1(1)-3(2)=1$, which contradicts
Lemma~\ref{small subgraphs part 2}.)  
Suppose there exists $W'\subseteq \VGp$ with $\rho_{m,G'}(W')\le -2$.  Since
$G'[W']\not\subseteq G$, we have $\{w,x\}\subseteq W'$.  Now
$\rho_{m,G}(W'\cup\{v\}) = \rho_{m,G'}(W') +(\rho_{m,G}(\{v,w,x\}) -
\rho_{m,G'}(\{w,x\})) \le -2 +(-1)$, which is a contradiction.  
So assume
instead that $G'$ contains a copy of $K_4$ or $M_7$; let $W'$ denote its vertex
set.  In this case $\rho_{m,G}(W'\cup\{v\})=\rho_{m,G'}(W')-1\le 0$. This contradicts
Lemma~\ref{small subgraphs part 2} unless $W'\cup\{v\}=\VG$.  However, in
that case we can easily construct an explicit nb-coloring of $G$ (when
$G'[W']=K_4$ we have only a single case, and when $G'[W']=M_7$ we have
four cases).
\end{proof}

\begin{lemma} \label{no precolor multi}
$F_p=\emptyset$.
\end{lemma}
\begin{proof}
Since $\delta(G)\ge3$, we have $|E(G)|=\frac12\sum_{v\in \VG}d(v)\ge
\frac32|\VG|$.  Now $\rho_{m,G}(\VG)=3|U_p|+|F_p|-2|E(G)|\le 3|\VG| - 2|F_p| -
2(\frac32|\VG|)=-2|F_p|$.  By assumption $\rho_{m,G}(\VG)\ge -1$, so
$F_p=\emptyset$.
\end{proof}

\begin{lemma}
$G$ has at most one vertex $w$ with $d(w)>3$.  If $w$ exists, then $d(w)=4$.
\label{nearly 3-regular}
\end{lemma}
\begin{proof}
Choose arbitrary vertices $v,w\in \VG$.  Since $\delta(G)\ge3$, we have
$2|E(G)|\ge 3(|\VG|-2)+d(v)+d(w)$.  Thus $\rho_{m,G}(\VG)\le
3|\VG|-(3(|V(G)|-2)+d(v)+d(w))=6-d(v)-d(w)$.  Since $\rho_{m,G}(\VG)\ge -1$, we
get $d(v)+d(w)\le 7$.  Since $d(v)\ge 3$, the lemma holds.
\end{proof}

\begin{lemma} \label{no multi edges}
$G$ has no multiedges.
\label{no parallels}
\end{lemma}
\begin{proof}
Suppose, to the contrary, that $G$ has a multiedge.  By the previous lemma,
one of its endpoints has degree 3.  So let $v$ be a 3-vertex with neighborhood
$\{w,x\}$, and with a multiedge to $x$.  
By Lemma \ref{put one vertex in F} there exists an nb-coloring of $G-v$ with $w \in F$.
This is a contradiction, as such a coloring can be extended to $G$ by coloring $v$ with the color not on $x$.
\end{proof}

\begin{lemma}
$|\VG|\ge 6$.
\end{lemma}
\begin{proof}
Since $\delta(G)=3$, and $G$ has no multiedge, $|\VG|\ge 4$.  If $|V(G)|=4$,
then $G$ is $K_4$, which is a contradiction.  So suppose that $|\VG|=5$.
By Lemma~\ref{nearly 3-regular}, $G$ has four 3-vertices and a 4-vertex.  Thus,
$G$ is formed from $K_5$ by deleting two independent edges.  
So let $I$ consist of two non-adjacent vertices, and let $F=\VG\setminus I$.
This nb-coloring of $G$ is a contradiction.
Thus, $|\VG|\ge 6$, as desired.
\end{proof}

\begin{lemma}
$G$ has no 3-cycle.
\label{no3-cycle-lem}
\end{lemma}
\begin{proof}
First suppose that $G$ contains 3-cycles $vwx$ and $wxy$.  Form $G'$ from
$G-\{w,x\}$ by identifying $v$ and $y$; call this new vertex $z$.  If there
exists $W'\subseteq \VGp$ with $\rho_{m,G'}(W')\le -2$, then clearly $z\in
W'$.  So $\rho_{m,G}((W'\setminus\{z\})\cup\{v,w,x,y\})\le -2 + 3(3)-5(2)=-3$, a
contradiction.  Suppose instead that $G'$ contains a copy of $M_7$, and let $W'$
denote its vertex set.  Similar to before, $\rho_{m,G}((W'\setminus\{z\}) \cup
\{v,w,x,y\})\le \rho_{m,M_7}(V(M_7))+3(3)-5(2)=-2$, a contradiction.  Finally,
suppose that $G'$ contains a copy of $K_4$, and let $W'$ denote its vertex set.
Now $\rho_{m,G}((W'\setminus \{z\})\cup\{v,w,x,y\})\le 0 + 3(3) -5(2)=-1$. 
This contradicts Lemma~\ref{gap-multi}, unless $\VG=(W'\setminus
\{z\})\cup\{v,w,x,y\}$.  However, in that case, we can easily check that
$G=M_7$, a contradiction.  Since $G'$ is smaller than $G$ and satisfies the
hypotheses of the Main Theorem, $G'$ has an nb-coloring, $I',F'$.  
And we easily extend $I',F'$ to $G$, which is a contradiction.

Now suppose that $G$ contains a 3-cycle $vwx$ and none of its edges lie on
another 3-cycle.  Assume, without loss of generality that $d(w)=d(x)=3$.
Let $y$ denote the third neighbor of $x$.  Since $w$ and $x$ have distinct
neighbors off the 3-cycle, we can also assume that $d(y)=3$.  
Form $G'$ from $G-\{v,x\}$ by identifying
$w$ and $y$; call this new neighbor $z$.  If there exists $W'\subseteq \VGp$
with $\rho_{m,G'}(W')\le -2$, then also $\rho_{m,G}(W'\setminus\{z\})\cup
\{w,x,y\})\le -2+2(3)-2(2)=0$, which contradicts Lemma~\ref{gap-multi}, since
$(W'\setminus\{z\})\cup \{w,x,y\}\subseteq \VG \setminus \{v\}$.  Note that $G'$ cannot contain
$K_4$, since $G$ does not contain two 3-cycles with a common edge.
Suppose instead that $G'$ contains $M_7$.  Recall that $M_7$ contains two
edge-disjoint copies of two 3-cycles sharing an edge.  Since $G$ contains no
such subgraph, both copies must contain the new vertex $z$.  
But this is impossible: since $d(w)=d(y)=3$, also $d_{G'}(z)=3$.
\end{proof}

\begin{lemma} \label{finish of multi}
$G$ does not exist.  That is, part (A) of the Main Theorem is true.
\end{lemma}
\begin{proof}
Choose a vertex $v\in \VG$ with $d(v)=3$.  Let $\{w,x,y\}=N(v)$. Form $G'$ from
$G-v$ by identifying $w$ and $x$; call this new vertex $z$.  By
Lemma~\ref{no3-cycle-lem} $G$ has no 3-cycle, so $G'$ cannot contain $K_4$ or
$M_7$, since neither has a single vertex contained in all of its 3-cycles.  For
each $W'\subseteq \VGp$, Lemma~\ref{gap-multi} implies $\rho_{m,G'}(W')\ge
\rho_{m,G}((W'\setminus\{z\})\cup\{w,x\})-3\ge 2-3=-1$.  Thus, by minimality,
$G'$ has an nb-coloring $I',F'$.  And it is easy to extend this to $G$. 
Specifically, remove $z$ from whichever set contains it and add $w$ and $x$ to
this set.  Now, if both $y$ and $z$ were in $F'$, then add $v$ to $I'$;
otherwise add $v$ to $F'$.
\end{proof}

\subsection{Simple Graphs: More Reducible Configurations}
\label{simple proof sec 1}
\label{simple}

In this section we continue the proof of part (B) of the Main Theorem, which we
began in Section~\ref{basic-sec}.  Our approach mirrors what we did in
Section~\ref{multi-sec}, where we showed (for part (A)) that a minimal
counterexample must be well-behaved.  The main results of the section are that
$\delta(G)\ge 3$ and that the subgraph induced by uncolored 3-vertices is a forest.
To prove these properties, a key step is strengthening our earlier gap lemma,
which we do in Lemma~\ref{gap-simple}.  In Section~\ref{color-simple-sec} we will complete
the proof of part (B).  Using the structural results that we prove here, there
we will give a discharging argument to show that $G$ is very nearly comprised
entirely of uncolored 3-vertices that induce a forest, together with uncolored
4-vertices that induce an independent set.  (In fact $G$ can vary slightly from
this, but in each case we explicitly construct an nb-coloring.)

We will frequently use our next lemma to extend an nb-coloring from a subgraph
of $G$ to all of $G$.

\begin{lem} 
\label{color-k-cycle}
Suppose $C=x_1\ldots x_k$ is an induced cycle in $G$ with $d(x_i)=3$ for all $i$, and let
$\{z_i\}=N(x_i)\setminus\{x_{i-1},x_{i+1}\}$ for all $i$.  Fix an nb-coloring
$I',F'$ of $G-V(C)$.  We can extend $I',F'$ to $G$ unless (i) $z_i\in I'$ for
all $i$ or (ii) $k$ is odd and $z_i\in F'$ for all $i$ and all $z_i$ are in the
same component of $G[F']$.
\end{lem}
\begin{proof}
Fix an nb-coloring $I',F'$ of $G-V(C)$.
First suppose that there exist $z_i\in I'$ and $z_j\in F'$.  By symmetry,
assume that $z_k\in I'$ and $z_1\in F'$.  We iteratively add each $x_i$ to
either $I'$ or $F'$.  Let $I_1=I'\cup\{x_1\}$ and $F_1=F'$.  For each $j>1$, do
the following.  If $I_{j-1}\cap\{x_{j-1},z_j\}=\emptyset$, then
$I_j=I_{j-1}\cup\{x_j\}$ and $F_j=F_{j-1}$; otherwise $I_j=I_{j-1}$ and
$F_j=F_{j-1}\cup\{x_j\}$.  It is easy to prove by induction on $j$ that
$I_k,F_k$ is an nb-coloring of $G$.

Now instead assume that $z_i\in F'$ for all $i\in [k]$.  If $k$ is even, then let
$I=I'\cup\bigcup_{i=1}^{k/2}x_{2i-1}$ and $F=F'\cup\bigcup_{i=1}^{k/2}x_{2i}$.
Now $I,F$ is an nb-coloring of $G$.  So assume $k$ is odd.  Suppose, by
symmetry, that $z_{k-1}$ and $z_k$ are in different components of $G[F']$.
Let $I=I'\cup\bigcup_{i=1}^{(k-1)/2}x_{2i-1}$ and $F=F'\cup\{x_k\}\cup
\bigcup_{i=1}^{(k-1)/2}x_{2i}$. Again $I,F$ is an nb-coloring of $G$.
\end{proof}

Our next construction is motivated by our desire to avoid the exceptional cases
in the previous lemma.  Clearly, this is achieved by every nb-coloring of
$G(C,z_1,z_2)$, which we define next.  Ultimately, we will use this
construction and lemma after it to show that the uncolored 3-vertices of $G$
induce a forest.  But the proof that $G(C,z_1,z_2)$ has an nb-coloring is
tricky, and we will break it into Lemmas~\ref{no-short-cycle-of-3s},
\ref{no-5cycle-of-3s}, and \ref{no cycle of 3s}.

\begin{defn}
\label{GCz1z2-defn}
Let $C=x_1\ldots x_k$ be a $k$-cycle in $G$ induced by 3-vertices, and let
$\{z_i\}=N(x_i)\setminus\{x_{i-1},x_{i+1}\}$ for all $i$.  Let $W=\VG-C$.
We construct an auxiliary graph \EmphE{$G(C,z_1,z_2)$}{-4mm} as follows\footnote{Part
(ii) of Definition~\ref{GCz1z2-defn} is the most important place where we
construct edge-gadgets.  A key consequence of using an edge-gadget is that
$G(C,z_1,z_2)$ is smaller than $G$, which is essential for the proof of
Lemma~\ref{no-short-cycle-of-3s}.}: 
\begin{enumerate}
	\item[(i)] if $z_1$ and $z_{2}$ are the endpoints of an edge-gadget, then $G(C,z_1,z_2) = G[W]$; otherwise 
	\item[(ii)] if $z_1z_2 \in \EG$, then $G(C,z_1,z_2)$ is formed from
$G[W]$ by removing $z_1z_2$ and replacing it with an edge-gadget; otherwise
	\item[(iii)] $G(C,z_1,z_2)$ is formed from $G[W]$ by adding edge $z_1z_{2}$.
\end{enumerate}
\end{defn}

\begin{figure}[!h]
\centering
\begin{tikzpicture} [semithick, scale=.9]
\tikzset{every node/.style=uStyle}

\begin{scope}[xshift=4.2in,xscale=1.1]
\draw (0,0) node {} (.9,0) node {} (2.3,0) node {} -- (3.2,0) node {};
\arrowdrawthick (1.35,0) -- (1.85,0);
\end{scope}

\begin{scope}[xshift=2.1in,xscale=1.1]
\draw (0,0) node {} -- (.9,0) node {} (2.3,0) node (z1) {} (3.2,0) node (z2) {};
\path[draw=black,fill=gray!60] (z1.center) to[bend left=30] (z2.center) (z2.center)
to[bend left=30] (z1.center);
\arrowdrawthick (1.35,0) -- (1.85,0);
\draw (z1) node {} (z2) node {};
\end{scope}

\begin{scope}[xscale=1.1]
\draw (0,0) node (z1) {} (.9,0) node (z2) {} (2.3,0) node (z3) {} (3.2,0) node (z4) {};
\path[draw=black,fill=gray!60] (z1.center) to[bend left=30] (z2.center) (z2.center)
to[bend left=30] (z1.center);
\path[draw=black,fill=gray!60] (z3.center) to[bend left=30] (z4.center) (z4.center)
to[bend left=30] (z3.center);
\arrowdrawthick (1.35,0) -- (1.85,0);
\draw (z1) node {} (z2) node {} (z3) node {} (z4) node {};
\end{scope}

\end{tikzpicture}
\caption{Vertices $z_1$ and $z_2$ in the construction of
$G(C,z_1,z_2)$ in Definition~\ref{GCz1z2-defn}.\label{GCz1z2-fig}}
\end{figure}
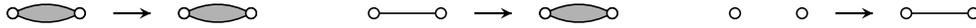

To find an nb-coloring of $G(C,z_1,z_2)$ by minimality, we must show that
$G(C,z_1,z_2)\notin \HH$.  Our next lemma helps us do this.

\begin{lem} 
\label{links are special}
We use notation from Definition \ref{linked def}.
If vertices $v$ and $w$ are linked in $G$ through subgraph $H'$ with $|\VHp|
< |\VG|$, then $v$ and $w$ are specially-linked (and the linking graph is in
$\HH$).
\end{lem}
\begin{proof}
Suppose, to the contrary, that $v$ and $w$ are linked via subgraph $H'$, where
$H' + vw\cong H$ for some $H \notin \HH$.  Now $|\EHp|\le |\EG|-2$ (since
$\delta(G)\ge 2$), so $|\EH|\le |\EG|-1$.  Thus, $|\VH|+|\EH|<|\VG|+|\EG|$,
which implies that $H$ is smaller than $G$ in our ordering.  
By the definition of linked, $H$ is not near-bipartite.
So by the minimality of $G$, either $H$ contains as a subgraph some graph in
$\HH$ or else there exists $W\subseteq V(H)$ such that $\rho_{s,H}(W)\le -5$.
In the latter case, $\rho_{s,G}(W) \leq \rho_{s,H}(W) + 5 \leq 0$, which
contradicts our gap lemma, Lemma~\ref{small subgraphs part 1}.
So $H$ contains as a subgraph a graph from $\HH$.  Since $\HH\subseteq \HH'$,
vertices $v$ and $w$ are specially-linked, as desired.
\end{proof}

\begin{lem}\label{no cycle of 3s of length at most 4}
\label{no-short-cycle-of-3s}
If $C = x_1\cdots x_k$ is a cycle in $G$ induced by 3-vertices,
then at least one of the following holds: 
\begin{enumerate}
\item[(i)] $k \geq 5$, or
\item[(ii)] some $x_i$ is incident to an edge-gadget, or 
\item[(iii)] $V(C) \cap F_p \neq \emptyset$.
\end{enumerate}
\end{lem}
\begin{proof}
Suppose, to the contrary, that $k \in \{3,4\}$, no $x_i$ is incident to an edge-gadget, and each $x_i$ is uncolored.
Let $W = \VG \setminus V(C)$, and let $\{z_i\} = N(x_i)\setminus\{x_{i-1},x_{i+1}\}$ for all $i$.

First suppose $k=4$.  Let $G'=G[W]$ with $I_p'=I_p$ and
$F_p'=F_p\cup\{z_1\}$.  By Lemma~\ref{put one vertex in F}, $G'$ has an
nb-coloring $I',F'$ that extends $I_p',F_p'$.  By Lemma~\ref{color-k-cycle}, we
can extend $I',F'$ to an nb-coloring of $G$.

Now assume $k=3$.
By Definition~\ref{defn H'} and Remark~\ref{two zi}, there exists $j$ such that $z_j
\neq z_{j+1}$ and $z_j, z_{j+1}$ are not specially-linked through a subgraph of $G[W]$.
By Lemma~\ref{links are special}, vertices $z_j$ and $z_{j+1}$ are not linked through any subgraph of $G[W]$.
By symmetry, we assume $j=1$.  Let $G'=G(C,z_1,z_2)$.  We first show that we can
extend any nb-coloring $I',F'$ of $G'$ to $G$, by Lemma~\ref{color-k-cycle}.
By construction of $G'$, at least one of $z_1$ and $z_2$ is in $F'$.
Assume $z_1,z_2,z_3\in F'$. Note that $z_1$ and $z_2$ must be in
different components of $G[F']$, even if they are in the same component of $G'[F']$.
By Lemma~\ref{color-k-cycle}, we can extend $I',F'$ to $G$, as desired. 

Now we must show that $G'$ does indeed have the desired nb-coloring $I',F'$.
By construction, $G'$ is smaller than $G$, and $G'$ contains no forbidden
subgraph, since $z_1$ and $z_2$ are not linked in $G$.
By the minimality of $G$, if $\rho_{s,G'}(U) > -5$ for all $U \subseteq W$, then $G'$ is near-bipartite.
If $\{z_1,z_2\} \not\subseteq U$, then $\rho_{s,G'}(U) = \rho_{s,G}(U) > -5$.
If $\{z_1,z_2\} \subseteq U$, then 
\begin{equation}\label{remove edge add C eqn}
\rho_{s,G'}(U) \geq \rho_{s,G}(U \cup C) - 8(3) + 5(5) - 6 = \rho_{s,G}(U\cup C) - 5.
\end{equation}
If there exists $U$ such that $\rho_{s,G'}(U) \leq -5$, then $\rho_{s,G}(U \cup C) \leq 0$.
By Lemma \ref{small subgraphs part 1}, this implies that $U \cup V(C) = \VG$. 
So $z_3 \in U$, and we can add the edge $x_3z_3$ to the calculation in
(\ref{remove edge add C eqn}); the new bound claims $\rho_{s,G}(U \cup C) \leq -5$, which is a contradiction.
Thus, $G'$ has an nb-coloring $I',F'$.
\end{proof}

A key intermediate result in this section is our improved gap lemma, Lemma~\ref{small subgraphs simple}.
Our next result is designed to help us prove this gap lemma.

\begin{lem}\label{large potential when missing just two vertices}
If $W \subset \VG$ such that $|W| = |\VG| - 2$, then $\rho_{s,G}(W) \geq 5$.
\end{lem}
\begin{proof}
Assume, to the contrary, that there exists $W$ satisfying the hypotheses with $\rho_{s,G}(W) \leq 4$.
Let $\{v,w\} = \VG \setminus W$.
If $v$ and $w$ are both uncolored and not incident to edge-gadgets, then they
each have degree at least 3, by Lemma~\ref{uncolored is min degree 3}; and so
together they are incident to at least $2(3)-1=5$ edges (with equality if
$d(v)=d(w)=3$ and $v$ and $w$ are adjacent).  Now $\rho_{s,G}(\VG) \le
\rho_{s,G}(W)+ 2(8) - 5(5) \le 4 -9=-5$, which contradicts the hypothesis of the
Main Theorem.  

Now we assume instead that at least one of $v$ and $w$ is either precolored or incident
to an edge-gadget.  Recall from Lemma~\ref{fancy edges incident to uncolored} 
that both endpoints of each edge-gadget are uncolored.
Each precolored vertex has potential 5 less than each uncolored
vertex, and is still incident to at least 2 edges, by Lemma~\ref{min degree 2};
so the analysis remains the same.  Thus, we assume that $v$ and $w$ are both uncolored.
Suppose that at least one of $v$ and $w$ is
incident to an edge-gadget, but $vw$ is not an edge-gadget itself. 
If $k$ denotes the total number of edge-gadgets incident to $v$ and $w$, then
$v$ and $w$ are also incident to at least $5-2k$ more edges.
Since each edge-gadget decreases potential more than 2 edges do, the analysis remains the same.
Finally, assume that $vw$ is an edge-gadget and $v$ and $w$ are each incident
to only one other edge.  (If at least one of $v$ and $w$ has degree 3, then
together they have one incident edge-gadget, and at least three more incident
edges, so the analysis is similar to before.) Let $x$ denote the neighbor of $w$
other than $v$.  By Lemma~\ref{put one vertex in F}, we can nb-color $G[W]$ with
precoloring $I_p'=\emptyset$ and $F_p'=F_p\cup\{x\}$.
To extend this coloring to $G$, add $v$ to $F$ and $w$ to $I$.
%
\end{proof}

Now we can prove our stronger gap lemma.

\begin{lem}
\label{small subgraphs simple}
\label{gap-simple}
If $W \subset \VG$ such that $0 < |W| \leq |\VG|-2$, then $\rho_{s,G}(W) \geq 4$.
\end{lem}
\begin{proof}
Suppose, to the contrary, that some $W$ satisfies the hypotheses and has $\rho_{s,G}(W) \leq 3$.
We assume further that $W$ minimizes $\rho_{s,G}(W)$ among all such vertex subsets.
Let $\overline{W} = \VG \setminus W$.

We first show that if $v \in \overline{W}$, then $|N(v) \cap W| \leq 1$.
Suppose, to the contrary, that $|N(v) \cap W| \geq 2$, which gives that
$\rho_{s,G}(W \cup \{v\}) \leq \rho_{s,G}(W) + 8 - 10$.
Lemma~\ref{large potential when missing just two vertices} implies that $|W| <
|\VG|-2$, so $|W \cup \{v\}| \leq |\VG|-2$, which contradicts the minimality of $W$.
Thus $|N(v)\cap W|\leq 1$, as desired.

By minimality, $G[W]$ has an nb-coloring $I_W,F_W$.
We construct a graph $G'$ with vertex set $\overline{W} \cup \{w_f, w_i\}$,
similar to the proof of our first gap lemma, Lemma~\ref{gap-lemma-basic}.
We give $G'$ the precoloring $I_p',F_p'$, where $I_p' = \{w_i\}$ and $F_p' =
(F_p \cap \overline{W}) \cup \{w_f\}$.  The edge set of $G'$ is given by 
$$E(G') = e(\overline{W}) \cup \{uw_i : u x \in \EG, u \in \overline{W}, x \in
I_W\} \cup \{vw_f : v z \in E_G, v \in \overline{W}, z \in F_W\}.$$
If $w_f$ or $w_i$ has degree $0$, then we delete it.  Using Lemma~\ref{put one
vertex in F}, we will assume that $w_f$ is not deleted.  Recall that $|N(v)
\cap W| \leq 1$ for each $v \in \overline{W}$, so $G'$ is a simple graph.

If $G'$ has an nb-coloring $I', F'$, then we delete $\{w_i,w_f\}$
and use the nb-coloring $I_W,F_W$ on $G[W]$ to get an nb-coloring of $G$.
This contradicts that $G$ is a counterexample, so $G'$ must not satisfy the
hypotheses of the Main Theorem. 
Thus, $G'$ contains either a forbidden subgraph or else a vertex set $U'$ such
that $\rho_{s,G'}(U') \leq -5$.
We start with the latter case.
Pick $U' \subseteq V(G')$ to minimize $\rho_{s,G'}(U')$.  Now 
\begin{eqnarray*}
\rho_{s,G'}(U')  & \geq & \rho_{s,G}((U'\setminus\{w_f,w_i\})\cup W) -
\rho_{s,G}(W) + \rho_{s,G'}(U' \cap \{w_f,w_i\}) \\ 
                & \geq & \rho_{s,G}((U'\setminus\{w_f,w_i\})\cup W) - 3 +
\rho_{s,G'}(U' \cap \{w_f,w_i\}) .
\end{eqnarray*} 
The explanation of the above inequality is identical to that of
\eqref{4.6-ineq} in the proof of Lemma~\ref{small subgraphs part 1}.
Trivially, $\rho_{s,G'}(U' \cap \{w_f,w_i\}) \geq 0$.

If $\overline{W} \not\subseteq U'$, then $(U'\setminus\{w_f,w_i\}) \cup W \neq \VG$,
so Lemma \ref{small subgraphs part 1} implies $\rho_{s,G}((U'\setminus\{w_f,w_i\}) \cup W) \geq 1$.
Thus, $\rho_{s,G'}(U') \ge 1-3+0=-2$.
If $\overline{W} \subseteq U'$, then the minimality of $\rho_{s,G'}(U')$ implies
that $w_f \in U'$, so $\rho_{s,G'}(U' \cap \{w_f,w_i\}) = 3$.
By hypothesis $\rho_{s,G}((U'\setminus\{w_f,w_i\}) \cup W) \geq -4$.
So $\rho_{s,G'}(U') \geq -4 - 3 +3 =-4$.

Now assume that $G'$ contains a subgraph $H' \in \HH$.
Because $G$ is a minimal counterexample, $G$ is nb-critical, 
so $H' \not\subset G$, which implies $\VHp \cap \{w_i,w_f\} \neq \emptyset$.
Recall that graphs in $\HH$ have only uncolored vertices, so the potential of $H'$ minus $\VHp \cap \{w_f,w_i\}$
can be calculated as a graph in $\HH$ minus one or two uncolored vertices, even though what we have done
is remove a precolored vertex from a subgraph of $G'$.
Moreover, if any other vertex in $H'$ is precolored, it contributes less to the potential than
an uncolored vertex, so 
\begin{eqnarray}
 \rho_{s,G}((\VHp \setminus \{w_f,w_i\}) \cup W) & \leq & (\rho_{s,H'}(\VHp) - 8) + \rho_{s,G}(W) \notag \\
		& \leq & \rho_{s,H'}(\VHp)  -5. \label{ineq1} 
\end{eqnarray}

\textbf{Case 1: $\bbs{H'\notin \{K_4, M_7\}}$}.
By Corollary~\ref{how to show H is not present}(iii), we know that $\rho_{s,H}(\VHp) \leq 0$.
This implies that $\rho_{s,G}((\VHp\setminus\{w_i,w_f\})\cup W) \leq -5$, which
contradicts that $G$ is a counterexample.
\smallskip

For Cases 2 and 3, we will use the following fact.  Let
$U=(\VHp\setminus\{w_f,w_i\})\cup W$.
By Corollary~\ref{how to show H is not present}(i), $\rho_{s,H}(\VHp) \leq 2$,
so inequality~\eqref{ineq1} gives $\rho_{s,G}(U) \leq -3$.  Now
Lemma~\ref{gap-lemma-basic} implies that $U = \VG$.
\smallskip

\textbf{Case 2: $\bbs{H'=M_7}$}.
If $\VHp \supset \{w_f , w_i\}$, then inequality~\eqref{ineq1} improves to
$\rho_{s,G}(U) \leq \rho_{s,H'}(\VHp)-13$. 
So $\rho_{s,G}(U) \le 2-13 = -11$, a contradiction.
Instead assume that $|\VHp \cap \{w_f,w_i\}| = 1$.
For ease of notation, let $\{w_*\} = \VHp \cap \{w_f,w_i\}$.
Note that each vertex in $M_7$ is in a copy of $K_4 - e$.
Let $x,y,z$ be vertices in $H'$ such that $H[\{x,y,z,w_*\}]$ is $K_4 - e$.
By construction, $\{x,y,z\} \subseteq \overline{W}$.
So $\rho_{s,G}(W\cup\{x,y,z\}) = \rho_{s,G}(W) + 8(3) - 5(5) < \rho_{s,G}(W)$. 
Since $0<|W\cup\{x,y,z\}|\le |\VG|-3$, 
this contradicts the minimality of $\rho_{s,G}(W)$.

\textbf{Case 3: $\bbs{H'=K_4}$}.
Because $w_f$ and $w_i$ are not adjacent (if they both exist), $|\VHp \cap \{w_i,w_f\}| = 1$.
So $G[\overline{W}]=K_3$ and each vertex of $\overline{W}$ has one edge into $W$.
By Lemma \ref{no cycle of 3s of length at most 4}, either $\overline{W}$ contains a
precolored vertex or else is incident to an edge-gadget.  In each case, 
the above inequality $\rho_{s,G}(U) \leq \rho_{s,H'}(\VHp)  -5$ improves to
$\rho_{s,G}(U) \leq \rho_{s,K_4}(V_{K_4}) - 10 \leq -8$, which contradicts that
$G$ is a counterexample.
\end{proof}

The previous lemma gives the following three easy corollaries.  The first is
analogous to Lemma~\ref{put one vertex in F}, but now we can add a vertex to $I_p$.
The third slightly extends Lemma~\ref{no cycle of 3s of length at most 4}.

\begin{lem} \label{put one vertex in I}
If $W \subset \VG$ such that $|W| \leq |\VG|-2$ and $w \in W$, then $G[W]$ has
an nb-coloring that extends the precoloring $I_p \cup \{w\}, F_p$.
\end{lem}
\begin{proof}
Let $G'=G[W]$ with precoloring $I_p \cup \{w\}, F_p$.  Each $U\subseteq W$
satisfies $\rho_{G',s}(U)\ge \rho_{G,s}(U)-8\ge 4-8=-4$, so $G'$ has the desired
coloring by the Main Theorem.
\end{proof}

\begin{lem}\label{edge gadgets form matching}
Each vertex in $G$ is incident to at most one edge-gadget.
\end{lem}
\begin{proof}
If, to the contrary, some $v$ is incident to edge-gadgets with endpoints
$w$ and $x$, then $\rho_{s,G}(\{v,w,x\})\le 8(3)-11(2)=2$, which contradicts
Lemma~\ref{gap-simple}.  (A short case analysis shows that $|\VG|\ge 5$.)
\end{proof}

\begin{lem}\label{no cycle of 3s of length at most 5}
\label{no-5cycle-of-3s}
If $C = x_1\cdots x_k$ is a cycle in $G$ induced by 3-vertices,
then at least one of the following holds: 
\begin{enumerate}
\item[(i)] $k \geq 6$, or
\item[(ii)] some $x_i$ is incident to an edge-gadget, or 
\item[(iii)] $V(C) \cap F_p \neq \emptyset$.
\end{enumerate}
\end{lem}
\begin{proof}
The proof is nearly identical to the case $k=3$ in the proof of
Lemma~\ref{no-short-cycle-of-3s}.  Let $\{z_i\} =
N(x_i)\setminus\{x_{i-1},x_{i+1}\}$ for all $i$.  
By Remark~\ref{two zi} and symmetry, assume $z_1\ne z_2$.
If we let $G'=G(C,z_1,z_2)$, then the only difference is in proving that $G'$
has an nb-coloring.  For each $U\subseteq \VGp$ with $|U|\ge 2$,
Lemma~\ref{gap-simple} gives $\rho_{s,G'}(U)\ge \rho_{s,G}(U)-6\ge 4-6=-2$.  So
$G'$ has an nb-coloring by the Main Theorem.
\end{proof}

We now prove that $\delta(G)\ge 3$, which will be helpful for our discharging
argument in the next subsection.

\begin{lem}\label{min degree 3 simple}
$\delta(G) \geq 3$.
\end{lem}
\begin{proof}
Suppose, to the contrary, that some $v \in \VG$ has $d(v) \leq 2$.
Lemma~\ref{min degree 2} implies that $d(v) = 2$.
Lemma~\ref{uncolored is min degree 3} shows that either $v \in F_p$ or $v$ is incident to an
edge-gadget, and Lemma~\ref{fancy edges incident to uncolored} implies that $v$
cannot satisfy both.  Let $N(v) = \{w_1, w_2\}$.

\textbf{Case 1: $\bbs{v\in U_p}$ and $\bbs{vw_1}$ is an edge-gadget.}
By Lemma~\ref{edge gadgets form matching}, $vw_2$ is an edge and not an edge-gadget.
Let $G'=G-v$.  By Lemma~\ref{put one vertex in F}, $G'$ has an nb-coloring $I',F'$
with $w_2 \in F'$.  To extend $I',F'$ to $G$, we color $v$ with the color
unused on $w_1$.  This contradicts that $G$ is a counterexample.

So now assume that $v \in F_p$, and both $vw_1, vw_2$ are edges and not edge-gadgets.
Note that $w_1$ and $w_2$ are both uncolored, since otherwise
$\rho_{s,G}(\{v,w_i\})= 3(2)-5=1$, which contradicts Lemma~\ref{gap-simple}.

\begin{figure}[!h]
\centering
\begin{tikzpicture} [semithick, scale=.9]
\tikzset{every node/.style=uStyle}

\begin{scope}[xscale=1.1,xshift=-.375in]
\draw (-.10,0) node (w1) {} (.9,0) node (w2) {} --
(.4,-.35) node (v) {} (2.3,0) node {}  (3.10,0) node {} (3.69,.05) node[lStyle] {$\in F'$};
\path[draw=black,fill=gray!60] (v.center) to[bend left=30] (w1.center) (w1.center)
to[bend left=30] (v.center);
\draw (v) node {} (w1) node {};
\arrowdrawthick (1.35,0) -- (1.85,0);
\end{scope}

\begin{scope}[xshift=2.15in,xscale=1.1]
\draw (0,0) node (z0) {} -- (.45,-.35) node {} -- (.9,0) node {} -- (z0) (2.3,0) node (z1) {} (3.2,0) node (z2) {};
\path[draw=black,fill=gray!60] (z1.center) to[bend left=30] (z2.center) (z2.center)
to[bend left=30] (z1.center);
\arrowdrawthick (1.35,0) -- (1.85,0);
\draw (z1) node {} (z2) node {};
\end{scope}

\begin{scope}[xshift=4.25in,xscale=1.1]
\draw (0,0) node {} -- (.45,-.35) node {} -- (.9,0) node {} (2.3,0) node (z1) {} -- (3.2,0) node (z2) {};
\arrowdrawthick (1.35,0) -- (1.85,0);
\draw (z1) node {} (z2) node {};
\end{scope}

\end{tikzpicture}
\caption{Constructing $G'$ from $G$ for Cases 1, 2, and 3 in the proof of
Lemma~\ref{min degree 3 simple}.\label{lemma min degree 3 simple fig}}
\end{figure}
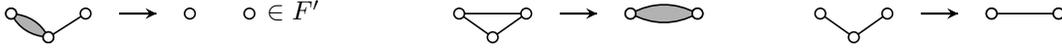

\textbf{Case 2: $\bbs{w_1\in N(w_2)}$.}
We form $G'$ from $G-v$ by 
replacing $w_1w_2$ with an edge-gadget (if it is not already an edge-gadget);
This is analogous to our earlier construction of $G(C,z_1,z_2)$.  To extend any
nb-coloring $I',F'$ of $G'$ to $G$, we simply add $v$ to $F'$.
Because $G'$ is smaller than $G$, by minimality $G'$ must contain a forbidden
subgraph or a vertex set $W'$ such that $\rho_{s,G'}(W') \leq -5$.

By hypothesis, $G$ contains no forbidden subgraph, and by construction graphs in
$\HH$ have no edge-gadgets.  So $G'$ contains no forbidden subgraph.
To reach a contradiction, we show that $\rho_{s,G'}(W')\ge -4$ for all
$W'\subseteq \VGp$.  If $\{w_1, w_2\} \not\subseteq W'$,
then $\rho_{s,G'}(W') = \rho_{s,G}(W')$; and if $\{w_1,w_2\} \subseteq W'$, then 
$$\rho_{s,G'}(W') \geq \rho_{s,G}(W' \cup \{v\}) - 6 - 3 + 5(2) \geq \rho_{s,G}(W'
\cup \{v\}) +1 \geq -3.$$

\textbf{Case 3: $\bbs{w_1 \notin N(w_2)}$.}
We form $G'$ from $G-v$ by adding edge $w_1w_2$.
If $G'$ has an nb-coloring $I',F'$, then we can extend it to $G$ by
adding $v$ to $F'$.  So we assume $G'$ has no nb-coloring.
By construction, $G'$ is smaller than $G$.  So by minimality $G'$ has a
forbidden subgraph or contains a vertex subset $W'$ such that $\rho_{s,G'}(W')
\leq -5$.  Similar to Case 2, 
$$\rho_{s,G'}(W') \geq \rho_{s,G}(W' \cup \{v\}) - 5 - 3 + 10 \geq \rho_{s,G}(W' \cup \{v\}) +2 \geq -2.$$
So $G'$ must contain a forbidden subgraph.

By definition, this implies that $w_1$ and $w_2$ are linked via some subgraph $H$.
By Lemma~\ref{links are special} they are specially-linked.
Corollary~\ref{how to show H is not present}(i) implies that 
\begin{equation}
\label{min degree 3 linked neighbors}
 \rho_{s,G}(\VH \cup \{v\}) \leq \rho_{s,H}(\VH) + 3 - 5 \leq 0. 
\end{equation}
Lemma~\ref{small subgraphs part 1} shows that $\VG = \VH \cup \{v\}$.
Further, $H$ is an induced subgraph and no vertex in $\VG\setminus\{v\}$ is
precolored; otherwise inequality \eqref{min degree 3 linked neighbors}
can be strengthened by $5$, which gives an outright contradiction.

It is straightforward to check that if $H \in \{K_4, W_5, J_7, J_{12}\}$ and
$w_1w_2\in \EH$, then $H-w_1w_2$ has an nb-coloring $I',F'$ 
with $\{w_1,w_2\} \subseteq I'$.  So $H$ must contain a cycle $C =
x_1,\ldots,x_k$ as in Definition~\ref{defn H'}.  By Lemma~\ref{no cycle of 3s
of length at most 5}, $G$ contains no instance of $C$ as in
Definition~\ref{defn H'}.  So there exists $j$ such that either $w_1w_2 =
x_jz_j$ or else $w_1w_2 = x_jx_{j+1}$.  Thus, $H - C$ is an induced subgraph of
$G$, so it has an nb-coloring $I',F'$.

\textbf{Case 3.a: $\bbs{w_1w_2 = x_jz_j}$.}
By symmetry, assume $j=1$.  By Lemma~\ref{put one vertex in I}, $H-C$ has an
nb-coloring $I',F'$ with $z_1\in I'$.  To extend $I',F'$ to $G$, let
$I=I'\cup\{x_1\}$ and $F'=F\cup\{v\}\cup\{x_2,\ldots,x_k\}$.

\textbf{Case 3.b: $\bbs{w_1w_2 = x_jx_{j+1}}$.}
By symmetry, assume $j=k-1$.
By Lemma~\ref{put one vertex in F}, we assume $z_1 \in F'$.  By
Lemma~\ref{linked vertices lemma} and Definition~\ref{defn H'}(ii.c), we assume
that $\{z_1, \ldots, z_k\} \subseteq F'$.  To extend $I',F'$ to $G$, 
if $k$ is even, then let
$I=I'\cup \bigcup_{i=1}^{k/2}x_{2i}$ and
$F=F'\cup\{v\}\cup\bigcup_{i=1}^{k/2} x_{2i-1}$.
If $k$ is odd, then let
$I=I'\cup\{x_k\}\cup \bigcup_{i=1}^{(k-1)/2}x_{2i}$ and
$F=F'\cup\{v\}\cup\bigcup_{i=1}^{(k-1)/2} x_{2i-1}$.
\end{proof}

Now we can show that the uncolored 3-vertices, with no incident edge-gadgets, induce a forest.
We extend the ideas of Lemma~\ref{no cycle of 3s of length at most 5} to all finite $k$.

\begin{lem}
\label{no cycle of 3s}
If $C = x_1\cdots x_k$ is a cycle in $G$ induced by 3-vertices, then at least one of the following holds: 
\begin{enumerate}
\item[(i)] some $x_i$ is incident to an edge-gadget or 
\item[(ii)] $V(C) \cap F_p \neq \emptyset$.
\end{enumerate}
\end{lem}

\begin{proof}
Suppose, to the contrary, that $x_1\cdots x_k$ satisfies the hypotheses, but
both possible conclusions fail.  By Lemma~\ref{no cycle of 3s of length at most
5}, $k \geq 6$.  Let $W = \VG - C$ and let $\{z_i\} = N(x_i) - C$.
If $k$ is even, then by Lemma \ref{put one vertex in F} $G[W]$ has an nb-coloring 
with $z_1 \in F'$, and we can extend it to $G$ by Lemma~\ref{color-k-cycle}.  
Thus, we assume that $k$ is odd; so $k\ge 7$.  

If $z_1 = \cdots = z_k$, then $\rho_{s,G}(V(C) \cup \{z_1\}) = 8(k + 1) - 5(2k) = 8 - 2k \leq -6$,
which is a contradiction.  Thus, the set $\{z_1, \ldots, z_k\}$ contains at
least two distinct vertices.  Our plan for the rest of the proof is similar to
the first sentence of this paragraph.  We will find a subset $V_{J^*_\ell}$ of
$W$ that contains all $z_i$ and such that $\rho_{s,G}(V_{J^*_\ell})\le 7$. 
(We will show that each distinct pair $z_i$, $z_{i+1}$ is linked, and let
$V_{J^*_\ell}$ be the vertices of the union of their linking subgraphs.)
This implies that $\rho_{s,G}(V_{J^*_\ell}\cup C)\le
\rho_{s,G}(V_{J^*_\ell})+8k-5(2k)\le 7-2k\le -7$, which is a contradiction. 
So it remains to find this $V_{J^*_\ell}$ and prove that
$\rho_{s,G}(V_{J^*_\ell})\le 7$.

Suppose there exists $j\in\{1,\ldots,k\}$ such that $z_j\ne z_{j+1}$ and $z_j$
and $z_{j+1}$ are not linked.  Let $G'=G(C,z_j,z_{j+1})$.  Note that
$\rho_{s,G'}(U)\ge \rho_{s,G}(U)-6\ge 4-6=-2$ for all $U\subseteq W$.
Since $z_j$ and $z_{j+1}$ are not linked, $G'$ contains no forbidden subgraphs.
So, by minimality, $G'$ has an nb-coloring, $I',F'$.  And by
Lemma~\ref{color-k-cycle}, we can extend $I',F'$ to $G$.  Thus, for each $j$
with $z_j\ne z_{j+1}$, we know that $z_j$ and $z_{j+1}$ are linked.
By Lemma~\ref{links are special}, in fact they are specially-linked.  Let $L =
\{j: 1 \leq j < k, z_j \neq z_{j+1}\}$.  As shown above, $L \neq \emptyset$.
For each $j \in L$, let $H_j$ denote the subgraph of $G[W]$ that links $z_j$ with $z_{j+1}$.

\begin{clm}
For each $U\subseteq \VHj$, we have $\rho_{s,G}(\VHj)\le \rho_{s,G}(U)$.
\end{clm}
\begin{clmproof}
Let $\tilde{H}_j$ be the graph in $\HH$ formed from $H_j$ by adding edge $z_jz_{j+1}$.
We note that $\rho_{s,G}(H_j) = \rho_{s,\tilde{H}_j}(V_{\tilde{H}_j}) +5$, and
we consider the possibilities for $\rho_{s,\tilde{H}_j}(V(\tilde{H}_j))$.
If $\rho_{s,\tilde{H}_j}(V_{\tilde{H}_j}) \leq -1$, then $\rho_{s,G}(H_j) \leq 4$.
By the gap lemma, 
$\rho_{s,G}(U) \geq 4$.
If $\rho_{s,\tilde{H}_j}(V_{\tilde{H}_j}) \in \{0,1\}$, then $\rho_{s,G}(H_j) \leq 6$.
By Corollary~\ref{how to show H is not present}(iv), $\rho_{s,G}(U) \geq 6$.
Finally, assume that $\rho_{s,\tilde{H}_j}(V_{\tilde{H}_j}) \geq 2$.  
By Corollary~\ref{how to show H is not present}(i), this means that $\tilde{H}_j = K_4$.
The proper, induced, non-trivial subgraphs of $K_4$ are $K_1, K_2, K_3$, which have potentials $8,11,9$.
This proves the claim.
\end{clmproof}

Let $J_0 = \{z_1\}$, and for each $1 \leq i < k$, if $i \notin L$ let $J_{i} = J_{i-1}$, otherwise $J_{i} = J_{i-1} \cup H_{i}$.
Let $\ell$ be the minimum element of $L$, which implies $J_\ell = H_\ell$.
Because each graph in $\HH$ has potential at most $2$, we have $\rho(J_\ell) \leq 2 + 5 = 7$.

Let $t$ be an arbitrary element in $L$.
Since $z_t\in V(H_t\cap J_{t-1})$, it is a non-empty subset of $\VHj$. So the
previous claim implies that $\rho_{s,G}(V(J_{t-1})\cap V(H_t))\ge
\rho_{s,G}(V(H_t))$ for all $t$. Now Lemma~\ref{submodular} implies
$$\rho_{s,G}(V({J_t})) = \rho_{s,G}(V(J_{t-1})\cup V({H_t}))
\leq \rho_{s,G}(V({J_{t-1}})) + \rho_{s,G}(V({H_t})) -
\rho_{s,G}(V(J_{t-1}) \cap V({H_t})) \le \rho_{s,G}(V(J_{t-1})).$$
Clearly this inequality also holds if $t \notin L$.
By applying this inequality for each $t\in\{\ell + 1,\ldots,k\}$, we get
$$\rho_{s,G}(V({J_k})) \leq \rho_{s,G}(V(J_{\ell})) \leq 7,$$
which completes the proof.
\end{proof}

\subsection{Simple Graphs: Discharging and Finishing the Coloring}\label{simple proof sec 2}
\label{color-simple-sec}

In this section we continue our proof that our counterexample $G$ is
``well-behaved''; we ultimately construct an nb-coloring of $G$, which
contradicts that $G$ is a counterexample.  

\subsubsection{Discharging to Force Structure}
Let \Emph{$d'(v)$} denote the degree of vertex $v$, when we count
each edge-gadget as contributing 2 to the degree of each endpoint.  \emph{Throughout
this section whenever we write degree we mean $d'$.}

Let \Emph{$L$} denote the set of vertices in $G$ that are uncolored, degree 3,
and not incident to any edge-gadget ($L$ is for low degree, or little risk).
Let $B=V(G)\setminus L$ (here $B$ is for bigger degree, or bigger risk).
Let $B_j \subset B$\aside{$B_j$} denote the set of vertices in $G$ that are
uncolored, degree $j$, and not incident to any edge-gadget.
Let $B_j^{(eg)} \subset B$\aside{$B_j^{(eg)}$} denote the set of vertices in
$G$ that are degree $j$ and incident to an edge-gadget.
Let $B_j^{(f)} \subset B$\aaside{$B^{(f)}_j$}{4mm} denote the set of vertices in $G$
that are degree $j$ and in $F_p$.  
By Proposition~\ref{fancy edges incident to uncolored} each vertex $v$ is
incident to at most one edge-gadget, and not incident to an edge-gadget at all when $v\in F_p$.
That is, $B_j^{(eg)}\cap B_j^{(f)}=\emptyset$ for each $j$.
Let $B_* = B \setminus (B_4 \cup B_5 \cup B_4^{(eg)}\cup B_5^{(eg)} \cup
B_3^{(f)})$\aside{$B_*$}.
We will use discharging to show that nearly all of $V(G)$ is contained in $L\cup
B_4$ and that $G[B_4]$ has very few edges.  In particular, we will show that
$B_*=\emptyset$.  Our idea is to assign charges to $V(G)\cup E(G)$ that sum to
at most 4, and to discharge so that every vertex and edge has nonnegative
charge, but each vertex outside $L\cup B_4$ has positive charge.

We recall a few useful facts.
By Lemma~\ref{min degree 3 simple}, $\delta(G) \geq 3$.
By Lemma~\ref{no cycle of 3s}, $G[L]$ is a forest.
By Lemma~\ref{Ip-empty-lem}, $I_p = \emptyset$, and by hypothesis $\rho_{s,G}(\VG) \geq -4$.

We assign to each vertex $v$ and edge $e$ a charge, denoted $\ch(v)$ or $\ch(e)$
as follows.
For each vertex $v\in U_p$, let $\ch(v)=2.5d'(v)-8$, and for each $v\in F_p$, let
$\ch(v)=2.5d'(v)-3$.  For each edge-gadget $e$, let $\ch(e)=1$.\aside{$\ch(v)$\\
$\ch(e)$}  (Each edge $e$
that is not an edge-gadget has $\ch(e)=0$.)
The sum of these initial charges is

\begin{align*}
\sum_{v \in U_p}5d'(v)/2-8+\sum_{v\in F_p}5d'(v)/2-3 + e''(V(G)) &=
-8|U_p|-3|F_p|+5e'(\VG)+11e''(\VG) \\
&= -\rho_{s,G}(\VG) \\
&\le 4.
\end{align*}

We use only a single discharging rule, and write $\ch^*$\aside{$\ch^*$} for the
charges after applying it. 
\begin{itemize}
\item[(R1)] Each vertex $v\in B$ gives $1/2$ to each
adjacent 3-vertex and gives $1/2$ to each edge with its other endpoint in $B$
(which means giving $2/2$ to each incident edge-gadget).
\end{itemize}

Now we show that each vertex and edge ends with nonnegative charge.  
Note that each edge-gadget $e$ has $\ch^*(e)=1+4(1/2)=3$ since, by definition,
both its endpoints are in $B$.  Further, each edge $e$ induced by $B$ has $\ch^*(e)=0+2(1/2)=1$.  
For each tree $T$ of $G[L]$, we compute the charge of the entire tree (the sum of the charges
of its vertices), showing it is at least 1.  Let $k=|V(T)|$.  The
number of edges with exactly one endpoint in $T$ is $3k-2(k-1)=k+2$.  Note that
$\ch(T)=k((5/2)3-8)=-k/2$  So $\ch^*(T)=-k/2+(k+2)/2=1$. 

Now we consider vertices in $B$.
If $v\in F_p$, then $\ch^*(v)=5d'(v)/2-3-d'(v)/2=2d'(v)-3\ge 3$.
Recall that $\delta(G)\ge 3$; that is, each vertex has at least 3 neighbors
(excluding multiplicity for edge-gadgets).  So if $v$ is incident to an edge-gadget, then
$d'(v)\ge 4$.  Thus, if $v\in U_p\cap B$, then $d'(v)\ge 4$.  Hence,
if $v\in U_p\cap B$, then $\ch^*(v)=5d'(v)/2-8-d'(v)/2=2d'(v)-8\ge 0$.  

Let $\ell$ denote the number of components in $L$.
Recall that $e'(B)$ and $e''(B)$ denote, respectively, the number of edges in
$G[B]$ that are not edge-gadgets, and are edge-gadgets.
Our observations imply that
\begin{align}
\ell + e'(B) + 3e''(B) + 2|B_5| + 2|B_5^{(eg)}| + 3|B_3^{(f)}| +  4|B_*| \leq 4.
\label{dis-in}
\end{align}

In Lemma~\ref{restrict-lem} we use \eqref{dis-in} to greatly restrict the
structure of $G$.  For the proof we will use a key lemma about extending
nb-colorings of $G[B]$ to all of $G$.  To keep the flow of our
presentation, we state the lemma now, but defer its proof a bit longer.

\begin{lem-reword}
\label{extend-color-reword-lem}
For a graph $G$, let $\vph'$ be a coloring of some $W\subseteq \VG$ such that
$\vph'$ is an nb-coloring of $G[W]$, and such that $G-W$ is a forest in which
each vertex has degree 3 in $G$.  We can extend $\vph'$ to an nb-coloring of $G$
whenever each component $T$ of the forest has either (i) a leaf with no
neighbors in $W$ colored $F$ or (ii) an odd number of incident edges leading to
neighbors in $W$ colored $F$.
\end{lem-reword}
%
%
%

\begin{lem}
\label{restrict-lem}
$\ell+e'(B)\le 4$, $\ell\ge 1$, $e'(B)\ge 1$, and $V(G)=L\cup B_4$. 
\end{lem}
\begin{proof}
The first inequality follows directly from~\eqref{dis-in}.
Next we recall that $\delta(G)\ge 3$, which implies $|\VG|\ge 4$; combining these
inequalities yields $|\EG|\ge 6$.  Since \eqref{dis-in} implies
$e(B)\le 4$, we must have $L\ne \emptyset$.  That is, $\ell\ge 1$.
Since $\ell\ge1$, note that \eqref{dis-in} implies $B_*=\emptyset$.  Further,
if $|B_5^{(eg)}|\ge 1$, then \eqref{dis-in} fails, since $e''(B)\ge 1$, so
$1+3+2\not\le 4$; thus, $B_5^{(eg)}=\emptyset$.

All that remains is to show that $V(G)=L\cup B_4$.  This will imply $e'(B)\ge
1$, since otherwise we can color $B$ with $I$ and color $L$ with $F$.  
Since $B_*=\emptyset$ and $B_5^{(eg)}=\emptyset$, to show that
$V(G)=L\cup B_4$, we will show that $B_3^{(f)}=\emptyset$,
$B_4^{(eg)}=\emptyset$, and $B_5=\emptyset$.

Suppose that $B_3^{(f)}\ne\emptyset$.  Inequality~\eqref{dis-in}
implies that $e'(B)+e''(B)=0$, and $|B_3^{(f)}|=1$. 
Let $w$ denote the vertex in $B_3^{(f)}$, and let $v_1, v_2, v_3$ denote the
neighbors of $w$.  Since $e'(B)+e''(B)=0$, each $v_i$ is in $L$.  Let $T'$
denote the subgraph of $T$ that is the union of the three paths with endpoints
in $\{v_1,v_2,v_3\}$.  Either $T'$ is a subdivision of $K_{1,3}$ or else $T'$
is a path.  In the first case, let $x$ denote the vertex of degree 3 in $T'$.
Now we let $I=B\cup\{x\}\setminus\{w\}$ and $F=L\cup\{w\}\setminus\{x\}$.
In the second case, some $v_i$ has degree 2 in $T'$; by symmetry, say it is
$v_2$.  Now let $I=B\cup\{v_2\}\setminus\{w\}$ and
$F=L\cup\{w\}\setminus\{v_2\}$.  Thus, we must have $B_3^{(f)}=\emptyset$.

Suppose that $B^{(eg)}\ne\emptyset$, which implies that $e''(B)\ge 1$. Now
\eqref{dis-in} implies $e''(B)=1$, $e'(B)=0$, $\ell=1$, and $\VG=L\cup B_4\cup
B_4^{(eg)}$.  
Let $\tB$ denote the 2 endpoints of the edge-gadget.
Since $\ell=1$, let $T=G[L]$.  If $T$ has at least three leaves, then one of
them, call it $v$, has a neighbor not in $\tB$.  Choose $w\in\tB$ such that $v\notin
N(w)$.  Let $F=\{w\}$ and $I=B\setminus\{w\}$.  Since $v$ has two neighbors in
$B$ colored $I$, we can extend the coloring to $G$ by
Lemma~\ref{coloring-lem} (Rephrased), part (i).
Thus, we assume $T$ has only two leaves; that is, $T$ is a path.  
Further, we assume that each leaf of $T$ is adjacent to both vertices in $\tB$,
since otherwise the argument above still works.
Since $G$ has
no copy of $K_4$, the path $T$ is longer than a single edge.  So $B_4\supsetneq
\tB$.  Let $v$ denote a vertex of $\tB$ and $w$ a vertex in $B_4\setminus \tB$.
Let $I=B\setminus\{v,w\}$.  

Let $z_1,\ldots,z_4$ denote the neighbors of $w$ along the path $T$ (in order).
Let $I=(B\setminus\{v,w\}) \cup\{z_1,z_3\}$ and $F=(L\setminus\{z_1,z_3\})
\cup\{v\}$.  It is easy to check that $I,F$ is an nb-coloring of $G$.
Thus, $e''(B)=0$, which implies $B_4^{(eg)}=\emptyset$.

Finally, suppose $B_5\ne \emptyset$.  Now \eqref{dis-in} implies $|B_5|=1$,
$e'(B)=1$, and $\ell=1$.  So let $T=G[L]$.
Let $e$ denote the edge induced by $B$ and let $x$ denote an endpoint of $e$
with $d(x)=4$.  
Let $I=B\setminus\{x\}$ and $F=\{x\}$.  
The only edges incident to $T$ with an endpoint colored $F$ are the 3 edges
incident to $x$ (other than $e$).  So we can extend the nb-coloring of $B$ to
$V(G)$ by Lemma~\ref{coloring-lem}(ii).
 This shows that $B_5=\emptyset$, which completes the proof of the lemma.
\end{proof}

\subsubsection{Why the Theorem We Prove Must Be Sharp}
\label{why-sec}
In Section \ref{coloring-simple-subsec} we will show that if a graph $G$
satisfies $\delta(G) = 3$, $\Delta(G) = 4$, has its vertices of degree $3$
induce a forest with $\ell$ components, and has at most $4-\ell$ edges with both 
endpoints of degree $4$, then either $G$ (i) is near-bipartite, (ii)
contains a subgraph isomorphic to $M_7$, or (iii) is $J_7$ or $J_{12}$.
In Section~\ref{coloring-lems-sec} we prove several lemmas that help us 
find nb-colorings.  Even with these tools, Section~\ref{coloring-simple-subsec} 
consists of a long, technical case analysis.  So, before we continue, we should
explain why Section~\ref{coloring-simple-subsec} is essential.

Our case analysis would be greatly reduced if we could instead assume that
$\ell+e'(B)\le 3$, and it would be nearly trivial if $\ell+e'(B)\le 2$.
These assumptions correspond to the moderately weaker result that $G$ is 
near-bipartite whenever all $W\subseteq V(G)$ satisfy
$\rho_{s,G}(W) \geq -3$ (respectively $\rho_{s,G}(W) \geq -2$).
The work in Section~\ref{coloring-simple-subsec} is
necessary because such modifications would make our work up to this point 
more difficult, bordering on impossible.

The technique that we use---letting $G$ be a minimum counterexample---is akin
to a proof by induction.  A weaker theorem provides a weaker inductive
hypothesis\footnote{Leading to a dictum of Douglas West, ``If you can't prove
something, try proving something harder!''}.  The gaps in the gap lemmas 
($1-(-4)=5$ and $4-(-4)=8$) correspond to the decreases in potential resulting
from precoloring a single vertex ($8-3=5$ and $8-0=8$).
The latter values would not
change by altering the statement of the Main Theorem.  If we merely had the
weaker inductive hypothesis that graphs smaller than $G$ with potential at
least $-3$ are near-bipartite, then our first gap lemma (Lemma~\ref{small
subgraphs part 1}) would be insufficent to precolor a vertex (Lemma~\ref{put
one vertex in F}).  But we cannot delay proving Lemma~\ref{put one vertex in
F} until after a larger gap is proved \emph{precisely because} Lemma~\ref{put
one vertex in F} is used in the proofs of the stronger gap lemmas
(Lemmas~\ref{small subgraphs part 2} and~\ref{gap-simple}).


\subsubsection{Coloring Lemmas}
\label{coloring-lems-sec}
In the previous lemma we showed that $V(G)=L\cup B_4$. Further, $\ell\ge
1$, $e'(B)\ge 1$, and $\ell+e'(B)\le 4$.  In Section~\ref{coloring-simple-subsec}, we
will show how to color $G$.  
Our main tools will be Lemmas~\ref{coloring-lem} and~\ref{helper-lem}, which
allow us to extend partial nb-colorings to components of $G[L]$.  To prove
the first of these, we use the following technical result.  Let $S_1\dcup
S_2$\aaside{$S_1\dcup S_2$}{-6mm} denote the disjoint union of sets $S_1$ and $S_2$.
When vertices $v$ and $w$ are adjacent we write $v\adj w$, and
otherwise $v\nonadj w$\aaside{$v\adj w$\\ $v\nonadj w$}{-6.5mm}.
An operation that we will use repeatedly is to \Emph{suppress} a vertex of degree 2,
which is to delete it and add an edge between its neighbors.

\begin{figure}[!h]
\centering
\begin{tikzpicture}[semithick, scale=.85]
\tikzset{every node/.style=uStyle}

\draw (0,0) node {} -- (1,0) node {} -- (1,1) node {} (2,0) node {};
\draw (0,.25) node[lStyle] {\footnotesize{o}} (.75,1) node[lStyle]
{\footnotesize{o}} (2,.3) node[lStyle] {\footnotesize{i}};
\arrowdrawthick (1.25,0) to (1.75,0);

\begin{scope}[xshift=1.5in]
\draw (0,0) node {} -- (1,0) node {} -- (1,1) node {} (2,0) node {};
\draw (0,.3) node[lStyle] {\footnotesize{i}} (.75,1) node[lStyle]
{\footnotesize{o}} (2,.25) node[lStyle] {\footnotesize{o}};
\arrowdrawthick (1.25,0) to (1.75,0);
\end{scope}

\begin{scope}[xshift=3.0in]
\draw  (1,0) -- (2,0) (0,0) node {} -- (1,0) node {} -- (1,1) node {} (2,1) node {} -- (2,0)
node {} -- (3,0) node {};

\draw (0,.3) node[lStyle] {\footnotesize{i}} (.75,1) node[lStyle] {\footnotesize{i}};
\arrowdrawthick (3.5,0) to (4.0,0);
\begin{scope}[xshift=-1mm]
\draw (4.3,1) node {} .. controls ++(270:1.05) and ++(180:1.05) .. (5.3,0) node {};
\end{scope}
\end{scope}
\end{tikzpicture}
\caption{The induction step in the proof of Lemma~\ref{tree-helper}.}
\end{figure}
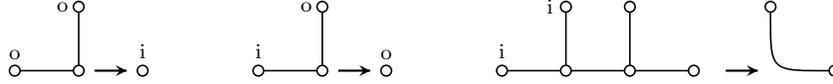

\begin{lem}
\label{tree-helper}
Let $T$ be a tree in which each non-leaf vertex has degree 3.  Let $S_{in}\dcup
S_{out}$ be a partition of the leaves of $T$.  If $|S_{out}|$ is odd, then $T$
has an independent set $S$ such that $S_{in}\subseteq S$ and $S_{out}\cap
S=\emptyset$, and also each component of $T-S$ contains at most one leaf of $T$.
\end{lem}
\begin{proof}
Let $k$ denote the number of leaves in $T$.
Our proof is by induction on $k$.  
If $k=1$, then $T$ is an isolated vertex contained by $S_{out}$.
Set $S = \emptyset$.
If $k=2$, then $T \cong K_2$ and $|S_{out}| = |S_{in}| = 1$.
Set $S = S_{in}$.
For good measure, we also consider $k=3$, where $T=K_{1,3}$. 
If all leaves are in $S_{out}$, then we take $S$ to be the center vertex.
Otherwise, one leaf is in $S_{out}$ and the other two are in $S_{in}$, so we
take $S$ to consist of the two leaves in $S_{in}$.

Now suppose that $k\ge 4$.  The number of non-leaf vertices in $T$ is $k-2$, and
each of these has at most two leaf neighbors.  By Pigeonhole, some non-leaf
vertex $v$ has exactly two leaf neighbors, say $w_1$ and $w_2$.  If $w_1,w_2\in
S_{out}$, then we apply induction to $T-\{w_1,w_2\}$, with leaf partition
$S'_{in}\dcup S'_{out}$, where $S'_{in}=S\cup\{v\}$ and
$S'_{out}=S_{out}\setminus\{w_1,w_2\}$.  If $|\{w_1,w_2\}\cap S_{out}|=1$, then
we assume $w_1\in S_{out}$ (by symmetry) and let $T'=T-\{w_1,w_2\}$.  We apply
induction to $T'$ with $S'_{out}=(S_{out}\setminus\{w_1\})\cup\{v\}$
and $S'_{in}=S_{in}\setminus\{w_2\}$, and let $S'$ be the guaranteed independent
set.  Let $S=S'\cup\{w_2\}$.  Finally, suppose that $w_1,w_2\in
S_{in}$.  Now let $x$ denote the third neighbor of $v$.  Form $T'$ from
$T-\{v,w_1,w_2\}$ by suppressing $x$.  Let $S'_{out}=S_{out}$ and
$S'_{in}=S_{in}\setminus\{w_1,w_2\}$.  Given the independent set $S'$ for $T'$ by
induction, let $S=S'\cup\{w_1,w_2\}$.
\end{proof}

\begin{remark}
Recall, from Lemma~\ref{restrict-lem}, that $V(G)=L\cup B_4$ and that $G[L]$ is a forest.  
All figures in the rest of the paper will denote nb-colorings of $G$.  Vertices
in $I$ are drawn as \tikz \draw node[uBstyle]{} node[uStyle]{}; and those in
$F$ are drawn as \tikz \draw node[uStyle]{};.  Edges in bold denote those induced by vertices of $L$.
\end{remark}

\begin{defn}
Fix an nb-coloring $\vph'$ of $G[B]$.  
Now each edge from a vertex of $L$ to a vertex of
$B$ colored $F$ is an \EmphE{$F$-edge}{-5mm} (an $I$-edge is defined
analogously).  We say that the $F$-edges incident to a
component $T$ of $G[L]$ are $F$-edges belonging to $T$. 
A component $T$ of $G[L]$ is
\EmphE{$F$-odd}{-6mm} (resp.~\EmphE{$F$-null}{-3mm}) if its number of $F$-edges is odd
(resp.~0).  Further, $T$ is \Emph{$F$-leaf-good} if some leaf of $T$ has two
neighbors in $B$ colored $I$.  (If $T$ is $F$-null, then clearly $T$ is
$F$-leaf-good.)
\end{defn}

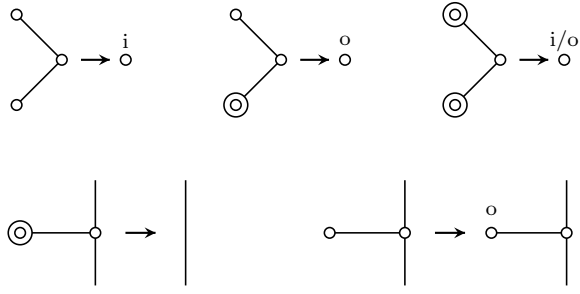
\begin{figure}[!h]
\centering

\begin{tikzpicture}[semithick]
\tikzset{every node/.style=uStyle}
\tikzstyle{lStyle}=[shape = rectangle, minimum size = 0pt, inner sep =
0pt, outer sep = 0pt, draw=none]

\def\rtt{.7071} 

\begin{scope}[xshift=.1in,scale=.85]
\draw (-\rtt,\rtt) node {} -- (0,0) node {} -- (-\rtt,-\rtt) node {} (1,0) node {};
\draw (1,.3) node[lStyle] {\footnotesize{i}};
\arrowdrawthick (.3,0) to (.75,0);

\begin{scope}[xshift=1.35in]
\draw (-\rtt,\rtt) node {} -- (0,0) node {} -- (-\rtt,-\rtt) node[uBstyle] {} node {} (1,0) node {};
\draw (1,.3) node[lStyle] {\footnotesize{o}};
\arrowdrawthick (.3,0) to (.75,0);

\end{scope}

\begin{scope}[xshift=2.7in]
\draw (-\rtt,\rtt) node[uBstyle] {} node {} -- (0,0) node {} 
-- (-\rtt,-\rtt) node[uBstyle] {} node {} (1,0) node {};
\draw (1,.3) node[lStyle] {\footnotesize{i/o}};
\arrowdrawthick (.3,0) to (.75,0);

\end{scope}
\end{scope}

\begin{scope}[yshift=-1.30in, xshift=.7cm]

\draw (0,.3) -- (0,1.7) (-1,1) node[uBstyle] {} node {} -- (0,1) node {} (1.2,.3) -- (1.2, 1.7);
\arrowdrawthick (.4,1) -- (.8,1);

\begin{scope}[xshift=1.62in]
\draw (0,.3) -- (0,1.7) (-1,1) node {} -- (0,1) node {}; (1.2,.3) -- (1.2, 1.7);
\arrowdrawthick (.4,1) -- (.8,1);
\begin{scope}[xshift=.15cm]
\draw (2,.3) -- (2,1.7) (1,1) node {} -- (2,1) node {} (1,1.3) node[lStyle] {\footnotesize{o}};
\end{scope}
\end{scope}

\end{scope}

\end{tikzpicture}
\caption{The construction of tree $T'$ in the proof of Lemma~\ref{coloring-lem}.}

\end{figure}

\begin{lem}
For a graph $G$, let $\vph'$ be a coloring of some $W\subseteq \VG$ such
that $\vph'$ is an nb-coloring of $G[W]$, and such that $G\setminus W$ is a
forest in which each vertex has degree 3 in $G$.  We can extend $\vph'$ to an
nb-coloring of $G$ whenever each component $T$ of the
forest is either (i) $F$-odd or (ii) $F$-leaf-good.
\label{coloring-lem}
\end{lem}
\begin{proof}
Suppose that $G$, $W$, and $\vph'$ satisfy the hypotheses.  Let $T$ be a
(tree) component of $G-W$.  We show how to extend $\vph'$ to $V(T)$ so that no
two of its vertices with incident $F$-edges are linked by a path in $T$
entirely colored $F$. 

From $T$ we form a new tree $T'$, and leaf partition $S_{in}\dcup S_{out}$,
 as follows.  When a non-leaf $v$ of $T$ has an incident $I$-edge, we
suppress $v$.  When a non-leaf $v$ of $T$ has an incident $F$-edge,
we add a leaf $w_v$ incident to $v$ and add $w_v$ to $S_{out}$.  
When a leaf $v$ of $T$ has two incident $F$-edges, we add $v$ to $S_{in}$. 
When a leaf $v$ has both an incident $I$-edge and an incident $F$-edge, we add
$v$ to $S_{out}$.
Now consider leaves of $T$ with two incident $I$-edges (if such leaves exist).  
For all but one of
these, say $w$, we add them to $S_{in}$ or $S_{out}$ arbitrarily.  Finally, we
add $w$ to either $S_{in}$ or $S_{out}$ so that $|S_{out}|$ is odd.  Under both
hypotheses (i) and (ii), we get that $|S_{out}|$ is odd.

Now we invoke Lemma~\ref{tree-helper}, to find an independent set $S$ such that
$S_{in}\subseteq S$ and $S_{out}\cap S=\emptyset$, and also each component of
$T-S$ contains at most one leaf of $T'$.
We color each vertex of $S$ with
$I$, except for leaves of $T$ with two incident $I$-edges.  It is easy to
check that no two vertices of $T$ with incident $F$-edges are linked by a path
in $T$ all colored $F$.
\end{proof}

\begin{lem}
Let $G$ be a tree or else be connected and have a single cycle $C$, which is
not a 3-cycle.  Form $G'$ from $G$ by
adding a new vertex $v$ and making $v$ adjacent to at most four vertices in
$\VG$, at least one of which is on $C$, if $C$ exists.  Now $G'$ has a near
bipartite coloring $I,F$ with $I\subseteq N_{G'}(v)$.
\label{helper-lem}
\end{lem}
\begin{proof}
First suppose that $G$ is a tree.  If at least $d_{G'}(v)-1$ neighbors of $v$
induce an independent set, then we color them with $I$ and color the rest of
$G'$ with $F$.  If this is not the case, then $d_{G'}(v)=4$ and the four
neighbors of $v$ induce either 2 or 3 edges.  In each case, we can color two
of these neighbors with $I$ and the rest of $G'$ with $F$.  

So assume instead that $G$ has a cycle, $C$.  Let $S=N_{G'}(v)$.  Our goal is
again to use color $I$ on some independent set $S'\subseteq S$. As before 
$S'$ must intersect every cycle in $G'$ through $v$, but now we also require
that some vertex in $S'$ lies on $C$.  If some independent $S'\subseteq S$ has
size at least $d_{G'}(v)-1$ and intersects $C$, then we are done.  This includes
the case when $S$ induces at most one edge, specifically when $d_{G'}(v)\le 2$.
If $d_{G'}(v)=3$, but the case above does not apply, then $S$ induces $P_3$
with only the center vertex on $C$; so we let $S'$ consist of this
center vertex.  Thus, we assume that $d_{G'}(v)=4$,  and that $S$
induces 2, 3, or 4 edges.  

First suppose that $S$ induces 4 edges.  Since $G$ has no 3-cycle, $G[S]=C_4$.  
Now all vertices of $S$ lie on $C$, so we take $S'$ to be either independent
subset of size 2.

Suppose instead that $S$ induces 3 edges; so $G[S]\in\{P_4,K_{1,3}\}$.  
When $G[S]=K_{1,3}$, let $S'$ be the independent subset of size 3 unless it
does not intersect $C$; in that case, let $S'$ be the other vertex.  
If $G[S]=P_4$, then denote the vertices of $S$ by $w_1,\ldots,w_4$ in order
along the path.  We either let $S'=\{w_1,w_3\}$ or let $S'=\{w_2,w_4\}$.
(If each choice for $S'$ misses some cycle in $G'$, then $G$ contains at least
two distinct cycles, contradicting the hypothesis.)

Finally, assume $S$ induces two edges; so $G[S]\in\{2K_2,P_3+K_1\}$.  
Suppose $G[S]=P_3+K_1$.  If the independent set $S'\subset S$ of size 3 has a
vertex on $C$, then we are done.  Otherwise, let $S'$ consist of the center
vertex of the $P_3$ and its nonneighbor.  So assume instead that $G[S]=2K_2$. Now
it is straightforward to check that we can use as $S'$ one of the independent
sets of size 2 (the general idea is to use one with as many vertices on $C$ as
possible, though not all such sets will work).
\end{proof}

\subsubsection{Coloring the Graph}
\label{coloring-simple-subsec}

Recall that $B=B_4$.
Let \emph{$\tB$} denote the subset of $B$ incident to edges in $G[B]$. (Since $e'(B)\le
3$, we have $|\tB|\le 6$.) We form \emph{$\tG$}\aaside{$\tB$, $\tG$}{-4mm} from $G$ by deleting all vertices of
$B\setminus \tB$ and suppressing all of their neighbors that were not leaves in
$G[L]$.  
(Later we also use the
notation $\tT$.  In each case, the reader should think of \Memph{$\mytilde$} as meaning
`shrinking down to the most important part'.) If $\tG$ has an nb-coloring $I,F$,
then we can extend this coloring to $G$ by adding the deleted vertices
of $B$ to $I$ and the suppressed vertices of $L$ to $F$.  Our goal is to color
$\tG$.  If we can't, then we try unshrinking a deleted vertex and its 4
suppressed neighbors.  If no vertex exists to unshrink, then we show that $G$
contains a forbidden subgraph, contradicting our hypothesis.

We often use Lemma~\ref{helper-lem} to extend an nb-coloring of $\tB$ to a
tree $T$ of $G[L]$, specifically when $F\cup V(T)$ induces a cycle.  The idea is
to find a vertex $x\in B\setminus \tB$ and add it to $F$.  This allows us to add
neighbors of $x$  in $T$ to $I$ (as long as they are not leaves in $T$).  When
we do this, we call $x$ the \Emph{helper} and say that we \emph{color
$T$ by Lemma~\ref{helper-lem}, with $x$ as helper}.

When we describe an nb-coloring of $B$, we often specify only the vertices
in $B\cap F$, implying that $B\setminus F$ is colored $I$.  We extend this
coloring to each component of $G[L]$ using Lemmas~\ref{coloring-lem} and~\ref{helper-lem}.

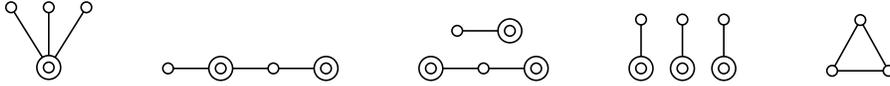
\begin{figure}[!htb]
\centering
\begin{tikzpicture}[semithick]  %
\tikzset{every node/.style=uStyle}
\begin{scope}[xscale=.5, yscale=.8,yshift=-.1in]
\draw (-1,0) node (v1) {} (0,0) node (v2) {}
(1,0)  node (v3) {} (0,-1) node[uBstyle] (v4) {} node {}; 
\draw (v1) -- (v4) -- (v2) (v4) -- (v3);
\end{scope}

\begin{scope}[xshift=.9in,xscale=.7,yshift=-.4in]
\draw (-1,0) node {} -- (0,0) node[uBstyle] {} node {} -- (1,0) node {} -- (2,0) node[uBstyle] {} node {}; 
\end{scope}

\begin{scope}[xshift=2.0in,xscale=.7,yshift=-.4in]
\draw (0,0) node[uBstyle] {} node {} -- (1,0) node {} -- (2,0) node[uBstyle] {} node {};
\draw (.5,.5) node {} -- (1.5,.5) node[uBstyle] {} node {};
\end{scope}

\begin{scope}[xshift=3.1in, yshift=-.4in, yscale=.65, xscale=.55]
\draw (0,1) node {} -- (0,0) node[uBstyle] {} node {} (1,1) node {} -- (1,0)
node[uBstyle] {} node {} (2,1) node {} -- (2,0) node[uBstyle] {} node {};
\end{scope}

\begin{scope}[xshift=4.1in, scale=.75, yshift=-.55in]
\draw (0,0) node {} -- (1,0) node {} -- (.5,.9) node {} -- (0,0);
\end{scope}

\end{tikzpicture}
\caption{We have five possibilities for $G[\tB]$ when
$e(B)=3$, in the proof of Lemma~\ref{lem4.37}.\label{lem4.37-fig}}
\end{figure}

\begin{lem}
If $e'(B)=3$, then $G$ is near-bipartite.
\label{lem4.37}
\end{lem}
\begin{proof}
Suppose that $e'(B)=3$.  Now Lemma~\ref{restrict-lem} implies $\ell=1$; hence
we write $T$ for $G[L]$.  Note that $G[\tB]\in\{K_{1,3},P_4,P_2+P_3, 3K_2,
K_3\}$; see Figure~\ref{lem4.37-fig}.  (Here $K_{1,3}$ denotes a tree on 4
vertices with three leaves, $P_t$
denotes a path on $t$ vertices, $P_2+P_3$ denotes the disjoint union of $P_2$
and $P_3$, and $3K_2$ denotes $K_2+K_2+K_2$.)  All cases but the last can be
handled quickly (as we show below) by coloring $\tB$ so that we can extend the
coloring to $T$ using Lemma~\ref{coloring-lem}.

In each case we describe $F$ and implicitly let $I=B\setminus F$.
If $G[\tB]=K_{1,3}$, then let $F$ consist of the leaves in the $K_{1,3}$.  
Since $T$ has 9 $F$-edges, it is $F$-odd, so we can extend the coloring
by Lemma~\ref{coloring-lem}.  If $G[\tB]=P_4$, then let $F=\{v,w\}$, where $v$
and $w$ are at distance two along the $P_4$.  Now $T$ has 5 $F$-edges.
If $G[\tB]=P_2+P_3$, then let $F=\{v,w\}$, where $v$ is a leaf of the $P_2$ and
$w$ is the center vertex of the $P_3$.  Now, $T$ has 5 $F$-edges, so is
$F$-odd.  Finally, suppose $G[\tB]=3K_2$.  Let $F$ consist of one vertex from
each $K_2$.
Again, $T$ has 9 $F$-edges, so is $F$-odd.

Now assume $G[\tB]=K_3$.  If $T$ has at least 4 leaves, then $G[B]$ also has
some isolated vertices, one of which is adjacent to a leaf $w$ of $T$.  Let
$F=\{v_1,v_2\}$, where the $v_i$ are two vertices of
$G[\tB]$ not adjacent to $w$.  Now we can extend the coloring to $T$, since it
is $F$-leaf-good.  So assume that $T$ has at most 3 leaves.  
Further, we assume that each leaf has two neighbors in $\tB$, since otherwise
the argument above still works.  Form \Memph{$\tT$} from $T$ by suppressing
each vertex $w$ with $d_T(w)=2$ that has a neighbor in $B\setminus \tB$.  Now
$\tT$ has six incident edges to $\tB$, so $\tT\in \{K_{1,3},P_4\}$; see
Figure~\ref{lem4.37b-fig}.

Suppose that $\tT=K_{1,3}$, and let $v_1,v_2,v_3$ denote the vertices of $\tB$.  
So $\tG=J_7$, as shown in Figure~\ref{examples fig}.  Let $w$
denote a leaf of $T$ that is not adjacent to $v_3$, and pick $x\in B\setminus
\tB$; vertex $x$ exists since $J_7$ is forbidden as a subgraph, so $G\ne \tG$.  
Let $F=\{v_1,v_2,x\}$, and color $w$ with $I$.  The subgraph induced by
$(V(T)\setminus\{w\})\cup\{v_1,v_2\}$ has a single cycle.  We assume that $x$
has a neighbor on this cycle; if not, then we repeat the argument with $v_1$ or
$v_2$ in place of $v_3$.  Thus, we can extend the coloring to
$V(T)\setminus\{w\}$ by Lemma~\ref{helper-lem}, using $x$ as helper.

\begin{figure}[!h]
\centering
\begin{tikzpicture}[semithick, scale=.6]
\tikzset{every node/.style=uStyle}

\begin{scope}[yscale=.75, xshift=-1.7in, yshift=.4in]
\draw (90:2cm) node (w1) {} (210:2cm) node (w2) {} (330:2cm) node (w3) {}
(90:3.5cm) node (w4) {};
\draw node (v1) at (barycentric cs:w1=1,w2=1) {};
\draw node (v2) at (barycentric cs:w2=1,w3=1) {};
\draw node (v3) at (barycentric cs:w3=1,w1=1) {};
\draw (w2) -- (v2) -- (w3) (v2) -- (v3) -- (w1) -- (v1) -- (v2);
\draw (w2) -- (v1) -- (v3) -- (w3); 
\draw[decorate,decoration=snake, ultra thick,segment aspect=0, segment amplitude=2pt, segment length=7pt] (w1) -- (w4);
\draw[decorate,decoration=snake, ultra thick,segment aspect=0, segment amplitude=2pt, segment length=9pt] 
(w2) to [out= 90, in = 225](w4); 
\draw[decorate,decoration=snake, ultra thick,segment aspect=0, segment amplitude=2pt, segment length=9pt] 
(w4) to [in= 90, out = -45](w3);
\draw (w1) node[uBstyle] {} node {} (v2) node[uBstyle] {} node {};
\draw (1.5,3.3) node {} -- (.75,2.5) {};
\end{scope}

\draw[ultra thick]  (.2,2) node (w1) {} (w1) -- (1,0) node (w2) {} -- (3,0) node (w3) {} -- (3.8,2) node (w4) {};
\draw (2,3.3) node (z1) {} -- (1.5,1.7) node (z2) {} -- (2.5,1.7) node (z3) {};
\draw (z1) -- (z3) -- (w2) (z3) -- (w3) (z1) -- (z2) -- (w1) (w1) -- (z1) -- (w4) (z2) edge[bend right=50] (w4);
\draw (w1) node[uBstyle] {} node {} (z3) node[uBstyle] {} node {} (w4) node[uBstyle] {} node {};

\begin{scope}[xshift=1.9in]
\draw[ultra thick]  (.2,2) node (w1) {} (w1) -- (1,0) node (w2) {} -- (3,0) node (w3) {} -- (3.8,2) node (w4) {};
\draw (2,3.3) node (z1) {} -- (1.5,1.7) node (z2) {} -- (2.5,1.7) node (z3) {} -- (w4);
\draw (z1) -- (z3) -- (w2) (z1) -- (z2) -- (w3) (w1) -- (z2) -- (z3) -- (w4) -- (z1) -- (w1);
\draw (w1) node[uBstyle] {} node {} (z3) node[uBstyle] {} node {} (w3) node[uBstyle] {} node {};
\end{scope}

\begin{scope}[xshift=3.8in]
\draw[ultra thick]  (.2,2) node (w1) {} (w1) -- (1,0) node (w2) {} -- (3,0) node (w3) {} -- (3.8,2) node (w4) {};
\draw (2,3.3) node (z1) {} -- (1.5,1.7) node (z2) {} -- (2.5,1.7) node (z3) {} -- (w4);
\draw (z1) -- (z3) -- (w3) (z1) -- (z2) -- (w2) (w1) -- (z2) -- (z3) -- (w4) -- (z1) -- (w1);
\end{scope}

\begin{scope}[xshift=5.7in]
\draw[ultra thick,snake=coil,segment aspect=0, segment amplitude=2pt,segment length=10pt] 
(.2,2) node (w1) {} -- (1,0) node (w2) {} (w2) -- (3,0) node (w3) {} -- (3.8,2) node (w4) {};
\draw (w4) -- (2,3.3) node (z1) {} -- (1.5,1.7) node (z2) {} -- (w2);
\draw (w4) -- (2.5,1.7) node (z3) {} -- (z2) -- (w1) -- (z1) -- (z3) -- (w3);
\draw (w1) node[uBstyle] {}  node {} (z3) node[uBstyle] {} node {};
\draw (4.2,.3) node {} -- (3.2,.7); 
\end{scope}
\end{tikzpicture}
\caption{When $G[\tB]=K_3$, in the proof of Lemma~\ref{lem4.37}, we have two
cases.  Left: $\tT=K_{1,3}$. Right: $\tT=P_4$.\label{lem4.37b-fig}}
\end{figure}
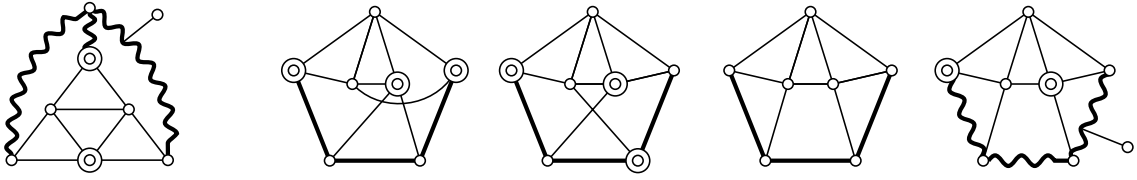

Assume instead that $\tT=P_4$.  Suppose that $T=P_4$.  
By Pigeonhole at least one vertex in $\tB$ is adjacent to both leaves of the
$P_4$.  Now we have three ways for the remaining two vertices of $\tB$ to
attach.  Thus, we have three possibilities
for $G$, each with 7 vertices.  Two of these are non-planar (one has a
$K_{3,3}$-minor and the other a $K_5$-minor).  Each non-planar case has an
independent set of size 3, which we take as $I$.
In fact, this approach works whenever
$\tG$ is either of these non-planar graphs; since $\tG$ has an $I,F$ coloring, so does $G$.
So assume instead that $\tG$ is the other possibility; it is planar and contains
$M_7$ as a (non-induced) subgraph.  This implies that $G\ne \tG$, so $T\ne \tT$.
Let $w$ denote a leaf of $T$ and $v_1,v_2$ its neighbors in $\tB$.  Since $T\ne
\tT$, tree $T$ has a helper vertex $x$.  Note that
$G[(L\setminus\{w\})\cup\{v_1,v_2\}]$ is unicyclic, and let $C$ denote its
cycle.  We assume that $x$ has neighbors on $C$, since if not, then we repeat
the argument with $w$ replaced by the other leaf of $T$.
Let $F=\{v_1,v_2,x\}$.  Now we can extend the coloring to $G$ by
Lemma~\ref{helper-lem}, using $x$ as the helper.
This finishes the case $e'(B)=3$.
\end{proof}

\begin{lem}
If $e'(B)=2$, then $G$ is near-bipartite.
\label{lem4.38}
\end{lem}
\begin{proof}
If $e'(B)=2$, then $G[\tB]\in \{P_3,2K_2\}$ and $\ell\le 2$.  Suppose
$G[\tB]=P_3$.  If $\ell=1$, then we color $G$ as follows. Let $v_1$ denote a 
leaf of $G[\tB]$ and let $F=\tB\setminus\{v_1\}$.  We are done by
Lemma~\ref{coloring-lem}(i), since $\tT$
(and therefore also $T$) has exactly 5 $F$-edges.  Thus, we assume $\ell=2$. 

We denote the two trees of $G[L]$ by $T_1$ and $T_2$.  Let $v_1$ and $v_2$
denote the leaves of $G[\tB]$, and let $w$ denote its non-leaf vertex.
If $w$ has a single neighbor in each of $T_1$ and $T_2$, then let 
$F=\{w\}$.  Each $T_i$ is $F$-odd, so we are done.
Thus, we assume $w$ has two neighbors in $T_1$ (by
symmetry).  Further, $T_1$ is a path with $w$ adjacent to both endpoints, since
otherwise letting $F=\{w\}$ makes both $T_1$ and $T_2$ be $F$-leaf-good.
Note that the numbers of edges incident to $v_1$ and $v_2$ that lead to $T_1$
must have the same parity.  If not, then we let $F=\{v_1,v_2,w\}$ and both $T_1$
and $T_2$ are $F$-odd.  So the possibilities for the numbers of edges from
$v_1$ and $v_2$ to $T_1$ are 0,0; 0,2; 2,0; 2,2; 1,1; 1,3; 3,1; and 3,3.
We refer to these as Case 0,0; Case 0,2; etc.

The easiest to handle are Cases 3,3 and 1,3 (and 3,1, by symmetry).
Let $F=\{w,v_2\}$, which makes $T_1$ to be $F$-odd and $T_2$ to be $F$-null.
So now assume that $v_1$ and $v_2$ each have at least one neighbor in $T_2$.

Before considering the other cases, we prove the following claim.

\begin{clm}
No vertex in $B\setminus \tB$ has a neighbor in $T_1$.
\label{clm4.39}
\end{clm}
\begin{clmproof}
Suppose, to the contrary, that such a vertex exists; call it $x$.
If $x$ has an odd number of edges to $T_1$ and $T_2$, the we let $F=\{w,x\}$.  
Both $T_1$ and $T_2$ are $F$-odd, so we are done.
If $N(x) \subseteq V(T_1)$, then let $F = \{w,x\}$;
since $T_2$ is $F$-null, we color it by Lemma~\ref{coloring-lem}, 
and we color $T_1 \cup \{w\}$ by Lemma~\ref{helper-lem}, using $x$ as helper.
So assume that $x$ has two edges to each of $T_1$ and $T_2$.
If a leaf of $T_2$ is not incident to $x$, then let $F=\{w,x\}$ so that $T_2$ is 
$F$-leaf-good and colorable by Lemma~\ref{coloring-lem}, while $T_1$ is
colorable by Lemma~\ref{helper-lem}, using $x$ as helper.
Assume instead that $T_2$ is a path whose endpoints are adjacent to $x$.
If $v_1$ has no neighbors in $T_1$, then let $F=\{w,v_1,x\}$; now $T_2$ is 
$F$-odd, so colorable by Lemma~\ref{coloring-lem}, and $T_1$ is colorable
by Lemma~\ref{helper-lem} using $x$ as helper.  If $v_1$ has one neighbor in
$T_1$, then let $F = \{w,v_1,x\}$ so that $T_1$ is $F$-odd, and thus colorable
by Lemma~\ref{coloring-lem}, while $T_2$ is colorable by
Lemma~\ref{helper-lem}, using $v_1$ as helper.  Because we have already ruled
out cases 3,1 and 3,3; it follows that $v_1$ must have exactly two neighbors in
$T_1$.  By symmetry, $v_2$ also has exactly two neighbors in $T_1$.

Let $y_1, \ldots, y_\ell$ be the vertices of $T_1$ in order.
Let $z_1, z_2, z_3, z_4$ be the four neighbors of $v_1$ or $x$ in $T_1$ in
order; note that these $z_i$ are distinct, since $w\adj\{y_1,y_\ell\}$.
If $x \adj \{y_1,y_\ell\}$, then let $F = \{w,x,v_1\}$ and color $T_2$ with
Lemma~\ref{coloring-lem} since $T_2$ is $F$-odd.  To color $T_1$, contract
$wv_1$ into a vertex $z$, and then color $T_1 \cup \{x\}$ by
Lemma~\ref{helper-lem} using $z$ as helper.  By symmetry, we assume that $x \not\adj y_1$.
By symmetry between $v_1$ and $v_2$, let us assume that $v_1 \not\adj y_1$, and
thus $y_1 \neq z_1$.  Under these assumptions, color $G \setminus T_2$ with 
$F = \{w,x,v_1\} \cup (T_1 \setminus \{y_1,z_2,z_4\})$ and $I = \VG \setminus
(F \cup V(T_2))$ and extend this coloring to all of $G$ via 
Lemma~\ref{coloring-lem} since $T_2$ is $F$-odd.  Therefore $N(T_1) \subseteq \tB$.
\end{clmproof}

This claim shows that Case 0,0 is impossible.
Since $G$ contains no copy of $K_4$, Cases 2,0 and 0,2 are also impossible.
So all that remain are Case 1,1 and Case 2,2.

Suppose we are in Case 1,1.  That is, $v_1$ and $v_2$ each send a single edge
to $T_1$.  By the claim, this implies that $T_1=K_2$.  
This is shown on the left in Figure~\ref{lem4.38-fig}, where $T_1$ is on top,
$T_2$ is on bottom, and $v_1, w, v_2$ are in the center.  Now $T_2$ must be a
path with $v_1$ and $v_2$ each adjacent to both endpoints of $T_2$ (otherwise
we let $F=\{v_1,w\}$ or $F=\{v_2,w\}$, so $T_1$ is $F$-odd and $T_2$ is
$F$-leaf-good).
If $T_2\ne K_2$, then let $F=\{v_1,v_2\}$.  
We extend this coloring to $G$ as
follows.  Color the endpoints of $T_2$ and $w$ with $I$ and the rest of $T_2$ with $F$,
and color all of $T_1$ with $F$.  (Now $T_1$ has a $v_1,v_2$-path in $F$, but it
does not extend to a cycle in $F$.)  
But if $T_2=K_2$, then $G$ contains the Moser Spindle (in fact $G-v_1w$ is the Moser
Spindle), which is a contradiction.  This completes Case 1,1.

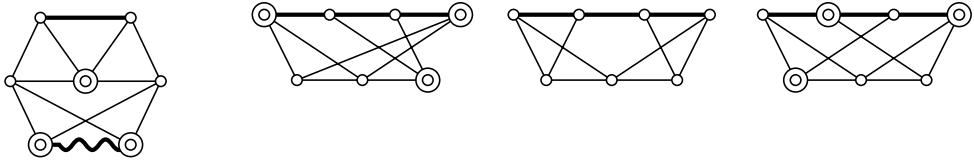
\begin{figure}[!h]
\centering
\begin{tikzpicture}[semithick]  
\tikzset{every node/.style=uStyle}
\begin{scope}[xshift=-1.5in, yshift=.055in, xscale=1.0, yscale=.65]
\draw (0,0) node (v1) {}  (1,0) node (w) {} (2,0) node (v2) {};
\draw[ultra thick] (.4, 1.3) node (x1) {} -- (1.6,1.3) node (x2) {}
(.4,-1.3) node (y1) {} (1.6,-1.3) node (y2) {};
\draw (v2) -- (y2) -- (v1) -- (w) -- (v2) -- (x2) (x1) -- (v1) -- (y1) --
(v2) (x1) -- (w) -- (x2);
\draw[ultra thick,snake=coil,segment aspect=0, segment amplitude=2pt,segment length=10pt] (y2) -- (y1);
\draw (y1) node[uBstyle] {} node {} (w) node[uBstyle] {} node {} (y2) node[uBstyle] {} node {};
\end{scope}

\begin{scope}[scale=.87,yshift=.07in]
\draw[ultra thick] (-.5,1) node (x1) {} -- (.5,1) node (x2) {} -- (1.5,1) node (x3) {} -- (2.5,1) node (x4) {};
\draw (x1) -- (0,0) node (v1) {} -- (1,0) node (w) {} -- (2,0) node (v2) {};
\draw (x1) -- (w) -- (x4) -- (v1) (x2) -- (v2) -- (x3);
\draw (x1) node[uBstyle] {} node {} (x4) node[uBstyle] {} node {} (v2) node[uBstyle] {} node {};

\begin{scope}[xshift=1.5in]
\draw[ultra thick] (-.5,1) node (x1) {} -- (.5,1) node (x2) {} -- (1.5,1) node (x3) {} -- (2.5,1) node (x4) {};
\draw (0,0) node (v1) {} -- (1,0) node (w) {} -- (2,0) node (v2) {};
\draw (x2) -- (v1) -- (x1) -- (w) -- (x4) -- (v2) -- (x3);
\end{scope}

\begin{scope}[xshift=3.0in]
\draw[ultra thick] (-.5,1) node (x1) {} -- (.5,1) node (x2) {} -- (1.5,1) node (x3) {} -- (2.5,1) node (x4) {};
\draw (0,0) node (v1) {} -- (1,0) node (w) {} -- (2,0) node (v2) {};
\draw (x3) -- (v1) -- (x1) -- (w) -- (x4) -- (v2) -- (x2);
\draw (v1) node[uBstyle] {} node {} (x2) node[uBstyle] {} node {} (x4) node[uBstyle] {} node {};
\end{scope}
\end{scope}
\end{tikzpicture}
\caption{When $G[\tB]=P_3$, in the proof of Lemma~\ref{lem4.38}, we have two cases.  
Left: Case 1,1. Right: Case 2,2 (with $T_2$ undrawn).\label{lem4.38-fig}}
\end{figure}

Suppose we are in Case 2,2.  That is, both $v_1$ and $v_2$ have two edges to
$T_1$ and so $T_1$ is a path on four vertices; see the right of
Figure~\ref{lem4.38-fig}.  Let $z_1,z_2,z_3,z_4$ denote
the vertices of $T_1$ in order.  If $z_1$ and $z_4$ are neighbors of
the same $v_i$, then by symmetry assume it is $v_1$.  Now let $F=\{v_1,w\}$. 
To color $T_1$, use $I$ on $z_1$ and $z_4$ and use $F$ on the rest
of $T_1$.  Finally, we can color $T_2$, since it is $F$-odd.  So
instead $z_1$ and $z_4$ must be neighbors of distinct $v_i$.  
By symmetry, we have only two cases:
either (a) $v_1\adj\{z_1,z_2\}$ and $v_2\adj\{z_3,z_4\}$ or (b)
$v_1\adj\{z_1,z_3\}$ and $v_2\adj\{z_2,z_4\}$.
In (a) subset $\tB\cup V(T_1)$ induces the Moser Spindle $M_7$, a contradiction.
In (b), let $F=\{w,v_2\}$ and color $z_1,z_2,z_3,z_4$ as $F,I,F,I$.  Finally,
color $T_2$ by Lemma~\ref{coloring-lem}, since it has exactly 1 $F$-edge.
This completes the case $G[\tB]=P_3$.
\bigskip

Now suppose that $G[\tB]=2K_2$.  Denote the vertices of $\tB$ by
$v_1,v_2,v_3,v_4$, where $v_1\adj v_2$ and $v_3\adj v_4$.  

\begin{clm}
$G[L]$ consists of two trees, $T_1$ and $T_2$.  We may assume that $|T_1|\ge 2$
and each leaf of $T_1$ has both neighbors (outside of $T_1$) in the same
component of $G[\tB]$.  So $T_1$ has at most 4 leaves.
\label{B-component-clm}
\end{clm}
\begin{clmproof}
If $G[L]$ has only a single component, then let $F$ consist of three vertices
in $\tB$.  Now the tree has 9 $F$-edges, so it is $F$-odd, and we are done by
Lemma~\ref{coloring-lem}(i).  Instead assume the forest has two trees,
$T_1$ and $T_2$.  For each $i\in[4]$, let $a_i$ denote the parity of the number
of edges from $v_i$ to $T_1$.  Suppose $a_1\ne a_3$.  Now let $F=\{v_1,v_3\}$.
We are done, since each $T_i$ is $F$-odd.  Thus $a_1=a_3$.  By swapping the
roles of $v_1$ and $v_2$, and also $v_3$ and $v_4$, we get $a_1=a_3=a_2=a_4$. 
By symmetry between $T_1$ and $T_2$, we assume that each $v_i$ has an even
number of edges to $T_1$.
Suppose there exists a leaf $w$ of $T_1$ with at most one neighbor in $\tB$.
(This includes the case that $|T_1|=1$, since $a_1=a_3=a_2=a_4$.)
Let $F$ consist of three vertices in $\tB$, excluding any neighbor of $w$.  
Now we are done, since $T_1$ is $F$-leaf-good, by $w$, and $T_2$ is $F$-odd.
Thus each leaf $w$ of $T_1$ must have both neighbors (outside of $T_1$) in $\tB$.  
Since $\tB$ sends at most 8 edges to $T_1$, we conclude that $T_1$ has at most
four leaves.
If some leaf $w$ of
$T_1$ has neighbors in two components of $G[\tB]$, then we are also done, as
follows.  Let $F$ consist of three vertices in $\tB$, including both neighbors
of $w$.  Again $T_2$ is $F$-odd, so we can color it by
Lemma~\ref{coloring-lem}(i). We can also color $T_1$, by treating $w$ like a
vertex with its two neighbors in $\tB$ colored $I$. Now $T_1$ may contain a path
colored $F$ linking these neighbors of $w$, but it will not extend to a cycle
colored $F$, since the neighbors of $w$ are in different components of $G[\tB]$.
\end{clmproof}

It suffices to color $\tT_1$, since we can
extend the coloring to $T_1$ by coloring each suppressed vertex with $F$.
We show that each vertex of $\tB$ has 2 edges to $T_1$.  (This number is always
either 0 or 2, as we showed just prior to Claim~\ref{clm4.39}.) Recall that
each leaf of $T_1$ has both
neighbors (outside $T_1$) in the same component of $G[\tB]$.
Since $T_1$ has a leaf, its two neighbors in $\tB$ each send two edges to $T_1$.
First suppose they are the only two such vertices in $\tB$.
Recall that each leaf of $T_1$ has its two neighbors in $\tB$ in the same component of
$G[\tB]$; so assume that $v_1$ and $v_2$ both have two edges to $T_1$ and $v_3$
and $v_4$ have none. Note that $T_1\ne K_2$, since $K_4\not\subset G$.  So there
exists $x\in B\setminus \tB$ with a neighbor in $T_1$.  If $x$ sends an odd
number of edges to each $T_i$, then we let $F=\{v_1,v_3,x\}$, and we are done
since each $T_i$ is $F$-odd.  So assume $x$ sends an even number of edges to
each $T_i$.  Now let $F=\{v_2,v_3,v_4,x\}$.  Again $T_2$ is $F$-odd.  And we
can color $T_1$ by Lemma~\ref{helper-lem}, with $x$ as 
helper.  

Now instead suppose that exactly three vertices in $\tB$ each have two edges to
$T_1$; by symmetry, say $v_1,v_2,v_3$.  
Form $\tT_1$ from $T_1$ by suppressing each vertex $w$ such that $d_{T_1}(w)=2$ and
$w$ has no neighbor in $\tB$.  It suffices to color $\tT_1$, since we can
extend the coloring to $T_1$ by coloring each suppressed vertex with $F$.
As is true for $T_1$, each leaf of
$\tT_1$ has both neighbors in the same component of $G[\tB]$, so $\tT_1$ has
only two leaves (that is, $T_1$ and $\tT_1$ are paths).  Denote the vertices of
$\tT_1$ by $z_1,z_2,z_3,z_4$.  So $\{z_1,z_4\}\adj\{v_1,v_2\}$ and
$\{z_2,z_3\}\adj v_3$.  Let $F=\{v_2,v_3,v_4\}$.  Now $T_2$ has five $F$-edges and
so it is $F$-odd.  To
color $T_1$, use $I$ on $z_3$ and use $F$ on $V(T_1)\setminus\{z_3\}$.  Thus,
we conclude that each of the four vertices of $\tB$ sends two edges to $T_1$,
so $|\tT_1|=6$.

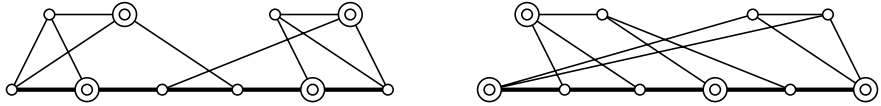
\begin{figure}[!h]
\centering
\begin{tikzpicture}[semithick]  
\tikzset{every node/.style=uStyle}
\draw[ultra thick] (0,0) node (z1) {} -- (1,0) node (z2) {} -- (2,0) node (z3) {} -- (3,0)
node (z4) {} -- (4,0) node (z5) {} -- (5,0) node (z6) {};
\draw (.5,1) node (v1) {} -- (1.5,1) node (v2) {} (3.5,1) node (v4) {} -- (4.5,1) node (v3) {};
\draw (z2) -- (v1) -- (z1) -- (v2) -- (z4) (z3) -- (v3) -- (z6) -- (v4) -- (z5);
\draw (v2) node[uBstyle] {} node {} (v3) node[uBstyle] {} node {} 
(z2) node[uBstyle] {} node {} (z5) node[uBstyle] {} node {};

\begin{scope}[xshift=2.5in]
\draw[ultra thick] (0,0) node (z1) {} -- (1,0) node (z2) {} -- (2,0) node (z3) {} -- (3,0)
node (z4) {} -- (4,0) node (z5) {} -- (5,0) node (z6) {};
\draw (.5,1) node (v1) {} -- (1.5,1) node (v2) {} (3.5,1) node (v3) {} -- (4.5,1) node (v4) {};
\draw (z1) -- (v3) -- (z6) -- (v4) -- (z1) (z5) -- (v2) -- (z4) (z3) -- (v1) -- (z2);
\draw (v1) node[uBstyle] {} node {} (z1) node[uBstyle] {} node {} 
(z4) node[uBstyle] {} node {} (z6) node[uBstyle] {} node {};
\end{scope}

\end{tikzpicture}
\caption{Two of the cases when $\tT_1$ is a 6-vertex path (in the proof of
Lemma~\ref{lem4.38}). Left: $w_2\ne w_3$.
Right: $w_2=w_3$. \label{lem4.38b-fig}}
\end{figure}

Suppose that $\tT_1$ is a path; see Figure~\ref{lem4.38b-fig}). Label its
vertices $z_1,\ldots,z_6$ (from left to right) and let $w_i$ denote the neighbor of $z_i$ in $B$,
for each $i\in\{2,\ldots,5\}$ (possibly the $w_i$ are not distinct).
If $w_2\ne w_3$, then color $\tB$ so that $w_2$ uses $F$, $w_3$ uses $I$, and in
each component of $G[\tB]$ one vertex uses $F$ and the other one uses $I$.  This
implies that $|F\cap \{w_4,w_5\}|=|I\cap \{w_4,w_5\}|=1$, since each leaf has
both neighbors in $\tB$ in the same component of $G[\tB]$.  To extend the
coloring to $\tT_1$, we use $I$ on the vertices $z_i$ and $z_j$ such that
$w_i,w_j\in F$ (and color the other $z_t$ with $F$).  
By symmetry, assume that $v_1,v_3 \in F$.  Because the neighbors of $z_1$ and
$z_6$ are in the same component of $\tB$, the above coloring of $T_1$ satisfies
the conclusion of Lemma~\ref{coloring-lem}; in particular there is no path
between $v_1$ and $v_3$ in $F$.  Thus we can color all of $V(T_2)$ with $F$.

So assume $w_2=w_3$ and (by symmetry) $w_4=w_5$.  Since $z_1$ and $z_6$ have
both neighbors in the same component of $G[\tB]$ (and $G$ is simple), we have
$w_2=w_3\adj w_4=w_5$.  So say $v_1=w_2=w_3$, $v_2=w_4=w_5$, and
$\{z_1,z_6\}\adj\{v_3, v_4\}$.  Let $F=\{v_2,v_3,v_4\}$.  To extend this
coloring to $T_1$, color $z_1,z_4,z_6$ with $I$ and color $z_2,z_3,z_5$ with $F$.
Since $T_2$ is $F$-odd, we can extend the coloring to $T_2$ by 
Lemma~\ref{coloring-lem}.  Thus, we conclude that $T_1$ is not a path.

\begin{figure}[!h]
\centering
\begin{tikzpicture}[semithick, scale=.93]  
\tikzset{every node/.style=uStyle}
\begin{scope}[xshift=-.3in]
\draw[ultra thick] (0,0) node (z1) {} -- (1,0) node (z2) {} -- (2,0) node (z3) {} -- (3,0)
node (z4) {} -- (4,0) node (z5) {} (z4) -- (3,.35) node (z*) {};
\draw (.5,1) node (v1) {} -- (1.5,1) node (v2) {} (2.5,1) node (v3) {} -- (3.5,1) node (v4) {};
\draw (z2) -- (v1) -- (z1) -- (v2) -- (z3) (z*) -- (v4) -- (z5) -- (v3) -- (z*);
\draw (v2) node[uBstyle] {} node {} (z2) node[uBstyle] {} node {} 
(z4) node[uBstyle] {} node {} (v3) node[uBstyle] {} node {};
\end{scope}

\begin{scope}[xshift=2in]
\draw[ultra thick] (0,0) node (z1) {} -- (1,0) node (z2) {} -- (2,0) node (z3) {} -- (3,0)
node (z4) {} -- (4,0) node (z5) {} (z3) -- (2,.35) node (z*) {};
\draw (.5,1) node (v1) {} -- (1.5,1) node (v2) {} (2.5,1) node (v3) {} -- (3.5,1) node (v4) {};
\draw (z1) -- (v1) -- (z5) -- (v2) -- (z1) (z2) -- (v3) -- (z*) -- (v4) -- (z4) (z*) -- (z3);
\draw (v3) node[uBstyle] {} node {} (z1) node[uBstyle] {} node {} 
(z3) node[uBstyle] {} node {} (z5) node[uBstyle] {} node {};
\end{scope}

\begin{scope}[xshift=3.9in]
\draw[ultra thick] (0,0) node (z1) {} -- (1,0) node (z2) {} -- (2,0) node (z3) {} -- (3,0)
node (z4) {} -- (4,0) node (z5) {} (z3) -- (2,.35) node (z*) {};
\draw (.5,1) node (v1) {} -- (1.5,1) node (v2) {} (2.5,1) node (v3) {} -- (3.5,1) node (v4) {};
\draw (z1) -- (v1) -- (z*) -- (v2) -- (z1) (z2) -- (v3) -- (z5) -- (v4) -- (z4);
\draw (v3) node[uBstyle] {} node {} (z*) node[uBstyle] {} node {} 
(z1) node[uBstyle] {} node {} (z4) node[uBstyle] {} node {};
\end{scope}

\end{tikzpicture}

\caption{Three examples when $\tT_1$ is a tree with 3 leaves (in the proof of
Lemma~\ref{lem4.38}).  Left: $z^*\adj z_4$.
Right: $z^*\adj z_3$.\label{lem4.38c-fig}}
\end{figure}
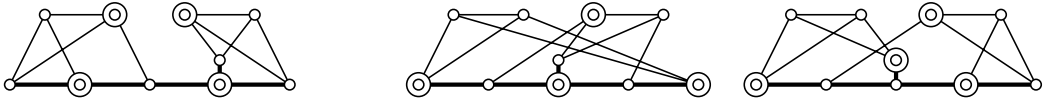

Suppose $\tT_1$ has exactly 3 leaves; see Figure~\ref{lem4.38c-fig}.  Now $\tT_1$ is formed from a 5-vertex path by
adding a pendant edge at one internal vertex.  Denote the vertices of the path
by $z_1,\ldots,z_5$ and the new leaf by $z^*$.  By symmetry, we assume either
$z^*\adj z_4$ or $z^*\adj z_3$.  In the first case, color one vertex in each
component of $G[\tB]$ with $I$ and the other with $F$, so that the neighbor of
$z_3$ is colored $I$.  Now each leaf of $\tT_1$ has one neighbor colored $I$ and
one colored $F$, so $z_2$ has a neighbor colored $F$.  
To extend the coloring to $\tT_1$, color $z_2,z_4$ with $I$ and color $z_1,z_3,z_5,z^*$
 with $F$.  Because the neighbors of the leaves are 
in the same component of $\tB$, the above coloring of $T_1$ satisfies the 
conclusion of Lemma~\ref{coloring-lem}; in particular there is no path between 
vertices of $\tB$ in $F$.  Thus, we can color all of $V(T_2)$ with 
$F$.  This finishes the
case when $z^*\adj z_4$.  So instead assume $z^*\adj z_3$.  Since each leaf has
both neighbors in the same component of $G[\tB]$, also $z_2$ and $z_4$ have
their neighbors in the same component of $G[\tB]$.  By symmetry between $z_1$
and $z_5$, assume this is not the component with vertices adjacent to $z_1$.
Now color the neighbor of $z_2$ in $\tB$ with $I$ and the rest of $\tB$ with
$F$.  We extend this coloring to $T_2$ using Lemma~\ref{coloring-lem},
since $T_2$ is $F$-odd. If $z^*$ has a neighbor colored $I$, then we extend the
coloring to the $z_i$'s by coloring $z_1,z_3,z_5$ with $I$ and coloring
$z_2,z_4,z^*$ with $F$.
Otherwise, only $z_2$ and $z_5$ have neighbors colored $I$, so we color
$z_1,z^*,z_4$ with $I$ and color $z_2,z_3,z_5$ with $F$.  This completes the
case that $\tT_1$ has three leaves.

Finally, suppose $\tT_1$ has exactly 4 leaves; see Figures~\ref{lem4.38d-fig}
and~\ref{lem4.38e-fig}.  Recall that all internal
vertices of $\tT_1$ have degree 3, so $\tT_1$ has two adjacent 3-vertices.
Let $z_1,z_2,z_3,z_4$ denote the leaves of $\tT_1$ with $\{z_1,z_2\}\adj
\{v_1,v_2\}$ and $\{z_3,z_4\}\adj\{v_3,v_4\}$; this follows from
Claim~\ref{B-component-clm}.  In Figures~\ref{lem4.38d-fig}
and~\ref{lem4.38e-fig}, vertices $v_1,v_2,v_3,v_4$ are drawn at top from
left to right.
By symmetry between $v_3$ and
$v_4$, we assume $\dist_{\tT_1}(z_1,z_4)=3$.  Let $z_5$ and $z_6$ denote
(respectively) the neighbors in $\tT_1$ of $z_1$ and $z_4$.  Either
$z_5\adj\{z_1,z_2\}$ (left) or else $z_5\adj\{z_1,z_3\}$ (center and right).  In the first case, let
$F=\{v_2,v_3,v_4\}$.  To extend the coloring to $\tT_1$, use $F$ on $z_1,z_2,z_6$
and use $I$ on $z_3,z_4,z_5$.  (Again $T_2$ is $F$-odd.) So assume we are in
the second case: $z_5\adj\{z_1,z_3\}$.  Suppose some pendant edge of $\tT_1$
corresponds to a path of length at least 2 in $T_1$; by symmetry, say it is
$z_1z_5$ (center).  Let $F=\{v_1,v_2,v_3\}$.  To color $T_1$, use $I$ on $z_1,z_2,z_5$
and use $F$ on $z_3,z_4,z_6$. (Again $T_2$ is $F$-odd.) 
Similarly, suppose $z_5z_6$ corresponds to a path of length at least 2 (right).  
Now let $F=\{v_1,v_3\}$.  Color $z_5,z_6$ with $I$
and color $z_1,z_2,z_3,z_4$ with $F$.  Because there is no
path in $F$ from $v_1$ to $v_3$, we may color $V(T_2)$ with $F$.
Thus, we conclude that $\tT_1=T_1$; see Figure~\ref{lem4.38e-fig}.

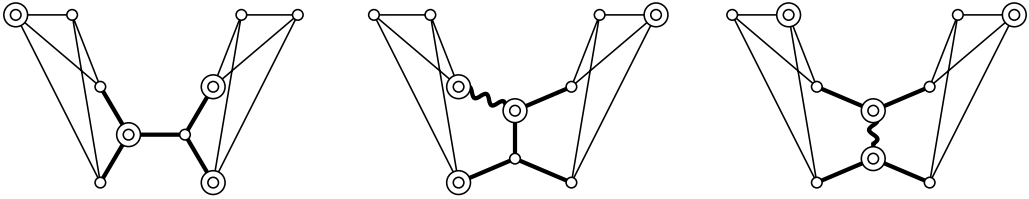
\begin{figure}[h!]
\centering
\begin{tikzpicture} [semithick,xscale = .75, yscale=.75*.85]
\tikzset{every node/.style=uStyle}

\draw (-2.5,2.5) node (v1) {} -- (-1.5,2.5) node (v2) {} (1.5,2.5) node (v3) {} -- (2.5,2.5) node (v4) {};
\draw[ultra thick] (-1,1) node (z1) {} -- (-.5,0) node (z5) {} -- (-1,-1) node (z2) {}
 (1,1) node (z3) {} -- (.5,0) node (z6) {} -- (1,-1) node (z4) {} (z5) -- (z6);
\draw (v1) -- (z1) -- (v2) -- (z2) -- (v1);
\draw (v3) -- (z3) -- (v4) -- (z4) -- (v3);
\draw (v1) node[uBstyle] {} node {} (z3) node[uBstyle] {} node {} (z4)
node[uBstyle] {} node {} (z5) node[uBstyle] {} node {};

\begin{scope}[xshift=2.5in]
\draw (-2.5,2.5) node (v1) {} -- (-1.5,2.5) node (v2) {} (1.5,2.5) node (v3) {} -- (2.5,2.5) node (v4) {};
\draw[ultra thick] (-1,1) node (z1) {}  (0,.5) node (z5) {} -- (1,1) node (z3) {}
 (-1,-1) node (z2) {} -- (0,-.5) node (z6) {} -- (1,-1) node (z4) {} (z5) -- (z6);
\draw (v1) -- (z1) -- (v2) -- (z2) -- (v1);
\draw (v3) -- (z3) -- (v4) -- (z4) -- (v3);
\draw[ultra thick,snake=coil,segment aspect=0, segment amplitude=1.6pt,segment length=7pt] 
(z1) -- (z5); 
\draw (v4) node[uBstyle] {} node {} (z1) node[uBstyle] {} node {} (z2)
node[uBstyle] {} node {} (z5) node[uBstyle] {} node {}; 
\end{scope}

\begin{scope}[xshift=5in]
\draw (-2.5,2.5) node (v1) {} -- (-1.5,2.5) node (v2) {} (1.5,2.5) node (v3) {} -- (2.5,2.5) node (v4) {};
\draw[ultra thick] (-1,1) node (z1) {} -- (0,.5) node (z5) {} -- (1,1) node (z3) {}
 (-1,-1) node (z2) {} -- (0,-.5) node (z6) {} -- (1,-1) node (z4) {};
\draw (v1) -- (z1) -- (v2) -- (z2) -- (v1);
\draw (v3) -- (z3) -- (v4) -- (z4) -- (v3);
\draw[ultra thick,snake=coil,segment aspect=0, segment amplitude=1.6pt,segment length=7pt] 
(z5) -- (z6); 
\draw (v2) node[uBstyle] {} node {} (v4) node[uBstyle] {} node {} (z5)
node[uBstyle] {} node {} (z6) node[uBstyle] {} node {}; 
\end{scope}

\end{tikzpicture}
\caption{If $\tT_1$ has 4 leaves (in the proof of Lemma~\ref{lem4.38}), then
$\tT_1=T_1$ with $z_5\adj\{z_1,z_3\}$ and $z_6\adj\{z_2,z_4\}$.\label{lem4.38d-fig}}
\end{figure}

Suppose some leaf $w$ of $T_2$ has a neighbor in $B\setminus \tB$.  In each
component of $G[\tB]$, color one vertex $F$ and the other $I$; do this so that
any neighbor of $w$ in $\tB$ is colored $I$.  Now $T_2$ is $F$-leaf-good.
By symmetry, we assume that $v_1,v_3\in F$ and
$v_2,v_4\in I$.  For $T_1$, color $z_6$ with $I$ and $z_1,\ldots,z_5$ with $F$.
This does create a $v_1,v_3$-path in $F$ through $T_1$, but this is okay, since
no such path exists in $T_2$.  Thus, each leaf of $T_2$ has no neighbors 
in $B\setminus\tB$.  Since $\tB$ has only 4 edges to $T_2$, we see that $T_2$ is a path.
Suppose a leaf $w$ of $T_2$ has neighbors in distinct components of $G[\tB]$, by
symmetry say $v_1$ and $v_3$.  Now we color $\tB\cup V(T_1)$ as in the
immediately previous case.  We color $w$ with $I$ and $T_2\setminus\{w\}$ with
$F$.  Thus, no such $w$ exists.  Suppose $T_2\ne K_2$.  Color all of $\tB$ with
$F$, color $N(\tB)\cap (T_1\cup T_2)$ with $I$, and color $(T_1\cup
T_2)\setminus N(\tB)$ with $F$.
Thus, we conclude that $T_2=K_2$.  So $G$ is the 12-vertex graph below, which is
nb-critical.  It is forbidden by the hypothesis, which is a contradiction.
This completes the case that $G[\tB]=2K_2$.
\end{proof}

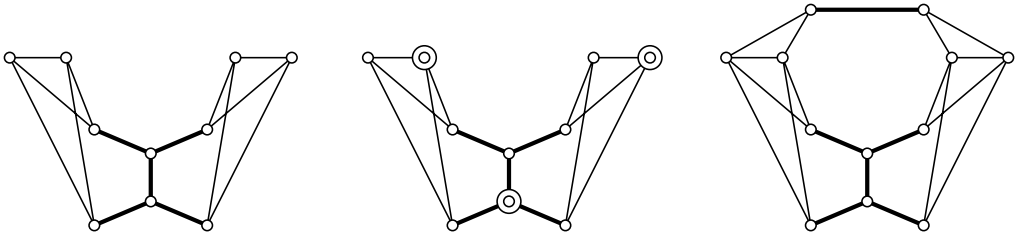
\begin{figure}[!h]
\centering
\begin{tikzpicture} [semithick,xscale = .75, yscale=.75*.85]
\tikzset{every node/.style=uStyle}

\begin{scope}
\draw (-2.5,2.5) node (v1) {} -- (-1.5,2.5) node (v2) {} (1.5,2.5) node (v3) {} -- (2.5,2.5) node (v4) {};
\draw[ultra thick] (-1,1) node (z1) {} -- (0,.5) node (z5) {} -- (1,1) node (z3) {}
 (-1,-1) node (z2) {} -- (0,-.5) node (z6) {} -- (1,-1) node (z4) {} (z5) -- (z6);
\draw (v1) -- (z1) -- (v2) -- (z2) -- (v1);
\draw (v3) -- (z3) -- (v4) -- (z4) -- (v3);
\end{scope}

\begin{scope}[xshift=2.5in]
\draw (-2.5,2.5) node (v1) {} -- (-1.5,2.5) node (v2) {} (1.5,2.5) node (v3) {} -- (2.5,2.5) node (v4) {};
\draw[ultra thick] (-1,1) node (z1) {} -- (0,.5) node (z5) {} -- (1,1) node (z3) {}
 (-1,-1) node (z2) {} -- (0,-.5) node (z6) {} -- (1,-1) node (z4) {} (z5) -- (z6);
\draw (v1) -- (z1) -- (v2) -- (z2) -- (v1);
\draw (v3) -- (z3) -- (v4) -- (z4) -- (v3);
\draw (v2) node[uBstyle] {} node {} (v4) node[uBstyle] {} node {} (z6) node[uBstyle] {} node {}; 
\end{scope}

\begin{scope}[xshift=5in]
\draw[ultra thick] (-1,3.5) node (w1) {}  -- (1,3.5) node (w2) {};
\draw (-2.5,2.5) node (v1) {} -- (-1.5,2.5) node (v2) {} (1.5,2.5) node (v3) {} -- (2.5,2.5) node (v4) {};
\draw[ultra thick] (-1,1) node (z1) {} -- (0,.5) node (z5) {} -- (1,1) node (z3) {}
 (-1,-1) node (z2) {} -- (0,-.5) node (z6) {} -- (1,-1) node (z4) {} (z5) -- (z6);
\draw (v1) -- (z1) -- (v2) -- (z2) -- (v1);
\draw (v3) -- (z3) -- (v4) -- (z4) -- (v3);
\draw (v1) -- (w1) -- (v2) (v3) -- (w2) -- (v4);
\end{scope}

\end{tikzpicture}
\caption{If $\tT_1$ has 4 leaves (in the proof of Lemma~\ref{lem4.38}) and $G$
has no nb-coloring, then $G=J_{12}$.\label{lem4.38e-fig}}
\end{figure}

\begin{lem}
If $e'(B)=1$, then $G$ is near bipartite.
\label{lem4.41}
\end{lem}
\begin{proof}
Suppose $G[\tB]=K_2$.  Denote $\tB$ by $\{v_1,v_2\}$.  If $G[L]$
has only a single component, then color $v_1$ with $F$ and color $v_2$ with $I$.
We can color $G[L]$, since it is $F$-odd.  Suppose instead that $G[L]$ has
two components; call them $T_1$ and $T_2$.  Suppose $v_1$ has 3 edges to $T_1$
(and none to $T_2$).  Let $F=\{v_1\}$.  Now $T_1$ is $F$-odd and $T_2$ is
$F$-null, so we are done.
Thus, by symmetry we assume $v_1$ and $v_2$ each have 1 edge to one tree and 2
edges to the other.  If $v_1$ and $v_2$ have (respectively) 1 and 2 edges to
$T_1$, then let $F=\{v_1,v_2\}$; now both $T_1$ and $T_2$ are $F$-odd.
So assume that $v_1$ and $v_2$ each have 1 edge to $T_1$ and 2 edges to $T_2$.
Suppose some leaf $w$ of $T_2$ has a neighbor $x\in B\setminus \tB$.  Let
$F$ consist of a single vertex of $\tB$ that is not adjacent to $w$.  Now $T_2$
is $F$-leaf-good and $T_1$ is $F$-odd.  Thus, all leaves of $T_2$
have no neighbors in $B\setminus\tB$.  So $T_2$ is a path.  Since
$K_4\not\subset G$, we know $T_2\ne K_2$.  So there exists $x\in
B\setminus\tB$ with a neighbor in $T_2$.  If $x$ sends an odd number of edges to
both $T_1$ and $T_2$, then we let $F=\{v_1,v_2,x\}$, and both $T_1$ and $T_2$
are $F$-odd.  Otherwise, let $F=\{v_1,x\}$.  Again, $T_1$ is
$F$-odd.  Also, we can color $T_2$ by Lemma~\ref{helper-lem}, with $x$ as the
helper.  Thus, we conclude that $G[L]$ has three components; we call these
$T_1$, $T_2$, $T_3$.

We say that $x\in B$ \Emph{splits~as $a_1/a_2/a_3$} if $x$ has $a_i$ edges to
$T_i$, for each $i\in[3]$.  
For $x\in \tB$ we have $a_1+a_2+a_3=3$ and for $x\in B\setminus \tB$, we have
$a_1+a_2+a_3=4$.  If we care only about the parities of
the $a_i$, we say, for example, that $x$ splits as e/o/o\aside{e/o/o} (to denote that $a_1$
is even, while $a_2$ and $a_3$ are odd).  If $v_1$ splits as $1/1/1$ or
as some permutation of $3/0/0$, then let $F=\{v_1\}$.  Now we are done, since $T_1$ is
$F$-odd, while $T_2$ and $T_3$ are both either $F$-odd or $F$-null.  So assume
that $v_1$ (and $v_2$, by symmetry) splits
as some permutation of $2/1/0$.  By symmetry between the $T_i$, we assume that
$v_1$ splits as $2/1/0$.  A priori we have 6 cases for how $v_2$ splits (in increasing
order of difficulty): 
(a) 1/2/0, 
(b) 2/0/1, 
(c) 0/1/2, 
(d) 1/0/2, 
(e) 0/2/1, 
(f) 2/1/0.  
Before considering these cases, we prove an easy claim.

\begin{clm}
If $v_i$ has 2 edges to $T_j$, then $T_j$ is a path with each endpoint adjacent
to $v_i$.
\label{tree-path-clm}
\end{clm}
\begin{clmproof}
Suppose not.  By symmetry we assume that $v_1$ has 2 edges to $T_1$, 1 edge to
$T_2$, and 0 edges to $T_3$, but $T_1$ has a leaf $w$ such that $w\nonadj v_1$.
Let $F=\{v_1\}$.  Now $T_1$ is $F$-leaf-good (by $w$), $T_2$ is $F$-odd, and
$T_3$ is $F$-null.  So we can extend the coloring of $B$ to all of $G$, a
contradiction.
\end{clmproof}

Now we consider cases (a)--(f).
For (a), let
$F=\{v_1,v_2\}$.  Now $T_1$ and $T_2$ are $F$-odd, while $T_3$ is $F$-null.
For (b), Claim~\ref{tree-path-clm} implies that $T_1$ is a path with each
endpoint adjacent to both $v_1$ and $v_2$.  Note that $T_1\ne K_2$, since
$K_4\not\subset G$.  Let $F=\{v_1,v_2\}$ and note that $T_2$ and $T_3$ are both
$F$-odd.  To color $T_1$, use $I$ on both leaves and $F$ everywhere else.  This
finishes (b).  Note that (d) and (e) are the same case, by symmetry
between both the $v_i$'s and the $T_j$'s.  Thus, we must consider cases (c),
(d), and (f).  In all figures for this proof, $v_1$ and $v_2$ are drawn on top;
$T_1$, $T_2$, and $T_3$ are drawn in the middle (from left to right); any
vertices drawn at bottom are in $B\setminus \tB$.

\textbf{Case (c): $\bbs{v_1}$ splits as 2/1/0 and $\bbs{v_2}$ splits as 0/1/2.}
By Claim~\ref{tree-path-clm}, $T_1$ is a path with both endpoints adjacent to
$v_1$; similarly, $T_3$ is a path with both endpoints adjacent to $v_2$.  
See Figure~\ref{lem4.41-fig}.
If
some $x\in B\setminus \tB$ splits as o/e/o, then let $F=\{v_1,x\}$.  Now each
$T_i$ is $F$-odd, so we are done.  
Suppose some $x\in B\setminus\tB$ splits as e/e/e; we consider the possibilities.
If $x$ splits as
0/0/4, then let $F=\{v_2,x\}$.  Now $T_1$ is $F$-null, $T_2$ is $F$-odd, and we
can color $T_3$ by Lemma~\ref{helper-lem}, with $x$ as helper.  So $x$
cannot split as 0/0/4; similarly, $x$ cannot split as 4/0/0.  If $x$ splits as
0/2/2, then let $F=\{v_2,x\}$. Now $T_1$ is $F$-null and $T_2$ is $F$-odd.  To
color $T_3$, use $I$ on one neighbor of $x$ and color the rest of $T_3$ with
$F$.  So assume no vertex splits as 0/2/2; similarly, no vertex splits as
2/2/0.  Thus each vertex that splits as e/e/e splits as 2/0/2 or 0/4/0.
If instead there exist $x,y\in B\setminus\tB$ that split (respectively)
as o/o/e and e/o/o, then let $F=\{v_1,x,y\}$.  Again, each $T_i$ is $F$-odd, so
we are done.  By symmetry (between $T_1$ and $T_3$) we assume that no vertex in
$B\setminus\tB$ splits as e/o/o.  Hence, every vertex splits as o/o/e or 2/0/2
or 0/4/0.

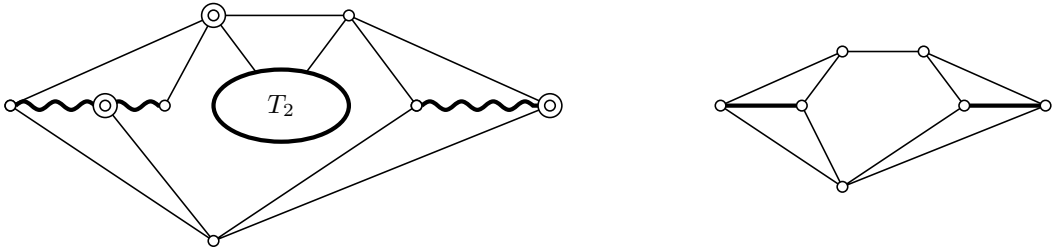
\begin{figure}[!h]
\centering
\begin{tikzpicture} [semithick,scale=1.2, xscale = .75]
\tikzset{every node/.style=uStyle}

\draw (-4,0) node (v1) {} (-1.72,0) node (v2) {} (2,0) node (v3) {} (3.97,0) node (v4) {}
(-1,1) node (z1) {} (1,1) node (z2) {} (-1,-1.5) node (w) {} (-2.6,0) node (y) {};
\draw (v1) -- (z1) -- (v2) (v3) -- (z2) -- (v4) (0,0) -- (z1) -- (z2) -- (0,0)
(v1) -- (w) -- (y) (v3) -- (w) -- (v4);
\draw[fill=white,ultra thick] (0,0) ellipse (1 cm and .4cm);
\draw (0,0) node[lStyle] {$T_2$};
\draw[ultra thick,snake=coil,segment aspect=0, segment amplitude=1.6pt,segment length=12pt] 
(v1) -- (v2) (v3) -- (v4); 
\draw (y) node[uBstyle] {} node {} (z1) node[uBstyle] {} node {} (v4) node[uBstyle] {} node {} (v1) node {} (v3) node {};

\begin{scope}[xshift=3.5in,scale=.6]
\draw (-4,0) node (v1) {} (-2,0) node (v2) {} (2,0) node (v3) {} (4,0) node (v4) {} 
(-1,1) node (z1) {} (1,1) node (z2) {} (-1,-1.5) node (w) {}; 
\draw (v1) -- (z1) -- (v2) (v3) -- (z2) -- (v4) (z1) -- (z2) (w) -- (v2)
(v1) -- (w) (v3) -- (w) -- (v4);
\draw[ultra thick] (v1) -- (v2) (v3) -- (v4);
\end{scope}

\end{tikzpicture}
\caption{Case (c) in the proof of Lemma~\ref{lem4.41}.\label{lem4.41-fig}}
\end{figure}

We consider the possibilities for a vertex $x\in B\setminus \tB$ that splits as o/o/e.  
If $x$ splits as 1/1/2, then let $F=\{v_1,v_2,x\}$.  Trees $T_1$ and $T_2$ are
both $F$-odd, and we can color $T_3$ by
Lemma~\ref{helper-lem}, with $x$ as helper.  So each $x\in B\setminus\tB$ must
split as 1/3/0, 3/1/0, 0/4/0, or 2/0/2.  Since $T_3$ has a
neighbor in $B\setminus\tB$, some $x\in B\setminus\tB$ splits as 2/0/2.
Suppose some $y$ splits as 1/3/0 or 3/1/0.  Let $F=\{v_1,v_2,x,y\}$.  Trees
$T_1$ and $T_2$ are $F$-odd, and we can color $T_3$ by Lemma~\ref{helper-lem},
with $x$ as helper.
So assume no such $y$ exists.  That is, each vertex splits as 2/0/2 or 0/4/0.
Recall that $x$ splits as 2/0/2, and suppose that $x$ has a neighbor $z$ that is
not a leaf of $T_1$ or $T_3$.  By symmetry, say $z\in T_1$.  Let $F=\{v_2,x\}$.
To color $T_1$, use $I$ on $z$ and $F$ on the rest of $T_1$.  To color $T_3$,
use $I$ on a neighbor of $x$ (and $F$ on the rest of $T_3$).  Finally, $T_2$ is
$F$-odd.  So assume that no such $z$ exists.  This implies that $x$ is unique.
So $T_1=K_2$ and $T_3=K_2$.  But now $\{v_1,v_2,x\}\cup V(T_1)\cup V(T_2)$
induces a Moser spindle, which is a contradiction.  This finishes Case (c).

\textbf{Case (d): $\bbs{v_1}$ splits as 2/1/0 and $\bbs{v_2}$ splits as 1/0/2.}  
By Claim~\ref{tree-path-clm} $T_1$ is a path with both endpoints adjacent to
$v_1$ and $T_3$ is a path with both endpoints adjacent to $v_2$.
Consider some vertex $x\in B\setminus \tB$ and the parities of edges
that $x$ has to $T_1$, $T_2$, and $T_3$.  A priori, the options are o/o/e,
o/e/o, e/o/o, and e/e/e.  If $x$ splits as e/o/o, then let $F=\{v_2,x\}$.
Now each $T_i$ is $F$-odd, so we are done.  Similarly, if $x$
splits as o/e/o, then let $F=\{v_1,x\}$.  So assume each vertex in
$B\setminus \tB$ splits as o/o/e or e/e/e.  

Suppose $T_3\ne K_2$, as on the left of Figure~\ref{lem4.41b-fig}.  
Let $x\in B\setminus\tB$ be a neighbor of some internal vertex $y$ of
$T_3$.  Suppose $x$ splits as e/e/e.  Let $F=\{v_1,v_2,x\}$.  Note that
$T_1$ and $T_2$ are $F$-odd.  To color $T_3$, we use Lemma~\ref{helper-lem},
with $x$ as helper.  So assume instead that $x$ splits as o/o/e.  (Since $x$
sends edges to $T_3$, it splits as 1/1/2.)  Let $F=\{v_1,x\}$, and note that
$T_1$ is $F$-odd.  Color $y$ with $I$ and color the
rest of $T_3$ with $F$. Finally, $T_2$ is $F$-even.
We color all of $T_2$ with $F$.  This creates a single $v_1,x$-path colored $F$
in $T_2$, but this is okay since neither $T_1$ nor $T_3$ has such a path.
This implies that $T_3=K_2$, as on the right of Figure~\ref{lem4.41b-fig}.  

\begin{figure}[!h]
\centering

\begin{tikzpicture} [semithick,scale=1.2, xscale = .75, yscale=.9]
\tikzset{every node/.style=uStyle}

\draw (-4,0) node (v1) {} (-1.72,0) node (v2) {} (1.8,0) node (v3) {} (4.12,0) node (v4) {}
(-1,1) node (z1) {} (1,1) node (z2) {} (-1,-1.5) node (w) {} 
(-3.4,0) node (y1) {} (-2.85,0) node (y2) {} (2.54,0) node (y3) {} (3.55,0) node (y4) {}; 
\draw (v1) -- (z1) -- (v2) (v3) -- (z2) -- (v4) (0,0) -- (z1) -- (z2) -- (y2)
(y1) -- (w) -- (0,0) (y3) -- (w) -- (y4);
\draw[fill=white,ultra thick] (0,0) ellipse (1 cm and .4cm);
\draw (0,0) node[lStyle] {$T_2$};
\draw[ultra thick,snake=coil,segment aspect=0, segment amplitude=1.6pt,segment
length=12pt] 
(v1) -- (v2) (v3) -- (v4); 
\draw (y2) node[uBstyle] {} node {} (z1) node[uBstyle] {} node {} (v4) node[uBstyle] {} node {} (y1) 
node {} (v1) node {} (v3) node {} (y3) node {} (y4) node {};

\begin{scope}[xshift=3.5in]
\draw (-4,0) node (v1) {} (-1.72,0) node (v2) {} (2.5,0) node (v3) {} (4,0) node (v4) {}
(-1,1) node (z1) {} (1,1) node (z2) {} (-1,-1.5) node (w) {} (-3.4,0) node (y1) {} (-2.85,0) node (y2) {};
\draw (v1) -- (z1) -- (v2) (v3) -- (z2) -- (v4) (0,0) -- (z1) -- (z2) -- (y2)
(y1) -- (w) -- (0,0) (v3) -- (w) -- (v4);
\draw[fill=white,ultra thick] (0,0) ellipse (1 cm and .4cm);
\draw (0,0) node[lStyle] {$T_2$};
\draw[ultra thick,snake=coil,segment aspect=0, segment amplitude=1.6pt,segment length=12pt] 
(v1) -- (v2); 
\draw[ultra thick] (v3) -- (v4);
\draw (y2) node[uBstyle] {} node {} (z1) node[uBstyle] {} node {} (v4) node[uBstyle] {} node {} (y1) node {}
(v1) node {};
\end{scope}

\end{tikzpicture}
\caption{Case (d), part 1, in the proof of Lemma~\ref{lem4.41}.
Left: $T_3\ne K_2$.  Right: $T_3=K_2$.\label{lem4.41b-fig}}
\end{figure}
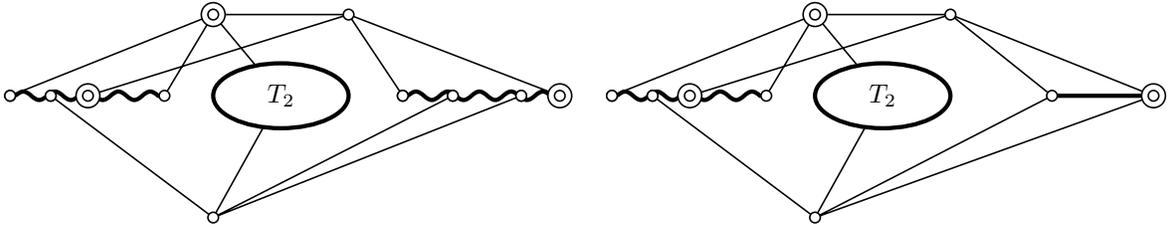

Let $x$ be a neighbor of $T_3$ other than $v_2$.  If $x$ splits as e/e/e,
then the argument in the previous paragraph still works.  So assume $x$ splits
as o/o/e, that is, as 1/1/2.

Suppose either $x$ or $v_2$ has a neighbor $z$ in $T_1$ that is not a leaf of
$T_1$.  Let $F=\{v_2,x\}$.  Note that $T_2$ is $F$-odd.  To color $T_1$, we use
$I$ on $z$ and use $F$ on the rest of $T_1$.  (Note that $x$ and $v_2$ each
have only a single neighbor in $T_1$, and one of these neighbors, $z$, is
colored $I$, so $T_1$ has
no $v_2,x$-path in $F$.)  To extend to $T_3$, we color one of its vertices with
$I$ and the other with $F$.  Thus, no such $z$ exists.  That is, $N_{T_1}(v_2,x)$
is simply the two leaves of $T_1$; see Figure~\ref{lem4.41c-fig}.

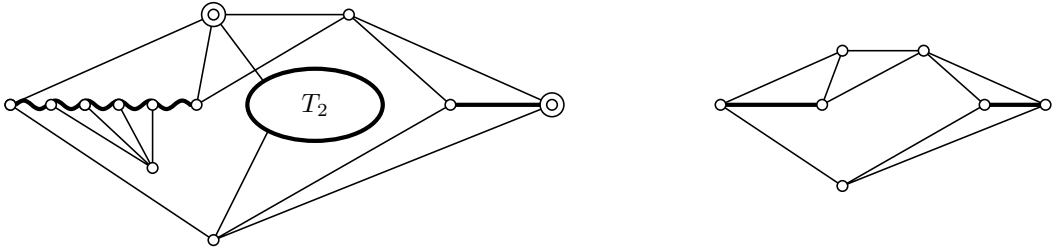
\begin{figure}[!hb]
\centering
\begin{tikzpicture} [semithick,scale=1.2, xscale = .75]
\tikzset{every node/.style=uStyle}
\tikzstyle{lStyle}=[shape = rectangle, minimum size = 0pt, inner sep =
0pt, outer sep = 0pt, draw=none]

\draw (-4,0) node (v1) {} (-1.25,0) node (v2) {} (2.5,0) node (v3) {} (4,0) node (v4) {} 
(-1,1) node (z1) {} (1,1) node (z2) {} (-1,-1.5) node (w) {} (-1.9,-.7) node (x) {}
(-3.4,0) node[lStyle] (y1) {} (-2.9,0) node[lStyle] (y2) {} (-2.4,0)
node[lStyle] (y3) {} (-1.9,0) node[lStyle] (y4) {};
\draw (v1) -- (z1) -- (v2) (v3) -- (z2) -- (v4) (0,0) -- (z1) -- (z2) -- (v2)
(v1) -- (w) -- (0,0) (v3) -- (w) -- (v4) (y1) -- (x) -- (y2) (y3) -- (x) -- (y4);
\draw[fill=white,ultra thick] (0.5,0) ellipse (1 cm and .4cm);
\draw (.5,0) node[lStyle] {$T_2$};
\draw[ultra thick,snake=coil,segment aspect=0, segment amplitude=1.6pt,segment
length=12pt] 
(v1) -- (v2); 
\draw[ultra thick] (v3) -- (v4);
\draw (y2) node {} (z1) node[uBstyle] {} node {} (v4) node[uBstyle] {} node {} (y1) node {}
(y2) node {} (y3) node {} (y4) node {} (v1) node {};

\begin{scope}[xshift=3.5in,scale=.6]
\draw (-4,0) node (v1) {} (-1.5,0) node (v2) {} (2.5,0) node (v3) {} (4,0) node (v4) {} 
(-1,1) node (z1) {} (1,1) node (z2) {} (-1,-1.5) node (w) {}; 
\draw (v1) -- (z1) -- (v2) (v3) -- (z2) -- (v4) (z1) -- (z2) -- (v2)
(v1) -- (w) (v3) -- (w) -- (v4);
\draw[ultra thick] (v1) -- (v2) (v3) -- (v4);
\end{scope}

\end{tikzpicture}
\caption{Case (d), part 2, in the proof of Lemma~\ref{lem4.41}.\label{lem4.41c-fig}}
\end{figure}

Suppose that $T_1\ne K_2$, and let $y$ be a neighbor of $T_1$ in $B\setminus
(\tB \cup \{x\})$.
Recall that each vertex in $B\setminus \tB$ splits as o/o/e or e/e/e.  If $y$
splits as o/o/e, then let $F=\{v_1,v_2,x,y\}$.  Note that $T_1$ and $T_2$ are
both $F$-odd.  Although $T_3$ is $F$-even, we simply color one of its vertices
with $I$ and the other with $F$.  So instead assume that $y$ splits as e/e/e.
Since $T_3$ is $K_2$, vertex $y$ sends no edges to $T_3$.  We let
$F=\{v_2,x, y\}$.  Now $T_2$ is $F$-odd, and $T_3$ is again easy to color. Since
$T_1$ is $F$-even, we color it by Lemma~\ref{helper-lem}, with $y$ as helper.
So we conclude that no such $y$ exists.  That is, $T_1=K_2$.  Now
$\{v_1,v_2,x\}\cup V(T_1)\cup V(T_3)$ induces a Moser spindle, which is a
contradiction.  This finishes case (e).

\textbf{Case (f): $\bbs{v_1}$ splits as 2/1/0 and $\bbs{v_2}$ splits as 2/1/0.}
By Claim~\ref{tree-path-clm}, $T_1$ is a path with each endpoint adjacent to
both $v_1$ and $v_2$.  Note that $T_1\ne K_2$, since $K_4\not\subset G$. 
If $T_2$ has a leaf adjacent to
neither $v_1$ nor $v_2$, then let $F=\{v_1,v_2\}$.  Now $T_2$ is $F$-leaf-good
and $T_3$ is $F$-null.  Since $T_1\ne K_2$, we can color both
leaves of $T_1$ with $I$ and its internal vertices with $F$.
So $T_2$ is a path with each leaf adjacent to one of $\{v_1,v_2\}$.  We consider
a vertex $x\in B\setminus \tB$ and the possible ways it splits.
If $x$ splits as
o/e/o, then let $F=\{v_1,x\}$.  Now each $T_i$ is $F$-odd, so we are done.
The other possibilities for the way that $x$ splits are 
1/3/0, 3/1/0, 0/1/3, 0/3/1, 1/1/2, 2/1/1, 4/0/0, 0/4/0, 0/0/4, 2/2/0, 2/0/2, 0/2/2.  
If $x$ splits as
1/3/0 or 3/1/0, then let $F=\{v_1,v_2,x\}$.  Now $T_1$ and $T_2$ are $F$-odd,
and $T_3$ is $F$-null.
If $x$ splits as 0/1/3 or 0/3/1, then let $F=\{v_1,v_2,x\}$.  Now $T_2$ and
$T_3$ are $F$-odd.  To color $T_1$, use $I$ on its two leaves and use $F$
elsewhere.  If $x$ splits as 4/0/0, then let $F=\{x,v_1\}$.  Now $T_2$ is
$F$-odd and $T_3$ is $F$-null.  We color $T_1$ by Lemma~\ref{helper-lem}, with
$x$ as helper.  If $x$ splits as 0/4/0, then let $F=\{v_1,v_2,x\}$.  Now $T_3$
is $F$-null.  To color $T_1$, use color $I$ on its leaves and use $F$
elsewhere.  To color $T_2$, use Lemma~\ref{helper-lem}, with $x$ as
helper.  If $x$ splits as 2/2/0, then let $F=\{v_1,x\}$.  Note that $T_3$ is
$F$-null and $T_2$ is $F$-odd.  To color $T_1$, use $I$ on one neighbor of $x$
in $T_1$, and use $F$ on the rest of $T_1$.  
Suppose that $x$ splits as 2/1/1.  By symmetry between $v_1$ and $v_2$, assume that
$v_1$ and $x$ do not dominate all leaves in $T_2$.  Now let $F=\{v_1,x\}$.
Clearly, $T_3$ is $F$-odd, and $T_2$ is $F$-leaf-good.  For $T_1$, color one
neighbor of $x$ in $T_1$ with $I$ and color the rest of $T_1$ with $F$.
We have handled all possibilities for the way $x$ splits except 1/1/2, 2/0/2,
0/2/2, and 0/0/4. 

Suppose $T_3$ is not a path (so it has at least three leaves).  Since $T_1\ne
K_2$, there exists $x\in B\setminus \tB$ that splits as either 1/1/2 or else 2/0/2. 
In the first case, let $F=\{v_1,v_2,x\}$.  Trees $T_1$ and $T_2$ are both
$F$-odd.  And $T_3$ is $F$-leaf-good, so we are done.  In the second case, let
$F=\{v_1,x\}$.  Again $T_3$ is $F$-leaf good, and $T_2$ is $F$-odd.  We color
$T_1$ by Lemma~\ref{helper-lem}, with $x$ as helper.  

\begin{figure}[!hb]
\centering
\begin{tikzpicture} [semithick,scale=1.2, xscale = .5]
\tikzset{every node/.style=uStyle}

\draw (-4,0) node (v1) {} (-2,0) node (v2) {} (-.8,0) node (v3) {} (.8,0) node (v4) {}
(-1,1) node (z1) {} (1,1) node (z2) {};
\draw (v1) -- (z1) -- (v2) (v1) -- (z2) -- (v2) (v3) -- (z1) -- (z2) -- (v4);
\draw[ultra thick,snake=coil,segment aspect=0, segment amplitude=1.2pt,segment length=7pt] 
(v1.center) -- (v2.center) (v3.center) -- (v4.center) (2,0) node {} -- (4,0) node{}; 
\draw (v1) node {} (v2) node {} (v3) node {} (v4) node {};

\begin{scope}[xshift=4.4in]
\draw (-4,0) node (v1) {} (-2,0) node (v2) {} (-.7,0) node (v3) {} (.75,0) node (v4) {}
(2.33,0) node (v5) {} (3.67,0) node (v6) {}
(-1,1) node (z1) {} (1,1) node (z2) {} (-1,-1) node (x1) {} (1,-1) node (x2) {}
(-3.33,0) node (v1a) {} (-2.67,0) node (v2a) {} ;
\draw (v1) -- (z1) -- (v2) (v1) -- (z2) -- (v2) (v3) -- (z1) -- (z2) -- (v4)
(v1a) -- (x1) -- (v5) -- (x2) -- (v2a) (x1) -- (v6) -- (x2) (x1) -- (v3) (x2) -- (v4);
\draw[ultra thick] (v1) -- (v1a) -- (v2a) -- (v2) (v3) -- (v4) (v5) -- (v6); 
\draw (v2a) node[uBstyle] {} node {} (v4) node[uBstyle] {} node {} (v6)
node[uBstyle] {} node{} (z1) node[uBstyle] {} node {};
\end{scope}

\end{tikzpicture}
\caption{Case (f), in the proof of Lemma~\ref{lem4.41}.\label{lem4.41d-fig}}
\end{figure}
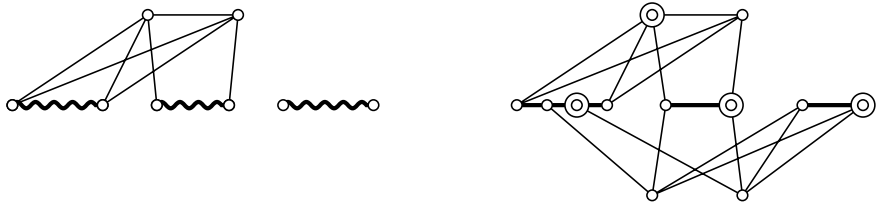

So assume $T_3$ is a path; see Figure~\ref{lem4.41d-fig}. 
Suppose some $y$ splits as 0/0/4.  Since $T_1\ne K_2$, some $x$ splits as 1/1/2
or 2/0/2.  If $x$ is not adjacent to both leaves of $T_3$, then we
can ignore $y$ and repeat the argument that starts this paragraph.  
If $x$ splits as 1/1/2, then let $F=\{v_1,v_2,x,y\}$, so that $T_1$ and $T_2$
are each $F$-odd, and color $T_3$ by Lemma~\ref{helper-lem}, with $y$ as helper.  
If $x$ splits as 2/0/2, then let $F =\{v_1,x,y\}$, so that
$T_2$ is $F$-odd, $T_1$ can be handled by coloring one neighbor of $x$ with $I$
(and the rest with $F$), and $T_3$ can be colored by Lemma~\ref{helper-lem},
with $y$ as helper.  Thus, no such $y$ exists.
Now we are down to three ways that vertices in $B\setminus\tB$ split: 1/1/2, 2/0/2, 0/2/2.

Suppose some $x$ splits as 2/0/2 and some $y$ splits as 1/1/2.  By the previous
paragraph, they must both be adjacent to both leaves of $T_3$.  Now let
$F=\{v_1,v_2,x,y\}$.  Trees $T_1$ and $T_2$ are both $F$-odd.  For $T_3$, we
color one leaf with $I$ and the rest of $T_3$ with $F$.  This implies that
vertices split as exactly one of the ways 2/0/2 and 1/1/2 (since $T_1\ne K_2$).
 Suppose $x$ splits as 2/0/2.  Since $T_2$ has more than two incident edges,
some $y$ splits as 0/2/2.  Let $F=\{v_1,x,y\}$.  Note that $T_2$ is $F$-odd.
Use $I$ to color a neighbor of $x$ in $T_1$ and a neighbor of
$y$ in $T_3$.  So no vertex splits as 2/0/2.

Since $T_1\ne K_2$, some vertex $x$ splits as 1/1/2.  If $x$ is not
adjacent to both leaves of $T_3$, then let $F=\{v_1,v_2,x\}$.  Now $T_1$ and
$T_2$ are $F$-odd, and $T_3$ is $F$-leaf-good.  So assume $x$ is adjacent to
both leaves of $T_3$.
Suppose there exists $y$ of type 0/2/2.  Let $F=\{v_1,v_2,x,y\}$.  Again, $T_1$
and $T_2$ are $F$-odd.  To color $T_3$, use $I$ on a neighbor of $y$, and use
$F$ elsewhere.  So no such $y$ exists.  That is, all vertices in $B\setminus
\tB$ are type 1/1/2.  Further, each is adjacent to both leaves of $T_3$, so
exactly two such vertices exist.  Thus, $T_3=K_2$ and $T_2=K_2$ and $T_1=P_4$. 
This implies that
$|G|=|T_1|+|T_2|+|T_3|+|\tB|+2=4+2+2+2+2=12$.  There is exactly one possibility
for $G$.  It is shown on the right in Figure~\ref{lem4.41d-fig}, along with an
nb-coloring.  This finishes Case (f), finishes the larger case that
$G[\tB]=K_2$, and completes the proof of (B) in our Main Theorem.
\end{proof}

\section{Algorithmic Details}
\label{algorithm section}
\label{alg-sec}

Section~\ref{proof-sec} contains two types of assertions: (i) graphs of a
certain form are near-bipartite and (ii) graphs of a certain form do not
satisfy the assumptions of the Main Theorem.  To prove each assertion of type
(i), we find an nb-coloring.  So our proof is constructive, and 
naturally yields an algorithm.  In this section we detail the efficiency
of this algorithm.
We assume the graph is stored as a list of vertices, and that each vertex stores
a list of incident edges, multiedges, edge-gadgets, and its precoloring (if this
exists).

Let $T_m(n)$ denote the maximum running time of the algorithm on a multigraph with $n$
vertices, and let $T_s(n)$ denote the corresponding function for simple graphs.
As before we write $T_*(n)$ in statements that hold for both $T_m(n)$ and $T_s(n)$.
Our algorithm is recursive, so our upper bound on $T_*(n)$ is in terms of $T_*(n-1)$.
We use the crude estimate $\sum_{i = 1}^{n} i^d \leq \int_{1}^{n+1} x^d dx \leq
n^{d+1}$ for sufficiently large $n$ and $d>1$.
Thus, to prove $T_*(n) \leq O(n^{d+1})$ it suffices to show that $T_*(n) \leq T_*(n-1) + O(n^d)$.
When $G$ contains a vertex set $W$ with $\rho_*(W)$ small, we first color $G[W]$
and second color $G'$, formed from $G$ by contracting $W$ down to two vertices.
That is, the algorithm recurses on two graphs $G[W]$ and $G'$, which satisfy $|V(G[W])|
+ |V(G')| = |\VG| +2$. (This case arises in the proofs of our gap lemmas.)
Simple calculus shows that $(n+2-k)^d + k^d$, with $3 \le k \le n-1$, is maximized when $k \in \{3,n-1\}$.
So if $T_*(j) \leq c j^{d+1}$ for all $j<n$ and some fixed $c$, then $\max_{3
\leq k \leq n-1} \{T_*(n+2-k) + T_*(k)\} \leq c (n-1)^{d+1} + O(n^d)$.
Hence, to prove $T_*(n) \leq O(n^{d+1})$ it also suffices to prove that $T_*(n)
\leq O(n^d) + \max_{3 \leq k \leq n-1} \{T_*(n+2-k) + T_*(k)\}$.
So in the individual steps below we focus on the time to construct the recursive
calls, and extend the colorings afterward.  Only after listing all steps do we
account for the time spent on the recursive calls.

We assume that every graph with at most $30$ vertices can be nb-colored in
time $O(1)$, if it has an nb-coloring.  We also assume that we can iterate
through each graph in $\HH$ in time $O(1)$.  Since each graph in $\HH$ has
at most $22$ vertices, we can determine whether a given pair of vertices is
linked in a graph with order $n$ by a graph in $\HH$ in time $O(n^{20})$.
In practice this can be done much faster, since we only need to consider connected subgraphs.

We start with Part (A) of the Main Theorem.  Let $G$ be an input graph with $n$ vertices.
We assume that $G$ satisfies the hypotheses of the Main Theorem, so $|E(G)|=O(n)$.
We list in order the steps of the algorithm.
Each step except the last describes how to color the graph if it satisfies certain conditions.
Each step assumes that the conditions of the previous steps fail to hold.
We will show that $T_m(n)=O(n^6)$.

\begin{enumerate}
\item \emph{$G$ is disconnected.}  We recurse on each component.  Determining
the components of a graph can be done by breadth first search in time $O(n\log(n))$ since $|E(G)|=O(n)$.
\item \emph{$G$ contains a vertex $v$ satisfying at least one of the following conditions:
$d(v)=1$, $v$ is precolored $I$, $|N(v)|=1$, or $d(v)=2$ and $v$ is uncolored.}
Each of these criteria can be tested in time $O(n)$.  If any criterion is satisfied,
then we apply the proof of Lemma~\ref{easy cases}, \ref{remove I}, \ref{min degree
2}, or \ref{uncolored is min degree 3}.  Constructing the graph to recurse on
takes time $O(n)$; extending the coloring takes time $O(1)$.
\item \emph{$G$ contains a proper non-trivial vertex subset $W$ with 
$\rho_{*,G}(W) \leq 0$.}  We find a subset $W$ with smallest potential, and
among them choose one with largest order (so $W\ne \emptyset$).  
By Corollary~\ref{find small potential 3} with $m_1 = 0$ and $m_2 = 1$, this takes
time $O(n^3\log(n))$.  We recurse on $G[W]$, and then construct $G'$ as
in the proof of Lemma~\ref{small subgraphs part 1}.  Constructing $G'$ takes time $O(n)$.  
Merging the two colorings takes time $O(n)$.  So the total time for
these steps is $O(n) + O(n) + O(n^3\log(n)) \leq O(n^4)$. 
%
\item \emph{$G$ contains a vertex subset $W$ with $\rho_{m,G}(W) = 1$ and $1
\leq |W| \leq n-1$.}  We use the same operations as in the previous step, but
apply Corollary~\ref{find small potential 3} with $m_1 = 1, m_2 = 1$.  Our
running time is now $O(n^5)$.
\item \emph{$G$ contains a vertex $v$ satisfying at least one of the following
conditions: $d(v)=2$, $v$ is precolored $F$, $v$ is incident to a multiedge, or
$v$ has neighbors that are adjacent.}
The first three criteria can be tested in time $O(n)$; the last in time
$O(n^3)$.  We apply the proof of Lemma~\ref{min degree 3 multi}, \ref{no precolor
multi}, \ref{no multi edges}, or \ref{no3-cycle-lem}.  Constructing the graph
to recurse on takes time $O(1)$; extending the coloring also takes time $O(1)$.
\item \emph{We apply the proof of Lemma~\ref{finish of multi}.}  
Constructing the graph to recurse on takes time $O(1)$; extending the coloring
also takes time $O(1)$.
\end{enumerate}

In each step above, the time spent on pre- and post-processing the recursive
calls is $O(n^5)$, and the time for the recursion is
$\max\{T_m(n-1), \max_{3 \leq k \leq n-1} \{T_m(n+2-k) + T_m(k)\}\}$. 
Thus, we have $T_m(n) \leq O(n^5) + 
\max\{T_m(n-1), \max_{3 \leq k \leq n-1}
\{T_m(n+2-k) + T_m(k)\}\}$. 
So $T_m(n) \leq O(n^6)$.
\smallskip

We now consider Part (B) of the Main Theorem.
Since we merged arguments in Section~\ref{common work}, the first three steps are the same;
so we omit them below. 
Before we list the algorithm's steps, we note that by
the start of Section~\ref{simple proof sec 1} (where we begin after skipping
the common three steps), we have proved Lemma~\ref{links are special}:
If two vertices in $G$ are linked, then they are specially-linked (and the
linking graph is in $\HH$).
So we can decide if a given pair of vertices is linked in time $O(n^{20})$.
Also note that $G(C, z_1, z_2)$ can be constructed in time $O(|C|)$.
Let $L$ denote the set of uncolored vertices of degree $3$ with no incident edge-gadgets.
Note that applying the arguments of Section~\ref{coloring-lems-sec} takes time $O(n)$.
As above, for each step we focus on the pre- and post-processing time.  Only at
the end do we consider the time for the recursion.
We will show that $T_s(n)\le O(n^{22})$.

\begin{enumerate}
\item[4.] \emph{$G[L]$ contains an induced cycle $C$ of length $3$ or $4$.}  
We can find $C$ in time $O(|L|^4) \leq O(n^4)$.  If $C$ has length $4$, then we
apply Lemma~\ref{no-short-cycle-of-3s}.  Constructing the graph to recurse on takes time $O(1)$; 
extending the coloring also takes time $O(1)$.  If $C$ has length $3$, then we
must find a pair of vertices $z_1,z_2$ in $N(C)$ that are not linked.  We
check ${3 \choose 2}$ pairs, which takes total time $O(n^{20})$. 
Constructing $G(C, z_1, z_2)$ as in Lemma~\ref{no-short-cycle-of-3s} (the graph we
recurse on) takes time $O(1)$; extending the coloring also takes time $O(1)$. 
\item[5.] \emph{$G$ contains a vertex subset $W$ with $\rho_{m,G}(W) <
4$ and $1 \leq |W| \leq n-2$.}
We perform the same operations as in step 3 above, but apply
Corollary~\ref{find small potential 3} with $m_1 = 1, m_2 = 2$.  
The running time is now $O(n^6)$.
\item[6.] \emph{$G[L]$ contains an induced cycle $C$ of length $5$.}  
We can find $C$ in time $O(|L|^5) \leq O(n^5)$. We perform the same operations
as in step 4 above, but now we check ${5 \choose 2}$ vertex pairs.
\item[7.] \emph{$G$ contains a vertex $v$ with $d(v)=2$.}  We can
find $v$ in time $O(n)$. We apply the proof of Lemma \ref{min degree 3 simple}. 
Note that Case 3 of Lemma~\ref{min degree 3 simple} (where the neighbors
of $v$ are linked) implies that $V(G)=V(H)\cup\{v\}$.  So $|V(G)|\le 23$.  
We assumed above that $n \geq 30$, so we can construct the graph to recurse on
in time $O(1)$; extending the coloring also takes time $O(1)$.
\item[8.] \emph{$G[L]$ contains an induced cycle $C$.}  Now $C$ can be found in
time $O(|L|) \leq O(n)$.  We perform the same operations as in step 4 above,
but with Lemma~\ref{no cycle of 3s} instead of
Lemma~\ref{no-short-cycle-of-3s}.  We only need to check for non-linked pairs
of vertices among neighbors of consecutive members of $C$, so we only check
$|C|-1$ pairs.  Since $|C| \leq n-1$, this step runs in time
$O(n*n^{20})=O(n^{21})$.  
\item[9.] \emph{$G$ contains a vertex $v$ that satisfies at least one of the
following: $d(v)=5$, $v$ is precolored, or $v$ is incident to an edge-gadget.}
Each of these criteria can be tested in time $O(n)$.  We apply the proof of
Lemma~\ref{restrict-lem}; finding the coloring takes time $O(n)$.
\item[10.] \emph{We apply the arguments of
Section~\ref{coloring-simple-subsec}.}
Finding the coloring takes time $O(n)$.
	
\end{enumerate}

Thus, $T_s(n) \leq O(n^{21}) + \max\{T_s(n-1), \max_{3 \leq k
\leq n-1}\{ T_s(n+2-k) + T_s(k)\}\}$, so $T_s(n) \leq O(n^{22})$.

\section*{Acknowledgments}
Thanks to Jiaao Li for numerous helpful suggestions that improved the clarity of
our exposition.  Thanks to Richard Hammack for ideas to improve the figures, and
to Marthe Bonamy for helpful discussions early on in this project.  Finally,
thanks to two anonymous referees.  One provided extensive comments to help
improve our presentation.

\end{document}